\documentclass[a4paper,12pt,headsepline,cleardoubleempty,reqno,bibtotoc,tablecaptionabove]{amsart}

\usepackage[utf8]{inputenc}
\author{David Wallauch}
\address{B\^atiment des Math\'ematiques, EPFL,
Station 8, CH-1015 Lausanne, Switzerland }
\email{david.wallauch@epfl.ch}
\thanks{This work was supported by the Austrian Science Fund FWF,
  Project J4875: ``On optimal blowup stability for supercritical wave equations''.
}
\title{On Strichartz estimates and optimal blowup stability of supercritical wave equations}
\usepackage{amsmath,amsfonts,amssymb,amsthm,empheq}
\usepackage{xcolor}

\usepackage{hyperref}
\usepackage{array}
\usepackage{fullpage}
\usepackage{graphicx}
\usepackage{mathtools}
\DeclarePairedDelimiter\ceil{\lceil}{\rceil}
\DeclarePairedDelimiter\floor{\lfloor}{\rfloor}
\numberwithin{equation}{section}

\newcommand{\C}{\mathbb{C}}

\newcommand{\R}{\mathbb{R}}

\newtheorem{thm}{Theorem}[section]
\newtheorem{defi}{Definition}[section]
\newtheorem{rem}{Remark}[section]

\newtheorem{prop}{Proposition}[section]
\newtheorem{lem}{Lemma}[section]

\newcommand{\N}{\textup{\textbf{N}}}

\newcommand{\Cf}{\textup{\textbf{C}}}

\newcommand{\X}{\mathcal{X}}

\newcommand{\Nf}{\textup{\textbf{N}}}
\newcommand{\Lf}{\textup{\textbf{L}}}

\newcommand{\Af}{\textup{\textbf{A}}}
\newcommand{\Ef}{\textup{\textbf{E}}}
\newcommand{\hf}{\textup{\textbf{h}}}

\newcommand{\K}{\textup{\textbf{K}}}
\newcommand{\Sf}{\textup{\textbf{S}}}
\newcommand{\I}{\textup{\textbf{I}}}
\newcommand{\Uf}{\textup{\textbf{U}}}
\newcommand{\uf}{\textup{\textbf{u}}}
\newcommand{\vf}{\textup{\textbf{v}}}

\newcommand{\gf}{\textup{\textbf{g}}}

\newcommand{\ff}{\textup{\textbf{f}}}
\newcommand{\Pf}{\textup{\textbf{P}}}
\newcommand{\Qf}{\textup{\textbf{Q}}}
\renewcommand{\Re}{\operatorname{Re}}
\renewcommand{\Im}{\operatorname{Im}}
\newcommand{\B}{\mathbb{B}}

\renewcommand{\H}{\mathcal{H}}

\newcommand{\Rm}{\mathcal{R}}
\renewcommand{\O}{\mathcal{O}}
\usepackage[title]{appendix}

\begin{document}

     \begin{abstract}
We establish Strichartz estimates, including estimates involving spatial derivatives, for radial wave equations with potentials in similarity variables. This is accomplished for all spatial dimensions $d\geq 3$ and almost all regularities above energy and below the threshold $\frac d2$. These estimates provide a unified framework that allows one to derive optimal blowup stability result for a wide range of energy supercritical nonlinear wave equations. To showcase their usefulness, an optimal blowup stability result for the quintic nonlinear wave equation is also obtained.
      \end{abstract}
      \maketitle
\section{Introduction}
The present work focuses on energy supercritical wave equations. As a motivating example, we consider the Cauchy problem
\begin{equation}\label{Eq:startingeq}
\begin{cases}\left(\partial_t^2-\Delta_x\right)u(t,x)=u(t,x)^5
\\
u[0]=(f,g)
\end{cases}
\end{equation}
for a fixed dimension $d\geq 3$, where $u:I\times \R^{d}\to \R$ for some interval $0\in I\subset\R$ and $u[t]=(u(t,.),\partial_t u(t,.))$. Notably, solutions to this equation conserve the energy
\begin{align*}
E(u)(t):=\tfrac{1}{2}\|u(t,.)\|_{\dot{H}^1(\R^d)}^2+\tfrac{1}{2}\|\partial_t u(t,.)\|_{L^2(\R^d)}^2- \tfrac 16 \|u(t,.)\|_{L^6(\R^d)}^6
\end{align*}
and, moreover, for any solution $u$ and any $\lambda>0$ the rescaling 
$$
u(t,x) \mapsto u_\lambda(t,x):= \lambda^{-\frac{1}{2}} u(\tfrac{t}{\lambda},\tfrac{x}{\lambda})
$$
produces another solution.
This scaling symmetry provides a lower bound on the regularity of $u[0]$ in order to have a useful local well-posedness theory. In particular, for a pair of $L^2$ based homogeneous Sobolev spaces $\dot H^s\times \dot H^{s-1}(\R^d), $ the scaling invariant space is given by  $s= \frac{d-1}{2}$, hence, for a meaningful well-posedness in a tuple of inhomogeneous Sobolev spaces $ H^s\times H^{s-1}(\R^d), $ one should enforce $s\geq \frac{d-1}{2}$. This regularity level is of course not high enough to generally produce classical solutions and therefore one instead resorts to the Duhamel formulation
\begin{align} \label{notion solution}
u(t,.)=\cos(t|\nabla|)f+\frac{\sin(t|\nabla|)}{|\nabla|}g+\int_0^t\frac{\sin((t-s)|\nabla|)}{|\nabla|}u(s,.)^5 ds
\end{align}
where $\cos(t|\nabla|)$ and $\frac{\sin(t|\nabla|)}{|\nabla|}$ are the standard wave propagators. 
This is a sensible expression for all $(f,g)\in H^s\times H^{s-1}(\R^d)$ and $s\geq \frac{d-1}{2}$. Furthermore, it is known that Eq.~\eqref{Eq:startingeq} is locally well-posed for such regularities (see \cite{LinSog95}). The difficulty to establish local well-posedness of Eq.~\eqref{Eq:startingeq} ''at scaling``, i.e. for $s=\frac{d-1}{2}$ stems from the fact that the nonlinearity cannot solely be controlled by soft arguments such as Sobolev embedding. One gets around this problem by employing Strichartz estimates. These are spacetime estimates of the form
\begin{align*}
\left\|\frac{\sin(t|\nabla|)}{|\nabla|}g \right\|_{L^{p}_t(\R)L^{q}(\R^d)}\lesssim \|g\|_{\dot H^{\frac{d-3}{2}}(\R^d)}
\end{align*}
for pairs $(p,q)$ with $p\in(2,\infty]$ and $q\in [2d,\infty)$ that satisfy the scaling relation $\frac{1}{p}+\frac d q =1$ and variants thereof. The improved integrability provided by spacetime estimates of this type allows one to perform a fixed point argument with the expression \eqref{notion solution}. Loosely, the goal of this paper is to establish related dispersive estimates for wave equations with certain potentials that enable one to prove optimal stability results of self-similar solutions to nonlinear wave equations.
\\
Sticking to the wave equation with a focusing quintic power nonlinearity as our motivating example, we infer that this equation has an explicit blowup solution given by
$$
u^T(t):=\left(\frac{3}{4}\right)^{\frac{1}{4}}(T-t)^{-\frac{1}{2}}
$$
for any $T>0$. This example nicely illustrates that solutions starting from smooth initial data might nevertheless break down in finite time. On top of that, the finite speed of propagation property allows us to smoothly cut off $u^T[0]$ far out, which produces solutions starting from smooth and compactly supported data and which look like $u^T$ inside the ball of some radius around the origin, hence also exhibit finite time blowup. The existence of this explicit blowup solution naturally compels one to study whether it plays a role in generic evolutions of solutions to Eq. \eqref{Eq:startingeq}. Since $u^T[0]$ does not belong to any of the spaces $H^s\times H^{s-1}(\R^d)$, the study of its properties is best done in a local setting. Hence, for a stability analysis of $u^T$ one  is lead to restrict to backwards light cones of the form
$$
\Gamma_T:=\{(t,x)\in [0,\infty)\times \R^d: t\in [0,T), |x| \leq [0, T-t]\}.$$
From now on, we only consider radial initial data for which Eq.~\eqref{Eq:startingeq} reads as
\begin{equation}\label{Eq:startingeq2}
\begin{cases}\left(\partial_t^2-\partial_r^2-\frac{d-1}{r}\partial_r\right) \widetilde u(t,r)= \widetilde u(t,r)^{5}
\\
\widetilde u[0]=(\widetilde f,\widetilde g),
\end{cases}
\end{equation} 
for $r=|x|$ and where $\widetilde u,\widetilde f, $ and $ \widetilde g$ are the radial representative of $u,f$, and $g,$ respectively.
To study the stability properties of $u^T$, one linearises the nonlinearity around this solution, leading to the potential $\frac{15}{4} (T-t)^{-2}$. Hence, in order to obtain a stability result of $u^T$ at the optimal regularity $s=\frac{d-1}{2}$ the need for Strichartz estimates for wave equations with potentials arises naturally.
To digress further into this, we introduce the similarity coordinates
\begin{align*}
\tau:=-\log(T-t)+\log(T),\qquad \rho:=\frac{r}{T-t}
\end{align*}
which map the radial light cone onto the infinite cylinder $(\tau,\rho)\in [0,\infty) \times (\overline{\B^d_1})$.
Setting  $\psi(\tau,\rho)=(Te^{-\tau})^{\frac{1}{2}}u(T-Te^{-\tau}, Te^{-\tau}\rho)$
and
\[\psi_1(\tau,\rho)=\psi(\tau,\rho),\qquad \psi_2(\tau,\rho)=\partial_\tau\psi(\tau,\rho)+\rho\partial_\rho\psi(\tau,\rho)+\frac{d-2}{2}\psi(\tau,\rho) \]
yields an abstract evolution equation of the form
 \[\partial_\tau \Psi(\tau)=\widehat{\Lf}_{5} \Psi(\tau)+ \Nf(\Psi(\tau)),\]
 where $\widehat{\Lf}_{5}$ is a linear spatial differential operator.
 More precisely,
 $\widehat{\Lf}_{5}$ is formally given by
 \begin{align*}
\widehat{\Lf}_{5} \begin{pmatrix}
f_1\\
f_2
\end{pmatrix}(\rho)
=\begin{pmatrix}
-\rho f_1'(\rho)-\frac{1}{2}f_1(\rho)+f_2(\rho)
\\
f_1''(\rho)+\frac{2}{\rho}f_1'(\rho)-\rho f_2'(\rho)-\frac{3}{2}f_2(\rho)
\end{pmatrix}+
\begin{pmatrix}
0\\
\frac{15}{4}f_1(\rho)
\
\end{pmatrix},
\end{align*}
The first summand on the right hand side is just the free radial wave equation in our new coordinates and where the second corresponds to the potential. As we are interested in more general equations than the quintic, we study the more general operator 
 \begin{align}\label{def: operator L}
\widehat{\Lf} \begin{pmatrix}
f_1\\
f_2
\end{pmatrix}(\rho)
=\begin{pmatrix}
-\rho f_1'(\rho)-\frac{d-2s}{2}f_1(\rho)+f_2(\rho)
\\
f_1''(\rho)+\frac{d-1}{\rho}f_1'(\rho)-\rho f_2'(\rho)-\frac{d-2s+2}{2}f_2(\rho)
\end{pmatrix}+
\begin{pmatrix}
0\\
V(\rho)f_1(\rho)
\
\end{pmatrix}
\end{align}
for a fixed spatial dimension $d\geq 3$, a fixed regularity $s\in \mathbb \R $ with $ 1 \leq s < \frac{d}{2}$ and $\lceil s\rceil  \leq \frac{d}{2}$, (here, $\lceil s\rceil$ denotes the smallest natural number $k$ that satisfies $s\leq k$), and a fixed smooth radial potential $V\in C^\infty (\overline{\B^d_1})$.
This more general operator arises by adapting the transformations to the $H^s$ critical radial wave equation in $d$ dimensions.
Before we can state our first theorem, we denote by $L^2_{rad}(\B^d_1)$ the space $\{f\in L^2(\B^d_1):f \text{ radial}\}$. The analogous definitions hold for $C^s_{rad}(\overline{\B^d_1})$ as well as $L^2$ based Sobolev spaces $H^s_{rad}(\B^d_1)$ or $L^p$ based ones, denoted by $W^{s,p}_{rad}(\B^d_1)$. Lastly, we denote by $\|f\|_{\dot{W}^{n,p}(\B^d_1)}$ the $L^p$ based seminorm of $f$, i.~e.
$$
\|f\|_{\dot{W}^{n,p}(\B^d_1)}^p:=\sum_{\alpha_1+\dots\alpha_d=n}\int_{\B^d_1}|\partial_{x_1}^{\alpha_1}\cdots\partial_{x_d}^{\alpha_d} f(x)|^p dx.
$$
Our main theorem now reads as follows.
\begin{thm}\label{thm:strichartz}
Let $3\leq d \in \mathbb N$ and $1 \leq s<\frac d2$ be two fixed numbers with $\ceil s \leq \frac{d}{2}$. Further, let $V \in C^\infty_{rad}(\overline{\B^d_1})$ be fixed and define $D(\widehat{\Lf}):= C^{\infty}_{rad}\times C^{\infty}_{rad}(\overline{\B^d_1})$. Then the operator $\widehat{\Lf} :D(\widehat{\Lf})\to H^{s}_{rad}\times H^{s-1}_{rad}(\B^d_1)$ as defined in \eqref{def: operator L} is closable and its closure $\Lf$ generates a one-parameter semigroup $\{\Sf(\tau):\tau\geq 0\}$ of bounded linear operators on $H^s_{rad}\times H^{s-1}_{rad}(\B^d_1)$ such that the following holds. There exists a finite dimensional subspace $\Uf \subset H^{s}_{rad}\times H^{s-1}_{rad}(\B^d_1) $ and a bounded linear operator $ \Pf:H^{s}_{rad}\times H^{s-1}_{rad}(\B^d_1)\to \Uf $ such that
\begin{align*}
\|\Sf(\tau)(\I-\Pf)\ff\|_{L^\infty_\tau(\R_+)H^s\times H^{s-1}(\B^d_1)}\lesssim \|(\I-\Pf)\ff\|_{H^s\times H^{s-1}(\B^d_1)}.
\end{align*}
In addition, the Strichartz estimates
\begin{align*}
\|[\Sf(\tau)(\I-\Pf)\ff]_1\|_{L^p_\tau(\R_+)\dot{W}^{n,q}(\B^d_1)}\lesssim\|(\I-\Pf)\ff\|_{H^s\times H^{s-1}(\B^d_1)}
\end{align*}
hold for all $n\in \mathbb{N}_0$ with $0\leq n\leq s-1$, $\ff \in H^{s}_{rad}\times H^{s-1}_{rad}(\B^d_1)$, and $p,q\in [2,\infty]$, that satisfy the scaling condition
$$\frac{1}{p}+\frac{d}{q}=\frac{d}{2}-s+n.$$
Furthermore, in case $s\notin \mathbb{N}$ the estimates
\begin{align*}
\|[\Sf(\tau)(\I-\Pf)\ff]_1\|_{L^p_\tau(\R_+)\dot{W}^{\floor s ,q}(\B^d_1)}\lesssim\|(\I-\Pf)\ff\|_{H^s\times H^{s-1}(\B^d_1)}
\end{align*}
hold for all $p\in [\frac{2}{s-\floor s},\infty]$ and $q\in [2,\infty]$ that satisfy the scaling condition
$$\frac{1}{p}+\frac{d}{q}=\frac{d}{2}-s+\floor s.$$

Moreover, the inhomogeneous estimate 
\begin{align*}
\left\|\int_0^\tau[\Sf(\tau-\sigma)(\I-\Pf)\hf(\sigma,.)]_1 d \sigma\right\|_{L^p_\tau((0,\tau_0))\dot{W}^{n,q}(\B^d_1)}\lesssim \|(\I-\Pf)\hf(\tau,.)\|_{L^1_\tau((0,\tau_0))H^s\times H^{s-1}(\B^d_1)}
\end{align*}
holds for the same numbers $n,p,q$, all $\hf \in C([0,\infty),H^s_{rad}\times H^{s-1}_{rad}(\B^d_1))\cap L^1([0,\infty),H^s_{rad}\times H^{s-1}_{rad}(\B^d_1))$ and all $\tau_0>0$. 
\end{thm}

Before we come to our second theorem, we would like to make the following remarks.
\begin{itemize}
\item
The finite dimensional subspace $\Uf$ corresponds to the unstable eigenspace introduced through the potential $V$. In particular, if all such unstable eigenvalues have positive real part, $\Uf$ is simply the union of all the associated generalized eigenspaces and $\Pf$ is the associated spectral projection. However, it could also happen that $V$ produces eigenvalues which lie on the boundary of the essential spectrum of the unperturbed operator $\Lf_0$. In this case, $\Pf$ can be thought of as the bounded part of the spectral projection of all newly appearing eigenvalues and $\Uf$ as its image. This will be made more precise later on.\smallskip
\item 
Prominent examples of potentials $V$ arise for instance from linearising around blowup solutions such as the wave maps blowup
\begin{align*}
\psi^T(t, x)=2\arctan\left(\frac{|x|}{\sqrt{d}(T-t)}\right)
\end{align*}
or the 
ODE blowup
\[
u^T(t):=c_{d,s}(T-t)^{\frac{2s-d}{2}}
\quad c_{d,s}=\left(\frac{(d-2s)(d-2s+2)}{4}\right)^{\frac{d-2s}{4}}
\]
and transforming the resulting potential to similarity coordinates.\smallskip
\item
Note also, that admissible potentials are allowed to exhibit singular behaviour at the tip of the light cone in standard Cartesian coordinates. For instance, the   potential corresponding to the ODE blowup $u^T$ is constant in similarity coordinates but reads as
\begin{align*}
c (T-t)^{-2}
\end{align*}
for some constant $c\in \R $ in Cartesian coordinates.
\smallskip
\item 
We would like to mention, that aside from the restriction for derivatives close to the prescribed regularity, our result covers the same range of homogeneous Strichartz estimates as the free radial wave equation. In particular, we also obtain $L^p L^\infty$ type endpoint estimates in case $s=\frac{d}{2}-\frac{1}{p}$. Moreover, a simple interpolation argument also yields estimates of the form 
\begin{align*}
\|[\Sf(\tau)(\I-\Pf)\ff]_1\|_{L^p_\tau(\R_+)W^{r,q}(\B^d_1)}\lesssim\|(\I-\Pf)\ff\|_{H^s\times H^{s-1}(\B^d_1)}
\end{align*}
with $r \notin \mathbb{N}$.
\smallskip
\item
Finally, we comment on the restriction $\ceil s \leq \frac d2$. This restriction only excludes regularities $s>\frac{d-1}{2}$ in case $d$ is odd. It is in place, as we arrive at the claimed estimates by deriving estimates at the $\ceil s$ and $\floor s$ level and interpolating between these. However, for this strategy to work in the excluded cases, we would need to derive estimates at the regularity $\frac{d+1}{2}$, which appears to not be possible.\smallskip
\item 
It is also instructive to record some explicit admissible values of $p,q,n$ for later usage. 
If $1\leq s-n\leq \frac{d-1}{2}$, then the $(p,q)=(\infty, \frac{2d}{d-2s+2n})$ and $(p,q)=(2,\frac{2d}{d-2s-1+2n})$ are admissible. Additionally, the pair $(\infty, \frac{2d}{d-2s+2n})$ is always admissible.  
\end{itemize}
To also provide a nice straightforward application of estimates derived in Theorem \eqref{thm:strichartz}, we revisit the quintic wave equation 
\begin{equation}\label{Eq:startingeq3}
\begin{cases}
\left(\partial_t^2-\partial_r^2-\frac{d-1}{r}\partial_r\right)u(t,r)=u(t,r)^{5}
\\
u[0]=(f,g)\in H^{\frac{d-1}{2}}_{rad}\times H^{\frac{d-3}{2}}_{rad}(\B^d_{1+\delta})
\end{cases}
\end{equation}
and prove the following result
\begin{thm}\label{thm: stability}
There exist constants $M>1$ and $\delta_0>0$ such that for $\delta \in (0,\delta_0)$ the following holds. Let $(f,g)\in H^{\frac{d-1}{2}}_{rad}\times H^{\frac{d-3}{2}}_{rad}(\B^d_{1+\delta})$ be such that
\begin{align*}
\|(f,g)-u^1[0]\|_{H^{\frac{d-1}{2}}\times H^{\frac{d-3}{2}}(\B^d_{1+\delta})}\leq \frac{\delta}{M}.
\end{align*} 
Then, there exists a $T$ in $[1-\delta, 1+\delta]$ and a unique solution $$u:\Gamma^T:=\{(t,x)\in [0,\infty)\times \R^d: |x|\leq T-t\}\to \C $$ to Eq.~\eqref{Eq:startingeq2} such that
\begin{align} \label{esti:bu1}
 \int_0^T\|u(t,.)-u^T(t,.)\|_{L^{\infty}(\B^d_{T-t})}^{2} dt\leq \delta^2
\end{align}
and
\begin{align} \label{esti:bu2}
\int_0^T\|u(t,.)\|_{W^{n,\frac{d}{n}}(\B^d_{T-t})}^2 dt\leq \delta^2
\end{align}
 for all $1\leq n\leq \floor s-1$.
\end{thm}
Once more, we would like to make some remarks.
\begin{itemize}
\item
A simple computation shows that
$$
 \|u^T(t)\|_{L^{\infty}(\B^d_{T-t})}^2\simeq (T-t)^{-1}
 $$
and one sees that in order for the estimates \eqref{esti:bu1} to hold, $u$ is forced to exhibit the same blowup behavior as $u^T$ in a Strichartz space sense.  Consequently, Theorem \ref{thm: stability} states that there is an open ball around $u^1[0]$ in the optimal topology such that data inside that ball leads to the ODE type blowup. Observe, however, that the actual blowup time gets slightly shifted in general. This shift is a consequence of the time translation symmetry of Eq. \eqref{Eq:startingeq2}. Additionally, \eqref{esti:bu2} ensures certain control and smallness of the derivatives of $u$.\smallskip
\item
We also shortly digress into our notation of a solution. Loosely, a solution, when transformed to similarity variables, is a fixed point of the Duhamel formula
\begin{align*}
[\Sf(\tau)\ff]_1+\int_0^\tau [\Sf(\tau-\sigma)\Nf((\phi(\sigma),0))]_1 d\sigma,
\end{align*}
which lies in an appropriate Strichartz space. For the precise definition we refer to Definition \ref{def:solutionstrichartz}).
It is of interest to note that said fixed point is an element of $H^s(\B^d_1)$ for all $\tau$ for which it exists. Hence, our concept of a solution is naturally compatible with other rough notions of solutions that can be found in the literature.
Further, we want to emphasize that, should the prescribed data be smooth, the associated solution will be as well. This follows by standard Gronwall arguments. \smallskip
\item
The topology used in which we construct solutions is optimal in terms of $L^2$ based inhomogeneous Sobolev spaces, in that the number of derivatives required cannot be lowered.
\item
Naturally, our proof can be easily modified to other nonlinearities, provided one can obtain the required spectral information. We just used the quintic nonlinear wave equation as a nice model to illustrate the useful of our estimates. \smallskip
\item
As a final remark, we want to slightly delve into the literature on Strichartz estimates in similarity variables and optimal blowup stability results for nonlinear wave equations. The first results in that direction were proven by Donninger, who established such Strichartz estimates and the optimal stability of the ODE blowup at energy \cite{Don17} in three spatial dimensions. The methods used in this work were then extended to the five dimensional problem \cite{DonRao20} and later also to small even dimensions \cite{Wal22}. Since these three works are confined to the energy topology, the next key step in the development of this theory was to advance it to higher derivatives. The first step in that direction was the derivation of Strichartz estimates in similarity variables on the $H^2$ level and employing these to prove the optimal blowup stability of the wave maps blowup in $4$ dimensions \cite{DonWal22a}. Later on, this was also accomplished in $3$ dimensions \cite{DonWal23}, which required an extension of the theory to Sobolev spaces of half integer order. This current work can therefore naturally be viewed as the next step, as it extends the existing framework developed in all these works to derive estimates for a very large range of (in general non-integer!) regularities.
\end{itemize}
\subsection{Outline of the proof of Theorem \ref{thm:strichartz}}
Given that the majority of this work is concerned with the derivation of Theorem \ref{thm:strichartz}, we provide a short nontechnical outline of its proof. As mentioned, this work builds on the framework developed in the earlier work 
 \cite{Don17,DonRao20,DonWal22a,Wal22} and starts with considering the radial linear wave equation with potential as an abstract evolution equation in the similarity variables
$\tau=-\log(T-t)$ and $\rho=\frac{r}{T-t}$.
Then, by Theorem 2.1 in \cite{Ost24} and the Bounded Perturbation Theorem, the associated operator $\Lf$ generates a semigroup $\Sf$ which satisfies
\begin{align*}
\|\Sf(\tau) (\I-\Pf)\|_{\mathcal{H}^s}\leq C_\varepsilon e^{\varepsilon\tau}
\end{align*}
for any $\varepsilon>0$
where $\H^s:=H^s_{rad}\times H^{s-1}_{rad}(\B^d_1)$ and where $\Pf$ is a bounded linear operator with finite rank. Furthermore, in the non-integer case, $\Sf$  decays exponentially on $\H^{\ceil s}$, up to finite many directions.
To move on, we asymptotically construct the resolvent of $\Lf$. This is accomplished by making use of the fact that the resolvent equation reduces to a second order ODE which is then analysed with by means of a Liouville Green Transformation and Bessel asymptotics. The resulting solutions, which, due to singular behaviors, need to be constructed separately near the poles $\rho=0$ and $\rho =1$, are then glued together and used to construct $(\lambda-\Lf)^{-1}$. 
While this follows the above mentioned framework, it is technically far more involved than the corresponding resolvent construction in previous works, as we allow $\Lf$ to live on $H^{\lceil s\rceil} \times H^{\lceil s\rceil-1}(\B^d_1)$ for high $s$ and dimension $d$. Hence, it needs also need fundamentally new ideas.
With this at hand, we employ the Laplace
representation
\begin{equation}\label{eq:laprepi} \Sf(\tau)(\I-\Pf)\ff=\frac{1}{2\pi
    i}\lim_{N\to\infty}\int_{\epsilon-iN}^{\epsilon+iN}e^{\lambda\tau}(\lambda-\Lf)^{-1}(\I-\Pf)\ff\,d\lambda
\end{equation}
and show that our delicate construction allows us to take large enough number of $\rho$ derivative inside the integral. 
Then, a careful study of the resulting oscillatory integrals leads to desired estimates. For this, it is necessary to view $\Lf$ and $\Sf$ in the $ H^{\lceil s\rceil} \times H^{\lceil s\rceil-1}(\B^d_1) $ topology.
Then, as Strichartz estimates for $\Sf_{0}$ (the semigroup corresponding to the free radial wave equation without potential) follow from the standard ones in Cartesian coordinates and a simple scaling argument, our aim is to proof weighted Strichartz estimates, which are essentially of the form 
\begin{align*}
\|[e^{-(\lceil s\rceil -s) \tau} (\Sf(\tau)-\Sf_{0}(\tau))(\I-\Pf) \ff]_1 \|_{L^p(\R_+) W^{n,q} (\B^d_1)}&\lesssim \|\ff\|_{H^{\lceil s\rceil} \times H^{\lceil s\rceil-1}(\B^d_1) }
\end{align*}
and 
\begin{align*}
\|[e^{(s-\floor s) \tau} (\Sf(\tau)-\Sf_{0}(\tau))(\I-\Pf) \ff]_1 \|_{L^p(\R_+) W^{n,q} (\B^d_1)}&\lesssim \|\ff\|_{H^{\lceil s\rceil} \times H^{\floor s-1}(\B^d_1) }
\end{align*}
for $\ff$ smooth (to be be completely precise, the estimates are slightly more involved, but are in spirit of this form). These bounds are obtained by careful study of the oscillatory integral on the right side of \eqref{eq:laprepi}. The desired estimates then follow from interpolating the above estimates.

\subsection{Related results}
Even though the most closely related results have already been mentioned in the introduction, we feel compelled to also  highlight several other recent works. For works on Strichartz estimates for wave equations with potentials, we refer the reader to \cite{DAnFan08,MetTar07,AiIfrTat24,DonGlo19} and the reference within these works. The study of blowup stability for energy supercritical wave equations was started with the works \cite{DonSch14,Don11, Don14} while extensive work on the subcritical problem was done by Merle and Zaag in the series of works \cite{MerZaa03,MerZaa05,MerZaa07,MerZaa08,MerZaa12a,MerZaa12b,MerZaa15} and further studied by Alexakis and Shao \cite{AleSha17} and Azaiez \cite{Aza15} and also Donninger and Schörkhuber \cite{DonSch12}. An extension of the work \cite{DonSch14} in which blowup stability for the energy supercritical radial wave equation was shown, were established in \cite{DonSch17} which generalises their result to all odd dimensions and \cite{DonSch16} which extends their work in three dimensions to the nonradial setting. More recently, Glogić and Schörkhuber investigated the blowup stability of a different explicit blowup for the supercritical cubic, and together with Csobo, quadratic wave equation \cite{GloScho21,CsoGloSch21}. We also want to bring up an extension of the three dimensional optimal blowup stability result of Donninger to randomized initial data \cite{Bri20}. Lastly and only loosely connected, but still definitely worth mentioning, Strichartz estimates and optimal blowup stability have been established for the slightly mass supercritical nonlinear Schrödinger equation \cite{Li23}.

\section{Similarity coordinates and Semigroup Theory}
To introduce the right functional set-up, we let $d\geq3 $ be fixed and consider the free
radial wave equation 
\begin{equation}\label{Eq:freeradial}
\left(\partial_t^2-\partial_r^2-\frac{d-1}{r}\partial_r\right)u(t,r)=0
\end{equation}
with $t,r \in \Gamma^T$ and where $d\geq 3$ is some fixed natural number. Then, for $s\in \R$ with  $2 \leq s < \frac{d}{2}$ fixed, we define the similarity coordinates $\tau, \rho$ as
\begin{align*}
\tau=-\log(T-t)+\log(T),\qquad \rho=\frac{r}{T-t}
\end{align*}
and set $\psi(\tau,\rho)=(Te^{-\tau})^{\frac{d-2s}{2}}u(T-Te^{-\tau},Te^{-\tau}\rho)$. This transforms  Eq.~\eqref{Eq:freeradial} into \begin{equation}\label{Eq: sim coordinates}
\begin{split}
\bigg[\partial_\tau^2 &+(d-2s+1)\partial_\tau+2\rho\partial_\tau\partial_\rho -(1-\rho^2)\partial_\rho^2 -\frac{d-1}{\rho}\partial_\rho+(d-2s+2)\rho\partial_\rho\bigg]\psi(\tau,\rho)
\\
&+\frac{(d-2s)(d-2s+2)}{4}\psi(\tau,\rho)
=0.
\end{split}
\end{equation}
Let now
\begin{align*}
\psi_1(\tau,\rho)&=\psi(\tau,\rho)
\\
\psi_2(\tau,\rho)&=\partial_\tau\psi(\tau,\rho)+\rho\partial_\rho\psi(\tau,\rho)+\frac{d-2s}{2}\psi(\tau,\rho).
\end{align*}
Then Eq. \eqref{Eq: sim coordinates} turns into the system
\begin{equation} \label{Eq:intermediate system2}
\begin{split}
\partial_\tau \psi_1(\tau,\rho)&=-\rho\partial_\rho\psi_1(\tau,\rho)-\frac{d-2s}{2}\psi_1(\tau,\rho)+\psi_2(\tau,\rho)
\\
\partial_\tau \psi_2(\tau,\rho)&=\partial_\rho^2\psi_1(\tau,\rho)+\frac{d-1}{\rho}\partial_\rho \psi_1(\tau,\rho)-\rho\partial_\rho \psi_2(\tau,\rho)-\frac{d-2s+2}{2}\psi_2(\tau,\rho),
\end{split}
\end{equation}
Next, we define $\H^r:=\{\uf \in H^r \times H^{r-1}(\B^d_1):\uf \text{ radial}\} $ for any $r\geq 1$ and denote by  $\|.\|_{\mathcal{H}^r}$ the radial $H^r\times H^{r-1}(\B^d_1)$ norm. 
Motivated by the above system, we define the operator \\$\widetilde{\Lf}:D(\widetilde{\Lf})\subset  \H^{\floor s} \to \H^{\floor s}$, as
\begin{align*}
\widetilde{\Lf} \ff(\rho)=\begin{pmatrix}
-\rho f_1'(\rho)-\frac{d-2s}{2}f_1(\rho)+f_2(\rho)\\
f_1''(\rho)+\frac{d-1}{\rho}f_1'(\rho)-\rho f_2'(\rho)-\frac{d-2s+2}{2}f_2(\rho)
\end{pmatrix},
\end{align*}
where $D(\widetilde{\Lf}):=\{\ff\in C^\infty\times C^\infty(\overline{\B_1^d}):\ff \text{ radial} \}$.
To proceed, we note that if we set $\theta$ to equal $s-\floor s$. Then, $\theta $ is chosen such that $\theta \ceil s +(1-\theta) \floor s=s$.
Therefore, by Theorem 2.1 in \cite{Ost24} (and an elementary interpolation argument) we know that $\widetilde \Lf$ is closable and its closure, denoted by $\Lf_0$, generates a strongly continuous semigroup $\Sf_0$ of bounded linear operators on $\mathcal{H}^{ \floor s}$.
Moreover, we know the following bounds:
\begin{enumerate}
\item
\begin{equation}\label{eq:semibound1}
\ceil s < \frac d2 \implies \begin{cases} 
\|\Sf_0(\tau) \ff\|_{\mathcal{H}^{ \ceil s}} &\lesssim  e^{ (-\ceil s +s)\tau}\|\ff\|_{\mathcal{H}^{\ceil s}}= e^{-(1-\theta)\tau}\|\ff\|_{\mathcal{H}^{\ceil s}}
\\
\|\Sf_0(\tau) \ff\|_{\mathcal{H}^{ \floor s}} &\lesssim  e^{ (s-\floor s)\tau}\|\ff\|_{\mathcal{H}^{\floor s}}= e^{\theta \tau}\|\ff\|_{\mathcal{H}^{\floor s}}
\\
\|\Sf_0(\tau) \ff\|_{\mathcal{H}^{ s}} &\lesssim \|\ff\|_{\mathcal{H}^{s}}
\end{cases}
\end{equation}
for all $\tau \geq 0$ and all $\ff\in \mathcal{H}^{\ceil s}$.
\item
\begin{equation}\label{semibound2}
\ceil s = \frac d2 \implies
\begin{cases} 
\|\Sf_0(\tau) \ff\|_{\mathcal{H}^{ \ceil s}} &\lesssim_{\varepsilon} e^{(s-\ceil s +\varepsilon)\tau}\|\ff\|_{\mathcal{H}^{\ceil s}}
\\
\|\Sf_0(\tau) \ff\|_{\mathcal{H}^{ \floor s}} &\lesssim_{\varepsilon}  e^{ (s-\floor s)\tau}\|\ff\|_{\mathcal{H}^{\floor s}}= e^{\theta \tau}\|\ff\|_{\mathcal{H}^{\floor s}}
\\
\|\Sf_0(\tau) \ff\|_{\mathcal{H}^{ s}} &\lesssim_{\varepsilon} e^{\varepsilon \tau} \|\ff\|_{\mathcal{H}^{s}}
\end{cases}
\end{equation}
for any $\varepsilon>0 $ fixed and
all $\tau \geq 0$ as well as all $\ff\in \mathcal{H}^{\ceil s}$.
\end{enumerate}
From now on, whenever we reference any of the operators $\Lf_0$ or $\Sf_0$ (or any related / derived operator), we always consider the one living in $\mathcal{H}^{\ceil s}$  in case $s$ is not an integer and $\mathcal{H}^{s+\frac{1}{100}}$ in case it is, unless specified otherwise. Using the space $\mathcal{H}^{s+\frac{1}{100}}$ might, understandably, seem odd. The point here is that it in this space, the essential spectrum of $\Lf$ is to the left of the imaginary axis and does not include it. This lies in contrast to the space $\mathcal H^s$.  
Moreover, going forward, we confine $T$ to the interval $\left[\frac{1}{2},\frac{3}{2}\right]$. This restriction of $T$ leads to no loss of generality as we only care for $T$ close to $1$ anyway. 
\begin{lem}\label{Strichart}
Let $d\geq 3$ and $1\leq s<\frac{d}{2}$ with $1\leq \ceil s\leq \frac{d}{2}$ be fixed. Then, the bounds
\begin{align*}
\|[\Sf_0(\tau)\ff]_1\|_{L^p_\tau(\R^+)\dot{W}^{n,q}(\B^d_1)}\lesssim \|(\I-\Pf)\ff\|_{\H^s}
\end{align*}
hold for all $n\in 
\mathbb{N}_0$ with  $0\leq n\leq s$, $\ff \in \H^s$, and $p,q\in [2,\infty]$  that satisfy both the scaling condition
$$\frac{1}{p}+\frac{d}{q}=\frac{d}{2}-s+n$$
as well as the wave admissibility condition
\begin{align*}
\frac{1}{p}+\frac{d-1}{2q}&\leq \frac{d-1}{4}.
\end{align*}
Additionally, the inhomogeneous estimate 
\begin{align*}
\left\|\int_0^\tau[\Sf_0(\tau-\sigma)\hf(\sigma,.)]_1 d \sigma\right\|_{L^p_\tau(I)\dot{W}^{n,q}(\B^d_1)}\lesssim \|(\I-\Pf)\hf(\tau,.)\|_{L^1_\tau(I)\H}
\end{align*}
holds for all such $n,p,q,$ and $\hf \in C([0,\infty),\H^s)\cap L^1([0,\infty),\H^s)$ as well as all $\tau_0>0$.
\end{lem}

\begin{proof}
Let $T\in \left[\frac{1}{2},\frac{3}{2}\right]$, $\ff \in C^\infty \times C^\infty (\overline{\B^d_1})$, and 
let $\Ef_T:H^s\times H^{s-1}(\B^d_T) \to H^s\times H^{s-1}(\R^d)$ be a family of Sobolev extensions which are uniformly bounded in $T$.
We define the scaling operator $\Af_T:H_{rad}^s\times H_{rad}^{s-1}(\B^d_T)\to \mathcal{H}$ by
 \begin{align*}
\Af_T \ff =( Tf_1(T.),T^2 f_2(T.)).
\end{align*}
In view of the coordinate transformations performed at the beginning of Section 2, the evolution $\Sf_0(.)\ff$
is given by the solution $u\in C^\infty(\R_+\times \R^d)$ of the equation
\begin{align*}
\begin{cases}
\left(\partial_t^2-\partial_r^2-\frac{d-1}{r}\partial_r
\right)u(t,r)=0\\
(u(0,.),\partial_0 u(0,.))=\Ef_T\Af_T^{-1}\ff
\end{cases}
\end{align*}
restricted to the light cone $\Gamma^T$.
Therefore, $$
\left[\Sf_0(\tau)\ff\right]_1(\rho)=(Te^{-\tau})^{\frac{d-2s}{2}}u(T-Te^{-\tau},Te^{-\tau}\rho).
$$
Let now $r\in \mathbb{N}$ be fixed and let $p,q$ be the upper endpoint pair (i.e. the unique tuple that satisfies both the scaling and the wave admissibility condition for said fixed $r$, which is chosen such that all other such admissible tuples ($\widetilde{p},\widetilde{q})$ satisfy $\widetilde{p}>p, \widetilde{q}<q$).
Then, $q$ satisfies
\begin{align*}
q= \frac{2pd}{pd-2sp+2rs-2}
\end{align*}
and we compute
\begin{align*}
\|[\Sf_0(\tau)\ff]_1\|_{L^p_\tau(\R_+)\dot{W}^{r,q}(\B^d_1)}&\leq\left\|(Te^{-\tau})^{\frac{d-2s}{2}} u(T-Te^{-\tau},Te^{-\tau}\rho)\right\|_{L^p_\tau(\R_+)\dot{W}^{r,\frac{2dp}{pd-2-2ps+2rp}}_\rho(\B^d_1)}
\\
&\lesssim \|e^{-\frac{\tau}{p}}u(T-Te^{-\tau},.)\|_{L^p_\tau(\R_+)\dot{W}^{r,\frac{2dp}{dp-2-2sp+2rp}}(\R^d_1)}
\\
&\lesssim \|u\|_{L^2((0,2))\dot{W}^{r,\frac{2dp}{dp-2-2sp+2rp}}(\R^d_1)} \lesssim \|u[0]\|_{H^s\times H^{s-1}(\R^d_1)}
\\
&= \|\Ef_T\Af_T^{-1}\ff\|_{H^s\times H^{s-1}(\R^d_1)} \lesssim \|\ff\|_{H^s  \times H^{s-1}(\B^d_1)}
\end{align*}
due to the classical Strichartz estimates in Cartesian coordinates.
The other endpoint estimate as well as the higher energy estimate 
\begin{align*}
\|[\Sf_0(\tau)\ff]_1\|_{L^\infty_\tau(\R^+)\dot H^s(\B^d_1)}\lesssim \|(\I-\Pf)\ff\|_{\H^s}
\end{align*}
follow likewise after which the general ones are a consequence of interpolation. The inhomogeneous ones are then derived by employing Minkowski's inequality as in the proof of Lemma 3.7 in \cite{DonWal22a}.
\end{proof}
With this at hand, we turn to the full perturbed equation. For this, we let $V$ be a radial smooth potential and
define $\Lf':C^\infty_{rad}\times C^\infty_{rad} (\overline{\B^d_1})$ to $ C^\infty_{rad}\times C^\infty_{rad} (\overline{\B^d_1})$ as
\begin{align*}
\Lf' \ff:=\begin{pmatrix}
0
\\
V f_1
\end{pmatrix}
\end{align*}
Moreover, we set $\Lf:=\overline{\Lf_0+\Lf'}$.
\begin{lem}
For every $\varepsilon>0$ there exist finitely many $\lambda_1,\dots, \lambda_n \in \C$ with $\Re \lambda_i>-s+\floor s+\delta$ such that $$\sigma(\Lf)\subset\{z \in \C:\Re z\leq s-\ceil s +\varepsilon\} \cup \{\lambda_1,\dots \lambda_n\}$$ where each of the $\lambda_i$ is an eigenvalue of finite algebraic multiplicity. 
Likewise, in case $s\in \mathbb{N}$, one has that 
$$\sigma(\Lf)\subset\{z \in \C:\Re z\leq s-\frac{1}{400}\} \cup \{\lambda_1,\dots \lambda_n\}$$
where each of the $\lambda_i$ is again an eigenvalue of finite algebraic multiplicity. 
\end{lem}
\begin{proof}
The inclusion
$$\sigma(\Lf)\subset\{z \in \C:\Re z\leq s-\ceil s)+\varepsilon\} \cup \{\lambda_1,\dots \lambda_n\}$$ 
follows immediately from Theorem B.1 in \cite{Glo22b} since $\Sf_0$ satisfies the growth bound
\begin{align*}
\|\Sf(\tau)\ff\|_{H^{\ceil s}\times H^{\ceil s-1}(\B^d_1)}&\lesssim_\varepsilon e^{(s-\ceil s+\varepsilon)\tau} \|\ff\|_{H^{\ceil s}\times H^{\ceil s-1}(\B^d_1)}
\end{align*}
for all $\varepsilon >0$, all $\tau \geq 0$, and all $\ff \in H^{\ceil s}\times H^{\ceil s-1}(\B^d_1)$, and the fact that $\Lf'$ is a compact operator.
To prove the claim in the integer case, one notes that by once more employing Theorem 2.1 in \cite{Ost24}, one obtains the bounds
\begin{align*}
\|\Sf_0(\tau) \ff\|_{\mathcal{H}^{s+1}}&\lesssim e^{-\frac{1}{3}\tau}\|\ff\|_{\mathcal{H}^{s+1}}
\\
\|\Sf_0(\tau) \ff\|_{\mathcal{H}^{s}} &\lesssim\|\ff\|_{\mathcal{H}^{s}}.
\end{align*}
These imply
\begin{align*}
\|\Sf_0(\tau) \ff\|_{\mathcal{H}^{s+\frac{1}{100}}} &\lesssim e^{-\frac{1}{300} \tau}\|\ff\|_{\mathcal{H}^{s+\frac{1}{100}}}
\end{align*}
and the claim follows.
\end{proof}
To continue, we remark that the equation
$(\lambda-\Lf)\ff=\gf$ reads as
\begin{align*}
\lambda f_1(\rho)+\rho f_1'(\rho)+\frac{d-2s}{2}f_1(\rho)-f_2(\rho)&=g_1(\rho)\\
\lambda f_2(\rho)-f_1''(\rho)-\frac{d-1}{\rho}f_1'(\rho)+\rho f_2'(\rho)+\frac{d-2s+2}{2}f_2(\rho)&=g_2(\rho)
\end{align*}
and the first of the above equations implies

\begin{align*}
f_2(\rho)=\rho f_1'(\rho)+\frac{d-2s+2\lambda}{2}f_1(\rho)-g_1(\rho).
\end{align*}
Plugging this into the second results in the ODE
\begin{align}\label{eq:genspec0}
(\rho^2-1)f_1''(\rho)&+\left((d-2s+2\lambda+2)\rho-\frac{d-1}{\rho}\right)f_1'(\rho)
\\
&
+\frac{d-2s+2\lambda}{4}(d-2s+2\lambda+2)f_1(\rho)+V(\rho)f_1(\rho)=G_\lambda(\rho)
\end{align} 
with
$G_\lambda(\rho)= (\lambda +\frac{d}{2}-s+1)g_1(\rho)+\rho  g_1'(\rho)+
g_2(\rho)
$.
With this at hand, we come to the last result of this section.
\begin{lem}\label{reg:eigenfunctions}
Let $\lambda \in \sigma_p(\Lf)$ with $\Re \lambda>s-\ceil s$ and let $U_\lambda$ be the associated finite dimensional generalised eigenspace. Then $\ff\in U_\lambda\implies \ff\in C^k\times C^{k-1}([0,1]).$ Furthermore, the first component of $\ff$, denoted by $f_1$ satisfies 
$$
f_1 \in W^{n,p} (\B^d_1)
$$
for all $1\leq n\leq k-1$ and $p=\frac{2d}{d+2n-2s-1}$. 
\end{lem}
\begin{proof}
Let $\lambda$ be as in the statement of the Lemma and let $\ff$ be an associated eigenfunction. Recall, that this implies that $f_1$ is a $H^k(\B^d_1)$ solution to the equation \eqref{eq:genspec0} with $G_\lambda=0.$
As the Frobenius indices of this equation are given by $(0,2-d)$ at $\rho=0$ and $(0,s-\frac{1}{2}-\lambda)$ at $\rho=1$, one sees that $f_1\in C^k([0,1])$. 
Next, we denote by $\widetilde f_1$ the unique solution to this equation which satisfies
$$
W(f_1,\widetilde f_1)=\rho^{1-d}(1-\rho)^{s-\frac{3}{2}-\lambda}.
$$
Let now $\gf \in \mathcal{H}$ be a first generalised eigenfunction, i.~e.~$(\lambda-\Lf)\gf=\ff.$
Then, by the variation of constants formula, its first component is necessarily of the form
\begin{align*}
g_1(\rho)=c_1 f_1(\rho)+\widetilde c_2\widetilde f_1(\rho)+ \widetilde f_1(\rho)\int_0^\rho\frac{t^{d-1}{F(t)}f_1(t) }{(1-t)^{s-\lambda-\frac{1}{2}}} dt +f_1(\rho)\int_\rho^{1}\frac{t^{d-1}F(t)\widetilde f_1(t) }{(1-t)^{s-\lambda-\frac{1}{2}}} dt
\end{align*}
where $F(\rho)= (\lambda +\frac{d}{2}-s+1)f_1(\rho)+\rho  f_1'(\rho)+
f_2(\rho)
$ and $c_1,\widetilde c_1\in \mathbb C$.
By scaling, one readily checks that
\begin{align*}
f_1(\rho)\int_\rho^{1}\frac{t^{d-1}F(t)\widetilde f_1(t) }{(1-t)^{s-\lambda-\frac{1}{2}}} dt \in C^k([0,1])
\end{align*}
as well as 
\begin{align*}
f_2(\rho)\int_0^\rho\frac{t^{d-1}F(t)\widetilde f_1(t) }{(1-t)^{s-\lambda-\frac{1}{2}}} dt\in C^k([0,\frac12]).
\end{align*}
Consequently, $\widetilde c_1$ needs to vanish.
For $\rho>\frac{1}{2}$, we compute that
\begin{align*}
\partial_\rho\left[ \widetilde f_1(\rho)\int_0^\rho\frac{t^{d-1}F(t) f_1(t) }{(1-t)^{s-\lambda-\frac{1}{2}}} dt\right]
&=\widetilde f_1'(\rho)\int_0^\rho\frac{t^{d-1}F(t) f_1(t) }{(1-t)^{s-\lambda-\frac{1}{2}}} dt+ \widetilde f_1(\rho)\frac{\rho^{d-1}F(\rho) f_1(\rho) }{(1-\rho)^{s-\lambda-\frac{1}{2}}}
\\
&= -\frac{\widetilde f_1'(\rho)}{(1-\rho)^{s-\lambda-\frac{3}{2}}}\frac{\rho^{d-1}F(\rho) f_1(\rho) }{s-\lambda-\frac12}
\\
&+ \frac{\widetilde f_1'(\rho)}{s-\lambda-\frac12}\int_0^\rho\frac{\partial_t[t^{d-1}F(t)f_1(t)] }{(1-t)^{s-\lambda-\frac{3}{2}}} dt
+
\widetilde f_1(\rho)\frac{\rho^{d-1}F(\rho) f_1(\rho) }{(1-\rho)^{s-\lambda-\frac{1}{2}}}.
\end{align*}
By iterating this scheme, one concludes that 
\begin{align*}
\partial_\rho^k \left[ \widetilde f_1(\rho)\int_0^\rho\frac{t^{d-1}F(t) f_1(t) }{(1-t)^{s-\lambda-\frac{1}{2}}} dt\right]= h(\rho) + c_\lambda f_2^{(k)}(\rho)\int_0^\rho \frac{\partial_t^{k-1}[t^{d-1}F(t)f_1(t)]}{(1-t)^{s-\ceil s +\frac{1}{2}-\lambda}} dt
\end{align*} 
for some $0\neq c_\lambda\in \C$ and $h\in C([0,\frac{1}{2}])$.
Consequently, one observes that 
$$g_1\in H^k(\B^d_1) \iff \int_0^1 \frac{\partial_t^{k-1}[t^{d-1}F(t)f_1(t)]}{(1-t)^{s-\ceil s +\frac{1}{2}-\lambda}} dt=0.$$
Thus, we obtain 
\begin{align*}
f_2^{(k)}(\rho)\int_0^\rho \frac{\partial_t^{k-1}[t^{d-1}F(t)f_1(t)]}{(1-t)^{s-\ceil s +\frac{1}{2}-\lambda}} dt=-f_2^{(k)}(\rho)\int_\rho^1 \frac{\partial_t^{k-1}[t^{d-1}F(t)f_1(t)]}{(1-t)^{s-\ceil s +\frac{1}{2}-\lambda}} dt\in C([\frac{1}{2},1]).
\end{align*}
Therefore, $g_1\in C^k([0,1])$ which implies $\gf \in C^k\times C^{(k-1)}([0,1]).$ 
By iterating this procedure, the first claim follows. To finish the Lemma, it suffices to show that
\begin{align*}
|.|^{-n+1}\in L^{p}(\B^d_1) 
\end{align*}
with $p$ as in the Lemma.  However, this is immediate, as
\begin{align*}
(-n+1)\frac{2d}{d-2s+2n-1}+d=\frac{d+d^2-2ds}{d-2s+2n-1}\geq \frac{d}{d-2s+2n-1},
\end{align*}
which implies that 
\begin{align*}
\||.|^{-n+1}\|_{ L^{p}(\B^d_1)}^p=\int_0^1 \rho^{(-n+1)p} \rho^{d-1}d\rho\leq\int_0^1 \rho^{\frac{d}{d-2s+2n-1}-1}d\rho<\infty.
\end{align*}
\end{proof}
As we aim to establish Strichartz estimates on $\Sf$, by recasting it as
an oscillatory integral of the resolvent of $\Lf$, our next step is a detailed analysis of the generalised spectral equation
\begin{equation}\label{eq:spectral equation}
  \begin{split}
&(\rho^2-1)u''(\rho)+\left((d-2s+2\lambda+2)\rho-\frac{d-1}{\rho}\right)u'(\rho)
\\
&\quad +\frac{d-2s+2\lambda}{4}(2\lambda +d-2s+2)u(\rho)+V(\rho)u(\rho)=f(\rho)
\end{split}
\end{equation}
with $f\in C^\infty_{rad} (\overline{\B^d_1})$.
\section{Analysis of  Eq.~\eqref{eq:spectral equation}}
To have a precise notation at hand, we will now heavily use symbol notation. For this we define functions of symbol type as follows. Let $I\subset \R$, $\rho_0\in \R \setminus I$, and $\alpha \in \R$. We say that a smooth function $f:I \to \C$ is of symbol type and write $f(\rho)=\O((\rho_0-\rho)^{\alpha})$ if
\begin{align*}
|\partial_\rho^n f(\rho)|\lesssim_n |\rho_0-\rho|^{\alpha-n}
\end{align*}
 for all $\rho \in I$ and all $n\in \mathbb{N}_0$. Likewise, let $g:U \subset \C\to \R$, then $g=\O(\langle\omega\rangle^{\alpha})$, provided that
 \begin{align*}
|\partial_\omega^n f(\varepsilon+i\omega)|\lesssim_n \langle\omega\rangle^{\alpha-n}
\end{align*}
where $\langle\omega\rangle$ denotes the Japanese bracket $\sqrt{1+|.|^2}$.
Analogously, 
$$
h(\rho,\lambda)=\O((\rho-\rho_0)^{\alpha} \langle\omega\rangle^{\beta}) \quad \text{ if } \quad |\partial_\rho^n\partial_\omega^k h(\rho,\lambda)|\lesssim_{n,k} |\rho_0-\rho|^{\alpha-n}\langle\omega\rangle^{\beta-k}
$$
for all $\ell,k\in \mathbb N$ and $\alpha,\beta \in \R$.
Moreover, we will from now on always assume that $\lambda=\varepsilon+i \omega \in S$ where
$$
S:= \{z \in i\R \times (s- \ceil s , s-\floor s)\}$$
in case $s\notin \mathbb N$ and 
$$
S:=\{z\in i\R \times (-\frac{1}{300},\frac{1}{300})\}$$
for $s\in \mathbb{N}$.
For such $\lambda$ we transform Eq.~\eqref{eq:spectral equation} by setting $$v(\rho)=\rho^{\frac{d-1}{2}}(1-\rho^2)^{\frac{3}{4}+\frac{\lambda}{2}-\frac{s}{2}}u(\rho).$$ For $f=0$ this yields the equation
\begin{equation}\label{eq:nofirst order}
  \begin{split}
v''(\rho)+\left(\frac{3-4\lambda-4\lambda^2+4s+8\lambda s-4s^2}{4(1-\rho^2)^2}+\frac{-3+4d-d^2}{4\rho^2(1-\rho^2)}
\right)&  v(\rho)
\\
-\frac{V(\rho)}{1-\rho^2}v(\rho)= 0.
\end{split}
\end{equation}

To construct solutions to this equation, we make one more definition.
\begin{defi}\label{defi:1}
Let $r>1$, $\rho_0\in [0,1),$ and $\lambda\in S$. Then, we define the function $\rho_\lambda$ as a smooth version of the function $\min\{\frac{r}{|s-\frac{1}{2}-\lambda|},\rho_0\}$.
\end{defi}
\begin{lem} \label{lem: fundi hom}
There exist $r>1$ and $\rho_0\in [0,1)$ such that for $\rho \in [\rho_\lambda,1),$  with $\rho_\lambda$ as defined in Definition \ref{defi:1} and $\lambda\in  S$, the equation
\begin{align} \label{eq: equation for h}
h''(\rho)+\left(\frac{3-4\lambda-4\lambda^2+4s+8\lambda s-4s^2}{4(1-\rho^2)^2}+\frac{-3+4d-d^2}{4\rho^2(1-\rho^2)}
\right) h(\rho)=0
\end{align}
has a fundamental system of solutions given by
\begin{align*}
h_1(\rho,\lambda)&:=\frac{\sqrt{1-\rho^2}}{\sqrt{2s-2\lambda-1}}\left(\frac{1-\rho}{1+\rho}\right)^{\frac{s}{2}-\frac{\lambda}{2}-\frac{1}{4}}[1+e_1(\rho,\lambda)]
\\
h_2(\rho,\lambda)&:=\frac{\sqrt{1-\rho^2}}{\sqrt{2s-2\lambda-1}}\left(\frac{1-\rho}{1+\rho}\right)^{\frac{\lambda}{2}-\frac{s}{2}+\frac{1}{4}}[1+e_2(\rho,\lambda)]
\end{align*}
where $e_j(\rho,\lambda)=(1-\rho)\O(\rho^{-1}\langle\omega\rangle^{-1})$ for $j=1,2$.
\end{lem}

\begin{proof}
Given that the term $\frac{-3+4d-d^2}{4\rho^2(1-\rho^2)} h(\rho)$ only has pole of first order at $\rho=1$ we treat it perturbatively and remark that \begin{align*}
w_1(\rho,\lambda)&:=\frac{\sqrt{1-\rho^2}}{\sqrt{2s-2\lambda-1}}\left(\frac{1-\rho}{1+\rho}\right)^{\frac{s}{2}-\frac{\lambda}{2}-\frac{1}{4}}
\\
w_2(\rho,\lambda)&:=\frac{\sqrt{1-\rho^2}}{\sqrt{2s-2\lambda-1}}\left(\frac{1-\rho}{1+\rho}\right)^{\frac{\lambda}{2}-\frac{s}{2}+\frac{1}{4}}
\end{align*}
are two linearly independent solutions to the equation
\begin{align}\label{eq:w}
w''(\rho)+\frac{3-4\lambda-4\lambda^2+4s+8\lambda s-4s^2}{4(1-\rho^2)^2}w(\rho)=0.
\end{align}
Moreover, their Wronskian $W(w_1(.,\lambda),w_2(.,\lambda))$ satisfies
\begin{align*}
W(w_1(.,\lambda),w_2(.,\lambda))=1.
\end{align*}
As a consequence, the variation of constants formula suggests the Volterra equation
\begin{align*}
h(\rho,\lambda)&=w_1(\rho,\lambda)+w_1(\rho,\lambda)\int_\rho^{\rho_1} \frac{w_2(t,\lambda)(3-4d+d^2)}{4t^2(1-t^2)}h(t,\lambda) dt
\\
&\quad-w_2(\rho,\lambda)\int_\rho^{\rho_1} \frac{w_1(t,\lambda)(3-4d+d^2)}{4t^2(1-t^2)}h(t,\lambda) dt
\end{align*}
for $\rho_1\leq 1$ to be chosen.
To continue, we divide this whole equation by $w_1$ and set $\widetilde{h}=\frac{h}{w_1}$ to arrive at the equation
\begin{align*}
\widetilde{h}(\rho,\lambda)&=1+\int_\rho^{\rho_1} \frac{w_2(t,\lambda)w_1(t,\lambda)(3-4d+d^2)}{4t^2(1-t^2)}\widetilde{h}(t,\lambda) dt
\\
&\quad-\frac{w_2(\rho,\lambda)}{w_1(\rho,\lambda)}\int_\rho^{\rho_1} \frac{w_1(t,\lambda)^2(3-4d+d^2)}{4t^2(1-t^2)}\widetilde{h}(t,\lambda) dt
\\
&=1+\int_\rho^{\rho_1} \frac{(3-4d+d^2)\left[1-\left( \frac{1-t}{1+t}\frac{1+\rho}{1-\rho}\right)^{s-\lambda-\frac{1}{2}}\right]}{4(2s-2\lambda-1)t^2}\widetilde{h}(t,\lambda) dt
\\
&=: 1+\int_\rho^{\rho_1} K(\rho,t,\lambda) \widetilde{h}(t,\lambda) dt.
\end{align*}
Furthermore, one estimates
\begin{align*}
\int_{|s-\frac{1}{2}-\lambda|}^{\rho_1}\sup_{|s-\frac{1}{2}-\lambda|^{-1}\leq \rho\leq t} |K(\rho,t,\lambda) |dt\lesssim \int_{|\lambda|^{-1}}^{\rho_1}\frac{1}{|2s-2\lambda-1|t^2} dt \lesssim 1.
\end{align*}
So, we can choose $\rho_1=1$ and infer the existence of a unique solution $\widetilde{h}(\rho,\lambda)$ to the equation 
\begin{align*}
\widetilde{h}(\rho,\lambda)&= 1+\int_\rho^1 K(\rho,t,\lambda) \widetilde{h}(t,\lambda) dt
\end{align*}
on the interval $[|s-\frac{1}{2}-\lambda|,1]$.
Moreover, from the arguments in the proof of Lemma  4.1 in \cite{DonWal22a} it follows that
$\widetilde{h}$ is of the form
\begin{align*}
\widetilde{h}(\rho,\lambda)&= 1+(1-\rho)\O(\langle \omega\rangle^{-1})+\O(\rho^{-1}(1-\rho)^2 \langle \omega\rangle^{-1}),
\end{align*}
where the inessential dependence on $\varepsilon=\Re \lambda $ is suppressed. 
However, for us this expansion of $\widetilde{h}$ is not good enough in terms of its regularity at the endpoint $\rho=1$. To derive better estimates, we compute that
\begin{align*}
\int_\rho^1 K(\rho,t,\lambda) dt &=\int_{\rho-1}^0 K(\rho,t+1,\lambda) dt=(1-\rho)\int_{-1}^0 K(\rho,y(1-\rho)+1,\lambda) dy
\\
&=  (1-\rho)\int_{-1}^0\frac{(3-4d+d^2)\left[1-\left(- \frac{y(1+\rho)}{2+ y(1-\rho)}\right)^{s-\lambda-\frac{1}{2}}\right]}{4(2s-2\lambda-1)(y(1-\rho)+1)^2} dy \\
&=  -(1-\rho)\int_{0}^1\frac{(3-4d+d^2)\left[1-\left( \frac{y(1+\rho)}{2-y(1-\rho)}\right)^{s-\lambda-\frac{1}{2}}\right]}{4(2s-2\lambda-1)(1-y(1-\rho))^2} dy
\end{align*}
which is a smooth expression up to $\rho=1$. As a consequence, one concludes that $\widetilde h(\rho,\lambda)$ is indeed smooth on the closed interval $[|s-\frac{1}{2}-\lambda|,1]$.
Further, we compute that
\begin{align*}
\partial_\rho \left( \frac{y(1+\rho)}{2- y(1-\rho)}\right)^{s-\lambda-\frac{1}{2}}&=(s-\lambda-\tfrac{1}{2})\left( \frac{y(1+\rho)}{2- y(1-\rho)}\right)^{s-\lambda-\frac{3}{2}}\frac{2y(1-y)}{(y(\rho-1)+2)^2}
\\
&=(s-\lambda-\tfrac{1}{2})y(1-y) \left( \frac{y(1+\rho)}{2- y(1-\rho)}\right)^{s-\lambda-\frac{3}{2}}\frac{2}{(y(\rho-1)+2)^2}
\end{align*}
as well as 
\begin{align}\label{eq:indentity}
 \frac{2(s-\lambda-\tfrac{1}{2})}{y[y(1-\rho)-2]} \left(\frac{y}{2- y(1-\rho)}\right)^{s-\lambda-\frac{3}{2}}=\partial_y \left(\frac{y}{2- y(1-\rho)}\right)^{s-\lambda-\frac{1}{2}}.
\end{align}
We keep this in mind and note that for any $\ell \in \mathbb{N}$ fixed on computes
\begin{align*}
\partial_\rho^\ell \int_\rho^1 K(\rho,t,\lambda) dt & = \sum_{j=0}^\ell \O(\langle\omega\rangle^{-1+j}) \int_0^1 y^{j}(y-1)^j \left( \frac{y(1+\rho)}{2- y(1-\rho)}\right)^{s-\lambda-\frac{1}{2}-j}f_j^\ell(\rho,y) dy
\\
&\quad+ \int_0^1 \O(\langle \omega \rangle^{-1})
f_\ell(\rho,y) dy 
\end{align*}
for smooth functions $f_j^\ell$ and $f_\ell$.
Hence, we can use the identity \eqref{eq:indentity} and a number of integrations by parts to infer that for $\rho$ in $[1-\delta,1]$, the estimate
$$
\left|\partial_\rho^\ell \int_\rho^1 K(\rho,t,\lambda) dt\right|\lesssim \langle\omega\rangle^{-1}
$$
holds, provided that $\delta$ is chosen sufficiently small.
Lastly, to estimate 
$$
\partial_\rho^\ell \partial_\omega^k \int_\rho^1 K(\rho,t,\lambda) dt
$$
near $\rho=1$, we compute that 
\begin{align*}
\partial_\omega^k \partial_\rho^\ell \int_\rho^1 K(\rho,t,\lambda) dt&= \partial_\omega^k\left[ \O(\langle\omega\rangle^{-1}) \int_0^1 \left( \frac{y(1+\rho)}{2- y(1-\rho)}\right)^{s-\lambda-\frac{1}{2}}f_{\ell}(\rho,y) dy
\right]
\\
&\quad 
+ \int_0^1 \O(\langle \omega \rangle^{-1-k})
g_{\ell}(\rho,y) dy 
\end{align*}
for smooth functions $f_{\ell}$ and $g_{\ell}$.
Moreover, for $|\Im \lambda|=|\omega|<1$, one estimates 
\begin{align*}
&\quad\left|\quad\partial_\omega^k  \int_0^1 \left( \frac{y(1+\rho)}{2- y(1-\rho)}\right)^{s-\lambda-\frac{1}{2}}f_{\ell}(\rho,y) dy\right| 
\\
&=  \left|\partial_\omega^k \int_0^1 \exp\left[(s-\lambda-\tfrac{1}{2})\log \left( \frac{y(1+\rho)}{2- y(1-\rho)}\right)\right]f_{\ell}(\rho,y) dy \right|\lesssim 1.
\end{align*}
Hence, we can assume that $|\omega>1|$ and set 
$$
\varphi(y;\rho)=\log \left( \frac{y(1+\rho)}{2- y(1-\rho)}\right).
$$
Then,
$$\varphi^{-1}(x;\rho)=\frac{2e^x}{e^x(1-\rho)+1+\rho}$$ and
\begin{align*}
&\quad\partial_\omega^k \int_0^1 \exp\left[(s-\lambda-\tfrac{1}{2})\log \left( \frac{y(1+\rho)}{2- y(1-\rho)}\right)\right]f_\ell(\rho,y) dy 
\\
&=\partial_\omega^k \int_0^1 \exp\left[(s-\lambda-\tfrac{1}{2})\varphi(y;\rho)\right]f_\ell(\rho,\varphi^{-1}(\varphi(y;\rho);\rho)) dy
\\
&= \partial_\omega^k \int_{-\infty}^{0} \exp\left[(s-\lambda-\tfrac{1}{2})x\right]f_\ell(\rho,\varphi^{-1}(x;\rho))(\varphi^{-1})'(x;\rho) dx
\\
&= \partial_\omega^k \int_{-\infty}^0\left(\frac{1}{\omega}\exp\left[(s-\lambda-\tfrac{1}{2})x\omega^{-1}\right]f_\ell(\rho,\varphi^{-1}(x \omega^{-1};\rho))(\varphi^{-1})'(\omega^{-1} x;\rho) dx \right)
\end{align*}
and one readily establishes 
\begin{align*}
\left|\partial_\omega^k \int_{-\infty}^0\left(\frac{1}{\omega}\exp\left[(s-\lambda-\tfrac{1}{2})x\omega^{-1}\right]f_\ell(\rho,\varphi^{-1}(x \omega^{-1};\rho))(\varphi^{-1})'(\omega^{-1} x;\rho) dx \right)\right|&\lesssim |\omega|^{-k}.
\end{align*}
So, by putting all of our bounds on $\widetilde h$, together we conclude that
\begin{align*}
\widetilde h(\rho,\lambda)=(1-\rho)\O(\rho^{-1}\langle \omega\rangle^{-1})
\end{align*}
and, upon setting $e_1=\widetilde h-1$, the first of the claimed solutions has been constructed.
For the second, we pick a $r>1$ and $\rho_0>0$ large enough such that $h_1$ is nonvanishing on $[\rho_\lambda,1]$ and note that, since 
$$\widehat{w}(\rho,\lambda):=w_1(\rho,\lambda)\int_{\rho_\lambda}^\rho w_1(t,\lambda)^{-2} dt$$ 
 is also a solution to Eq. \eqref{eq:w}, we can find functions $c_1(\lambda)$, $c_2(\lambda)$ such that
\begin{align*}
 w_2(\rho,\lambda)= c_1(\lambda) w_1(\rho,\lambda)+c_2(\lambda)\widehat{w}(\rho,\lambda).
\end{align*}
Explicitly, these are given by
\begin{align*}
c_1(\lambda)&=\frac{W(w_2(.,\lambda),\widehat{w}(.,\lambda))}{W(w_1(.,\lambda),\widehat{w}(.,\lambda))}=W(w_2(.,\lambda),\widehat{w}(.,\lambda))(\rho_\lambda)=\O(\langle\omega\rangle^0),
\\
c_2(\lambda)&=\frac{W(w_1(.,\lambda),w_2(.,\lambda))}{W(w_1(.,\lambda),\widehat w(.,\lambda))} =1.
\end{align*}
So, we make the ansatz
\begin{align*}
h_2(\rho,\lambda)&=(c_1(\lambda)+c(\lambda))h_1(\rho,\lambda)+h_1(\rho,\lambda)\int_{\rho_\lambda}^\rho h_1(t,\lambda)^{-2} dt
\\
&=c(\lambda) h_1(\rho,\lambda)+w_2(\rho,\lambda)[1+e_1(\rho,\lambda)]
+h_1(\rho,\lambda)\int_{\rho_\lambda}^\rho h_1(t,\lambda)^{-2}-b_1(t,\lambda)^{-2} dt
\\
&=
c(\lambda)h_1(\rho,\lambda)+w_2(\rho,\lambda)[1+e_1(\rho,\lambda)]
-h_1(\rho,\lambda)\int_{\rho_\lambda}^\rho \frac{e_1(t,\lambda)^2+2e_1(t,\lambda)}{w_1(t,\lambda)^2[1+e_1(t,\lambda)]^2} dt
\end{align*}
for some function $c(\lambda)$.
Now,
\begin{align*}
&\quad \int_{\rho_\lambda}^\rho \frac{e_1(t,\lambda)^2+2e_1(t,\lambda)}{w_1(t,\lambda)^2[1+e_1(t,\lambda)]^2} dt
\\
&= (2s-2\lambda-1)\int_{\rho_\lambda}^\rho (1+t)^{s-\lambda-\frac{3}{2}} (1-t)^{-s+\lambda-\frac{1}{2}} \frac{2e_1(t,\lambda)+e_1(t,\lambda)^2}{[1+e_1(t,\lambda)]^2}  dt.
\end{align*}
Hence, $\ceil s$ integrations by parts show that
\begin{align*}
&\quad- h_1(\rho,\lambda)\int_{\rho_\lambda}^\rho \frac{2e_1(t,\lambda)+e_1(t,\lambda)^2}{w_1(t,\lambda)^2[1+e_1(t,\lambda)]^2} dt \\
&= c_2(\lambda) h_1(\rho,\lambda) 
\\
&\quad + h_1(\rho,\lambda) \sum_{j=0}^{\ceil s-1}\O(\langle\omega\rangle^{-1-j})(1-\rho)^{-s+\lambda +\frac{1}{2}+j} \partial_\rho^j \left[(1+\rho)^{s-\lambda-\frac{3}{2}}  \frac{2e_1(\rho,\lambda)
+e_1(\rho,\lambda)^2}{[1+e_1(\rho,\lambda)]^2}\right]
\\
&\quad
+ h_1(\rho,\lambda)\O(\langle\omega\rangle^{-\ceil s})\int_{\rho_\lambda}^\rho (1-t)^{\ceil s-s-\lambda -\frac{1}{2}} \partial_t^{\ceil s} \left[(1+t)^{s-\lambda-\frac{3}{2}}  \frac{2e_1(t,\lambda)
+e_1(t,\lambda)^2}{[1+e_1(t,\lambda)]^2}\right] dt
\\
&=  c_3(\lambda) h_1(\rho,\lambda) 
\\
&\quad+ h_1(\rho,\lambda) \sum_{j=0}^{\ceil s-1}\O(\langle\omega\rangle^{-1-j})(1-\rho)^{-s+\lambda +\frac{1}{2}+j} \partial_\rho^j \left[(1+\rho)^{s-\lambda-\frac{3}{2}}  \frac{2e_1(\rho,\lambda)
+e_1(\rho,\lambda)^2}{[1+e_1(\rho,\lambda)]^2}\right]
\\
&\quad
+ h_1(\rho,\lambda)\O(\langle\omega\rangle^{-\ceil s})\int_\rho^1 (1-t)^{\ceil s -s-\lambda -\frac{1}{2}} \partial_t^{\ceil s} \left[(1+t)^{s-\lambda-\frac{3}{2}}  \frac{2e_1(t,\lambda)
+e_1(t,\lambda)^2}{[1+e_1(t,\lambda)]^2}\right] dt
\end{align*}
for appropriate function $c_2(\lambda)$ and $c_3(\lambda)$. Furthermore, we observe that
\begin{align*}
&\quad\sum_{j=0}^{\ceil s-1}\O(\langle\omega\rangle^{-1-j})(1-\rho)^{-s+\lambda +\frac{1}{2}+j} \partial_\rho^j \left[(1+\rho)^{s-\lambda-\frac{3}{2}}  \frac{2e_1(\rho,\lambda)
+e_1(\rho,\lambda)^2}{[1+e_1(\rho,\lambda)]^2}\right]
\\
&=(1+\rho)^{-\lambda}\sum_{j=0}^{\ceil s-1}(1-\rho)^{-s+\lambda+\frac12+j}\O(\langle\omega\rangle^{-1-j}\rho^{-1-j})
\end{align*}
for $\rho \in [\rho_\lambda,1]$. This implies that
\begin{align*}
&\quad h_1(\rho,\lambda)\sum_{j=0}^{\ceil s-1}\O(\langle\omega\rangle^{-1-j})(1-\rho)^{-s+\lambda +\frac{1}{2}+j} \partial_\rho^j \left[(1+\rho)^{s-\lambda-\frac{3}{2}}  \frac{2e_1(\rho,\lambda)
+e_1(\rho,\lambda)^2}{[1+e_1(\rho,\lambda)]^2}\right]
\\
&=w_2(\rho,\lambda)\O(\langle\omega\rangle^{-1}\rho^{-1}).
\end{align*}
Furthermore, 
\begin{align*}
&\quad h_1(\rho,\lambda)\O(\langle\omega\rangle^{-\ceil s})\int_\rho^1 (1-t)^{\ceil s-s-\lambda -\frac{1}{2}} \partial_t^{\ceil s} \left[(1+t)^{s-\lambda-\frac{3}{2}}  \frac{2e_1(t,\lambda)
+e_1(t,\lambda)^2}{[1+e_1(t,\lambda)]^2}\right] dt
\\
&=
w_2(\rho,\lambda)[1+e_1(\rho,\lambda)] \left(\frac{1-\rho}{1+\rho}\right)^{s-\lambda-\frac{1}{2}}\O(\langle\omega\rangle^{-\ceil s})
\\
&\quad \times \int_\rho^1 (1-t)^{\ceil s-s-\lambda -\frac{1}{2}} \partial_t^{\ceil s} \left[(1+t)^{s-\lambda-\frac{3}{2}}  \frac{2e_1(t,\lambda)
+e_1(t,\lambda)^2}{[1+e_1(t,\lambda)]^2}\right] dt
\\
&=w_2(\rho,\lambda)\O(\langle\omega\rangle^{-1}\rho^{-1})
\end{align*}
where the last step follows from similar considerations as in the construction of $h_1$.
Hence, by setting $c_2=-c$, we arrive at the desired solution and conclude this proof.
\end{proof}
Proceeding, we  define the diffeomorphism $\varphi:(0,1)\to (0,\infty)$ as
 $$
 \varphi(\rho):=\frac{1}{2}(\log(1+\rho)-\log(1-\rho))
 $$
and compute that
 $$
 \varphi'(\rho)=\frac{1}{(1-\rho^2)}.
 $$
Further, the associated Liouville-Green Potential $Q_{\varphi}$, which is defined as
 $$
 Q_\varphi(\rho):=-\frac{3}{4}\frac{\varphi''(\rho)^2}{\varphi'(\rho)^2}+\frac{1}{2}\frac{\varphi'''(\rho)}{\varphi'(\rho)^2},
 $$ is given by
 $$
 Q_{\varphi}(\rho )=\frac{1}{(1-\rho^2)^2}.
 $$ 
 This leads us to rewriting equation \eqref{eq:nofirst order} as
\begin{equation*}
  \begin{split}
&\quad v''(\rho)+\left(\frac{-1-4\lambda-4\lambda^2+4s+8\lambda s-4s^2}{4(1-\rho^2)^2}+\frac{-3+4d-d^2}{4\varphi(\rho)^2(1-\rho^2)^2}
\right) v(\rho)+Q_\varphi(\rho)v(\rho)
\\
&=\left(\frac{V(\rho)}{(1-\rho^2)}+\frac{-3+4d-d^2}{4\varphi(\rho)^2(1-\rho^2)^2}-\frac{-3+4d-d^2}{4\rho^2(1-\rho^2)}\right) v(\rho).
\end{split}
\end{equation*}
Now, we perform a Liouville-Green transformation. For this, we set  $w(\varphi(\rho))= \varphi'(\rho)^{\frac{1}{2}}v(\rho)$ which transforms
\begin{equation}\label{eq:beforebessel}
\begin{split}
v''(\rho)&+\left(\frac{-1-4\lambda-4\lambda^2+4s+8\lambda s-4s^2}{4(1-\rho^2)^2}+\frac{-3+4d-d^2}{4\varphi(\rho)^2(1-\rho^2)^2}
\right) v(\rho)
\\
&+
Q_\varphi(\rho)v(\rho)
=0
\end{split}
\end{equation}
into the equation
\begin{equation}\label{eq:bessel}
w''(\varphi(\rho))-\left(s-\frac{1}{2}-\lambda\right)^2w(\varphi(\rho))+\frac{-3+4d-d^2}{4\varphi(\rho)^2}w(\varphi(\rho))=0.
\end{equation}
Let $a(\lambda)=i\left(s-\frac{1}{2}-\lambda\right). $ Then, given that Eq.~ \eqref{eq:bessel} is a Bessel equation, it has a fundamental system given by
\begin{align*}
&\sqrt{\varphi(\rho)}J_{\frac{d-2}{2}}(a(\lambda)\varphi(\rho))\\
&\sqrt{\varphi(\rho)}Y_{\frac{d-2}{2}}(a(\lambda)\varphi(\rho)),
\end{align*}
where $J_\nu $ and $Y_\nu$ denote the Bessel functions of the first and second kind, respectively.
Consequently, a fundamental system for Eq.~\eqref{eq:beforebessel} his given by
\begin{align*}
b_{1}(\rho,\lambda)&= \sqrt{(1-\rho^2)\varphi(\rho)}J_{\frac{d-2}{2}}(a(\lambda)\varphi(\rho))\\
b_{2}(\rho,\lambda)&=\sqrt{(1-\rho^2)\varphi(\rho)}Y_{\frac{d-2}{2}}(a(\lambda)\varphi(\rho)).
\end{align*}
We can also enlarge $r>$ and $\rho_0$ such that neither $h_1(.,\lambda)$ nor $h_2(.,\lambda)$ are vanishing on the interval $(\rho_\lambda,1).$ 
Further, in addition to $\rho_\lambda$, we define  $\widehat{\rho}_\lambda:=\min \{4\frac{r}{|a(\lambda)|},\frac{1}{2}(1+\rho_0)\}$ (Strictly, we again need to use smooth version of this expression).
\begin{lem}\label{lem:fundi near 0}
Let $\rho\in (0,\widehat{\rho}_\lambda)$. Then there exists a fundamental system of solutions for Eq.~\eqref{eq:nofirst order} given by \begin{align*}
\psi_1(\rho,\lambda)&=b_1(\rho,\lambda)[1+\rho^2 e_3(\rho,\lambda)] 
\\
&=\sqrt{(1-\rho^2)\varphi(\rho)}J_{\frac{d-2}{2}}( a(\lambda)\varphi(\rho))[1+\rho^2 e_3(\rho,\lambda)]
\\
\psi_2(\rho,\lambda)&= b_2(\rho,\lambda)[1+\rho^2 e_3(\rho,\lambda)]+ \psi_1(\rho,\lambda)e_4(\rho,\lambda).
\\
&=\sqrt{(1-\rho^2)\varphi(\rho)}Y_\frac{d-2}{2}( a(\lambda)\varphi(\rho))[1+\rho^2e_3(\rho,\lambda)]+ \psi_1(\rho,\lambda)e_4(\rho,\lambda)
\end{align*}
where $e_3$ satisfies
\begin{align*}
e_3(\rho,\lambda)=\widehat{e}_3(\rho,\lambda)+\O(\rho^{d-\frac{5}{2}}\langle\omega\rangle^{d-\frac{5}{2}})
\end{align*}
with
\begin{align}\label{esti:e1}
\partial_\rho^m\partial_\omega^n \widehat{e}_3(\rho,\lambda)\lesssim_{m,n} \langle\omega\rangle^{m-n}
\end{align}
and with
\begin{align*}
e_4(\rho,\lambda)=\O(\rho^{-1}\langle\omega\rangle^{-3})+\sum_{j=0}^{d-5}\rho^{4-d+j} \O(\langle\omega\rangle^{2-d+j})\widetilde{e_j}(\rho,\lambda)
\end{align*}
where all of the $\widetilde{e}_j$ also satisfy the estimate \eqref{esti:e1}.
\end{lem}
\begin{proof}
Given that $b_1$ and $b_2$ form a fundamental system of solutions for the equation
\begin{align*}
&\quad v''(\rho)+\left(\frac{-1-4\lambda-4\lambda^2+4s+8\lambda s-4s^2}{4(1-\rho^2)^2}+\frac{-3+4d-d^2}{4\varphi(\rho)^2(1-\rho^2)^2}
\right) v(\rho)+Q_\varphi(\rho)v(\rho)
&=0
\end{align*}
with Wronskian 
$$
W(b_1(.,\lambda),b_2(.,\lambda))=\frac{2}{\pi},$$
we need to construct a solution to the fixed point equation
\begin{align}\label{Ansatz1}
b(\rho,\lambda)= b_1(\rho,\lambda)&-\frac{\pi}{2}b_1(\rho,\lambda)\int_{0}^{\rho} b_2(t,\lambda)\widetilde{V}(s)b(t,\lambda)d t\\
&+\frac{\pi}{2}b_2(\rho,\lambda)\int_{0}^{\rho} b_1(t,\lambda)\widetilde{V}(t,\lambda)b(t,\lambda) d t \nonumber
\end{align}
where
$$
\widetilde{V}(\rho)= \frac{V(\rho)}{(1-\rho^2)}+\frac{-3+4d-d^2}{4\varphi(\rho)^2(1-\rho^2)^2}-\frac{-3+4d-d^2}{4\rho^2(1-\rho^2)} \in C^\infty ([0,1)).
$$
To accomplish this we expand $b_1$ and $b_2$ around $0$ to conclude that
\begin{align*}
b_1(\rho,\lambda)=\varphi(\rho)^{\frac{d-1}{2}}\O(\langle\omega\rangle^{\frac{d-2}{2}})[1+\widehat{b}_1(\rho,\lambda)] 
\end{align*}
where $\widehat{b}_1$ satisfies the estimates
\begin{align*}
\partial_\rho^m\partial_\omega^n \widehat{b}_1(\rho,\lambda)\lesssim_{m,n} \langle\omega\rangle^{m-n}
\end{align*}
for all $\rho\in (0,\widehat{\rho}_\lambda)$.
Similarly, for $d$ even, we conclude that
\begin{align*}
b_2(\rho,\lambda)=\varphi(\rho)^{\frac{3-d}{2}}\O(\langle \omega\rangle^{\frac{2-d}{2}}) [1+\widehat{b}_2(\rho,\lambda)]+\O(\rho^{\frac{d-2}{2}}\langle \omega\rangle^{\frac{d-3}{2}})
\end{align*}
where $\widehat{b}_2$ satisfies the same estimate as $\widehat{b}_1$,
while for $d$ odd one obtains a simpler form of $b_2$ by using the explicit formulae
\begin{align*}
J_{\frac{d}{2}-1}(z)&=(-1)^{\frac{d-3}2}\sqrt{\frac{2}{\pi}} z^{\frac{d-2}{2}}\left(\frac{1}{z}\frac{d}{dz}\right)^{\frac{d-3}{2}} \frac{\sin(z)}{z}
\\
Y_{\frac{d}{2}-1}(z)&=-(-1)^{\frac{d-3}2}\sqrt{\frac{2}{\pi}} z^{\frac{d-2}{2}}\left(\frac{1}{z}\frac{d}{dz}\right)^{\frac{d-3}{2}} \frac{\cos(z)}{z}.
\end{align*}
Since the odd dimensional case is technically slightly less involved, we only demonstrate the construction of $\psi_1$ and $\psi_2$ for $d$ even. 
From  $\Im a(\lambda)\neq 0$ for all $\lambda\in S$, we can infer that $b_1$ is nonvanishing on $(0,\widehat{\rho}_\lambda]$, as all zeros of $J_\nu$ are real provided that $\nu>-1$ (see \cite{Olv97}, p. 244 Theorem 6.2). As a consequence, we are able to divide the whole equation~ \eqref{Ansatz1} by $b_1$, which, upon setting $e=\frac{b}{b_1}$,
yields the integral equation
\begin{equation}\label{eq:int0}
e(\rho,\lambda)=1+\int_{0}^{\rho}K(\rho,t,\lambda)e(t,\lambda) dt,
\end{equation}
where
$$
K(\rho,t,\lambda)=\frac{\pi\widetilde{V}(t)}{2}\left( \frac{b_2(\rho,\lambda)}{b_1(\rho,\lambda)}b_1(t,\lambda)^2-b_2(t,\lambda)b_1(t,\lambda)\right).
$$
By plugging in the symbol forms of $b_1$ and $b_2$, we deduce that
\begin{align*}
\frac{b_2(\rho,\lambda)}{b_1(\rho,\lambda)}b_1(t,\lambda)^2=\varphi(\rho)^{2-d}\O(\langle\omega\rangle^{0})
\varphi(t)^{d-1}\frac{[1+\widehat{b}_2(\rho,\lambda)][1+\widehat{b}_1(t,\lambda)]^2}{[1+\widehat{b}_1(\rho,\lambda)]}+\O(\rho^{-\frac{1}{2}}t^{d-1}\langle\omega \rangle^{d-\frac{5}{2}})
\end{align*}
and
\begin{align*}
b_1(t,\lambda)b_2(t,\lambda)=\varphi(t)\O(\langle\omega\rangle^{0})[1+\widehat{b}_1(t,\lambda)][1+\widehat{b}_2(t,\lambda)]+ \O(t^{d-\frac{3}{2}}\langle\omega\rangle^{d-\frac{5}{2}}).
\end{align*}
Thus,
\begin{align*}
\int_0^{\widehat{\rho}_\lambda} \sup_{\rho\in [t,\widehat\rho_\lambda]}|K(\rho,t,\lambda)|dt \lesssim \langle\omega\rangle^{-2}
\end{align*}
and we obtain the existence of a unique  solution $\widetilde e(\rho,\lambda)$ to Eq.~\eqref{eq:int0} with
\begin{align*}
\widetilde e(\rho,\lambda)= 1+O(\rho^2 \langle\omega\rangle^0).
\end{align*}
Moreover, from the explicit form of $K$, the desired estimates on $\widetilde e$ follow from a repeated usage of the identity
\begin{align*}
\widetilde e(\rho,\lambda)=1+\int_0^\rho  K(\rho,t,\lambda) dt+\int_0^\rho \int_0^t K(\rho,y,\lambda) \widetilde e(y,\lambda) dt dy.
\end{align*}
Thus, we set $e_3=\widetilde e-1 $ and arrive at the desired form of $\psi_1$.
For the second solution, we pick a $\widehat \rho \in (0,1]$ such that $\psi_1$ does not vanish for $\rho\leq\min\{\widehat \rho,\widehat{\rho}_\lambda\}=:\widetilde{\rho}_\lambda$ for any $\lambda \in S.$ Next, as
$$\widetilde{b}_1(\rho,\lambda):=b_1(\rho,\lambda)\int_{\rho}^{\widetilde{\rho}_\lambda} b_1(t,\lambda)^{-2} dt$$  also solves Eq.~\eqref{eq:beforebessel}, there exist functions $c_1(\lambda),c_2(\lambda)$ such that
\begin{align*}
b_2(\rho,\lambda)=c_1(\lambda) b_1(\rho,\lambda)+ c_2(\lambda)\widetilde{b}_1(\rho,\lambda).
\end{align*}
Moreover, we have the explicit formula
\begin{align*}
c_1(\lambda)&=\frac{W(b_2(.,\lambda),\widetilde{b}_1(.,\lambda))}{W(b_1(.,\lambda),\widetilde{b}_1(.,\lambda))}\\
c_2(\lambda)&=-\frac{W(b_2(.,\lambda),b_1(.,\lambda))}{W(b_1(.,\lambda),\widetilde{b}_1(.,\lambda))}.
\end{align*}
Using that
$W(b_2(.,\lambda),b_1(.,\lambda))=-\frac{2}{\pi}$ and $ W(b_1(.,\lambda),\widetilde{b}_1(.,\lambda))=-1$, we infer that
$c_2=-\frac{2}{\pi}$ and $c_1(\lambda)=-W(b_2(.,\lambda),\widetilde{b_1}(.,\lambda))$.
Evaluating $W(b_2(.,\lambda),\widetilde{b}_1(.,\lambda))$ at $\widetilde{\rho}_\lambda$ yields
\[
W(b_2(.,\lambda),\widetilde{b_1}(.,\lambda))=-b_2(\widetilde{\rho}_\lambda,\lambda)b_1(\widetilde{\rho}_\lambda,\lambda)^{-1}=\O(\langle\omega\rangle^{0}).
\]
Keeping these facts in mind, we turn our attention to $\psi_2$ and remark that a second solution of Eq.~\eqref{eq:nofirst order} is given by
$\widetilde{\psi}_1(\rho,\lambda)=\psi_1(\rho,\lambda)\int_{\rho}^{\widetilde{\rho}_\lambda} \psi_1(t,\lambda)^{-2} dt$.
Considering this, we calculate
\begin{align*}
\psi_2(\rho,\lambda):&=c_1(\lambda)\psi_1(\rho,\lambda)+ c_2\psi_1(\rho,\lambda)\int_{\rho}^{\widetilde{\rho}_\lambda} \psi_1(t,\lambda)^{-2}dt
\\
&=c_1(\lambda)\psi_1(\rho,\lambda)+
   c_2\psi_1(\rho,\lambda)\int_{\rho}^{\widetilde{\rho}_\lambda}
   b_1(t,\lambda)^{-2}dt \\
&\quad +c_2\psi_1(\rho,\lambda)\int_{\rho}^{\widetilde{\rho}_\lambda}\left [
   \psi_1(t,\lambda)^{-2}-b_1(t,\lambda)^{-2}\right ] dt
\\
&=b_2(\rho,\lambda)[1+\rho^2 e_3(\rho,\lambda)] +\frac{2}{\pi}\psi_1(\rho,\lambda)\int_{\rho}^{\widetilde{\rho}_\lambda} \frac{2t^2 e_3(t,\lambda)+t^4 e_3(t,\lambda)^2}{b_1(t,\lambda)^2[1+t^2 e_3(t,\lambda)]^2} dt.
\end{align*}
Note that
\begin{align*}
\int_{\rho}^{\widetilde{\rho}_\lambda} \frac{2t^2 e_3(t,\lambda)+t^4 e_3(t,\lambda)^2}{b_1(t,\lambda)^2[1+t^2 e_3(t,\lambda)]^2} dt&=\int_{\rho}^{\widetilde{\rho}_\lambda} \frac{\O(t^{d-\frac{1}{2}}\langle\omega\rangle^{d-\frac{5}{2}})}{b_1(t,\lambda)^2[1+t^2 e_3(t,\lambda)]^2} dt 
\\
&\quad+\int_{\rho}^{\widetilde{\rho}_\lambda} \frac{2t^2 \widehat e_3(t,\lambda)+t^4  \widehat e_3(t,\lambda)^2}{b_1(t,\lambda)^2[1+t^2 e_3(t,\lambda)]^2} dt.
\end{align*}
Moreover,
\begin{align*}
\int_{\rho}^{\widetilde{\rho}_\lambda} \frac{\O(t^{d-\frac{1}{2}}\langle\omega\rangle^{d-\frac{5}{2}})}{b_1(t,\lambda)^2[1+t^2 e_3(t,\lambda)]^2} dt &=\int_{\rho}^{\widetilde{\rho}_\lambda} \O(t^{\frac{1}{2}}\langle\omega\rangle^{-\frac{1}{2}}) dt
= \O(\rho^{\frac{3}{2}}\langle\omega\rangle^{-\frac{1}{2}})+\O(\langle\omega\rangle^{-2})
\\
&=\O(\rho^{-1}\langle\omega\rangle^{-3})
\end{align*}
for $\rho \in (0,\widehat{\rho}_\lambda).$
Hence, we only have to focus on
\begin{align*}
I(\rho,\lambda):=\int_{\rho}^{\widetilde{\rho}_\lambda} \frac{2t^2 \widehat e_3(t,\lambda)+t^4  \widehat e_3(t,\lambda)^2}{b_1(t,\lambda)^2[1+t^2 e_3(t,\lambda)]^2} dt
=\int_{\rho}^{\widetilde{\rho}_\lambda} \frac{\O(\langle\omega\rangle^{2-d})t^{3-d} [2  \widehat e_3(t,\lambda)+t^2 \widehat e_3(t,\lambda)^2]}{[1+\widehat b_1(t,\lambda)]^2[1+t^2 e_3(t,\lambda)]^2} dt.
\end{align*}
Now,
\begin{align*}
 \frac{1}{[1+t^2 e_3(t,\lambda)]^2} &=  \frac{1}{[1+t^2 \widehat e_3(t,\lambda)+ \O(t^{d-\frac12}\langle\omega
 \rangle^{d-\frac{5}{2}})]^2} 
 \\
 &=  \frac{1}{[1+t^2 \widehat e_3(t,\lambda)+ \O(t^{d-\frac12}\langle\omega
 \rangle^{d-\frac{5}{2}})]^2} - \frac{1}{[1+t^2 \widehat e_3(t,\lambda)]^2}
 + \frac{1}{[1+t^2 \widehat e_3(t,\lambda)]^2} 
 \\
 &=\O(t^{d-\frac12}\langle\omega
 \rangle^{d-\frac{5}{2}})+ \frac{1}{[1+t^2 \widehat e_3(t,\lambda)]^2}.
\end{align*}
Therefore,
\begin{align*}
I(\rho,\lambda)
=\int_{\rho}^{\widetilde{\rho}_\lambda} \frac{\O(\langle\omega\rangle^{2-d})t^{3-d} [2  \widehat e_3(t,\lambda)+t^2 \widehat e_3(t,\lambda)^2]}{[1+\widehat b_1(t,\lambda)]^2[1+t^2\widehat e_3(t,\lambda)]^2} + \O(t^{\frac{5}{2}}\langle \omega\rangle^{-\frac 12}) dt
\end{align*}
and for $d\geq 5$ integrating by parts yields 
\begin{align*}
&\quad\int_{\rho}^{\widetilde{\rho}_\lambda} t^{3-d} \frac{2 \widehat e_3(t,\lambda)+t^2 \widehat e_3(t,\lambda)^2}{[1+\widehat b_1(t,\lambda)]^2[1+t^2 \widehat e_3(t,\lambda)]^2} dst
\\
&= \frac{\widetilde{\rho}_\lambda^{4-d}}{4-d} \frac{2 \widehat e_3(\widetilde{\rho}_\lambda,\lambda)+\rho^2 \widehat e_3(\widetilde{\rho}_\lambda,\lambda)^2}{[1+\widehat b_1(s,\lambda)]^2[1+\widetilde{\rho}_\lambda^2 \widehat e_3(\widetilde{\rho}_\lambda,\lambda)]^2}
\\
&\quad- \frac{\rho^{4-d}}{4-d} \frac{2 \widehat e_3(\rho,\lambda)+\rho^2 \widehat e_3+(\rho,\lambda)^2}{[1+\widehat b_1(s,\lambda)]^2[1+\rho^2 \widehat e_3(\rho,\lambda)]^2} 
\\
&\quad -\int_{\rho}^{\widetilde{\rho}_\lambda} \frac{t^{4-d}}{4-d} \partial_t\left[\frac{2 \widehat e_3(t,\lambda)+t^2\widehat  e_3(t,\lambda)^2}{[1+\widehat b_1(t,\lambda)]^2[1+\rho^2 \widehat e_3(t,\lambda)]^2}\right] dt.
  \end{align*}
Moreover, one has that
\begin{align*}
\widetilde{e}_0(\rho,\lambda):=\frac{2  \widehat e_3(\rho,\lambda)+\rho^2 \widehat e_3(\rho,\lambda)^2}{[1+\widehat b_1(\rho,\lambda)]^2[1+\rho^2 \widehat e_3(\rho,\lambda)]^2} 
\end{align*}
satisfies
\begin{align} \label{esti:e0}
\partial_\rho^m\partial_\omega^n \widetilde{e}_0(\rho,\lambda)\lesssim_{m,n} \langle\omega\rangle^{m-n}
\end{align}
and
\begin{align}\label{esti:e}
\widetilde{e}_0(\widehat{\rho}_\lambda,\lambda)=\O(\langle\omega\rangle^0).
\end{align}
Hence, 
\begin{align*}
I(\rho,\lambda)&=\rho^{4-d}\O(\langle\omega\rangle^{2-d})\widetilde{e}_0(\rho,\lambda)+\O(\langle\omega\rangle^{-2})
\\
&\quad+\O(\langle\omega\rangle^{2-d})\int_{\rho}^{\widetilde{\rho}_\lambda} \frac{t^{4-d}}{4-d} \partial_t\left[\frac{2\widehat e_3(t,\lambda)+t^2 \widehat e_3(t,\lambda)^2}{[1+\widehat b_1(t,\lambda)]^2[1+t^2 \widehat e_3(t,\lambda)]^2}\right] dt.
\end{align*}
Now, for $d\geq 6$, integrating by parts once more yields
\begin{align*}
&\quad \O(\langle\omega\rangle^{2-d})\int_{\rho}^{\widetilde{\rho}_\lambda}t^{4-d} \partial_t\left[\frac{2 \widehat e_3(t,\lambda)+t^2 \widehat e_3(t,\lambda)^2}{[1+\widehat b_1(t,\lambda)]^2[1+\rho^2 \widehat e_3(t,\lambda)]^2}\right] dt
\\
&=\rho^{5-d}\O(\langle\omega\rangle^{2-d})\partial_\rho\widetilde{e}_0(\rho,\lambda)+\O(\langle\omega\rangle^{-2})
\\
&\quad\O(\langle\omega\rangle^{2-d})\int_{\rho}^{\widetilde{\rho}_\lambda} t^{5-d} \partial_t^2\left[\frac{2 \widehat e_3(t,\lambda)+t^2\widehat e(t,\lambda)^2}{[1+\widehat b_1(t,\lambda)]^2[1+\rho^2 \widehat e_3(t,\lambda)]^2}\right] dt.
\end{align*}
By continuing this iteratively and setting
$$
\widetilde{e}_j(\rho,\lambda)=\langle\omega\rangle^{-j} \partial_\rho^j \widetilde{e}_0(\rho,\lambda),
$$
we arrive at
\begin{align*}
I(\rho,\lambda)&=\O(\langle\omega\rangle^{2-d})\int_{\rho}^{\widetilde{\rho}_\lambda} t^{-2} \partial_t^{d-5}\left[\frac{2 \widehat e_3(t,\lambda)+t^2 \widehat e_3(t,\lambda)^2}{[1+\widehat b_1(t,\lambda)]^2[1+\rho^2 \widehat e_3(t,\lambda)]^2}\right] dt
\\
&\quad+\O(\langle\omega\rangle^{-2})+\sum_{j=0}^{d-5}\rho^{4-d+j} \O(\langle\omega\rangle^{2-d+j})\widetilde{e_j}(\rho,\lambda)
\end{align*}
where all the $\widetilde{e}_j$ satisfy estimates \eqref{esti:e0} and \eqref{esti:e}. Finally,
\begin{align*}
\int_{\rho}^{\widetilde{\rho}_\lambda} t^{-2} \partial_t^{d-5}\left[\frac{2 \widehat e_3(t,\lambda)+s^2 \widehat e(t,\lambda)^2}{[1+\widehat b_1(t,\lambda)]^2[1+\rho^2 \widehat e_3(t,\lambda)]^2}\right] dt=\O(\rho^{-1}\langle\omega\rangle^{-3})+\O(\langle\omega\rangle^{-2})=\O(\rho^{-1}\langle\omega\rangle^{-3})
\end{align*}
and we conclude that 
\begin{align*}
I(\rho,\lambda)&=\O(\rho^{-1}\langle\omega\rangle^{-3})+\sum_{j=0}^{d-5}\rho^{4-d+j} \O(\langle\omega\rangle^{2-d+j})\widetilde{e_j}(\rho,\lambda)
\end{align*}
for $\rho\in (0,\widetilde{\rho}_\lambda)$.
However, since for $|\lambda|$ large enough,
$\widetilde{\rho}_\lambda=\widehat{\rho}_\lambda$, we can safely
assume that $\widetilde{\rho}_\lambda=\widehat{\rho}_\lambda$.
\end{proof}
Lastly, one more Volterra iteration and similar considerations yield the following result.
\begin{lem}\label{lem: fundi near 1}
There exists a fundamental system
for Eq.~\eqref{eq:nofirst order} of the form
\begin{align*}
\psi_3(\rho,\lambda)&= h_1(\rho,\lambda)[1+r_1(\rho,\lambda)]
\\
\psi_4(\rho,\lambda)&=h_2(\rho,\lambda)[1+r_2(\rho,\lambda)]
\end{align*}
for all $\rho\geq \rho_\lambda$, where $r_j(\rho,\lambda)=(1-\rho)\O(\rho^0\langle\omega\rangle^{-1})$ for $j=1,2$.
\end{lem}
The final task to obtain satisfactory solutions to Eq.~\eqref{eq:nofirst order} consists of patching together the solutions constructed on the two different regimes. 
\begin{lem}\label{lem:connectioncoeff}
For $\rho\in[\rho_\lambda,\widehat{\rho}_\lambda]$ the solutions $\psi_3$ and $\psi_4$ have the representations
\begin{align*}
\psi_3(\rho,\lambda) &= c_{1,3}(\lambda)\psi_1(\rho,\lambda)+ c_{2,3}(\lambda)\psi_2(\rho,\lambda)\\
\psi_4(\rho,\lambda) &= c_{1,4}(\lambda)\psi_1(\rho,\lambda)+ c_{2,4}(\lambda)\psi_2(\rho,\lambda),
\end{align*}
with
\begin{align*}
c_{1,3}(\lambda)&= \frac{\pi W(h_1(.,\lambda),b_2(.,\lambda))(\rho_\lambda)}{2}+\O(\langle\omega\rangle^{-1})
\\
c_{2,3}(\lambda)&= -\frac{\pi W(h_1(.,\lambda),b_1(.,\lambda))(\rho_\lambda)}{2}+\O(\langle\omega\rangle^{-1})
\end{align*}
and
\begin{align*}
	c_{1,4}(\lambda)&= \frac{\pi W(h_2(.,\lambda),b_2(.,\lambda))(\rho_\lambda)}{2}+\O(\langle\omega\rangle^{-1})
	\\
	c_{2,4}(\lambda)&=-\frac{\pi W(h_2(.,\lambda),b_1(.,\lambda))(\rho_\lambda)}{2}+\O(\langle\omega\rangle^{-1}).
\end{align*}
\end{lem}
\begin{proof}
This follows as Lemma 3.4 in \cite{Wal22}
\end{proof}
Naturally, the whole construction also works for $V=0$. In this special case we assign all derived object an additional subscript ${\mathrm{f}
}$ (for instance $\psi_{\mathrm{f}_j}$ or $c_{\mathrm{f}_{1,2}})$.
We also note that $\psi_{\mathrm{f}_3}=h_1$ and $\psi_{\mathrm{f}_4}=h_2$. \\                                       
Next, consider a smooth cutoff function $\chi: [0,1]\times S\to
[0,1]$, $\chi_\lambda(\rho):=\chi(\rho,\lambda)$, that satisfies
$\chi_\lambda(\rho)=1$ for $\rho \in [0,\rho_\lambda]$, $
\chi_\lambda(\rho)=0$ for $\rho \in [\widehat{\rho}_\lambda,1] $, and
$|\partial_\rho^k\partial_\omega^\ell \chi_\lambda(\rho)|\leq
C_{k,\ell}\langle\omega\rangle^{k-\ell}$ for $k,\ell\in\mathbb N_0$. 
Then, two linearly independent solutions of Eq.~\eqref{eq:nofirst order} are given by
\begin{align*}
v_1(\rho,\lambda):=&\chi_\lambda(\rho)[c_{1,4}(\lambda)\psi_1(\rho,\lambda)+c_{2,4}(\lambda)\psi_2(\rho,\lambda)]
+\left(1-\chi_\lambda(\rho)\right)\psi_4(\rho,\lambda)\\
v_2(\rho,\lambda):=& \chi_\lambda(\rho)[c_{1,3}(\lambda)\psi_1(\rho,\lambda)+c_{2,3}(\lambda)\psi_2(\rho,\lambda)]
+\left(1-\chi_\lambda(\rho)\right)\psi_3(\rho,\lambda)
\end{align*}
for all $\rho\in (0,1)$.
Further, an evaluation at $\rho=1$ yields 
$$
W(v_1(.,\lambda),v_2(.,\lambda))=W(\psi_4(.,\lambda),\psi_3(.,\lambda))=-1.
$$
With this remark we return to the full equation~\eqref{eq:spectral equation}.
\section{The resolvent construction}
To obtain solutions to the homogeneous version of our original equation \eqref{eq:spectral equation}, we set $u_j(\rho,\lambda)=\rho^{-\frac{d-1}{2}}(1-\rho^2)^{\frac{s}{2}-\frac{3}{4}-\frac{\lambda}{2}}v_j(\rho,\lambda)$ for $j=1,2$. 
Observe (from Lemma \ref{lem: form of solutions} below), that neither of the above solutions is well behaved at $\rho=0$. To remedy this, we would like to define 
\begin{align*}
u_0(\rho,\lambda):= u_2(\rho,\lambda)-\frac{c_{2,3}(\lambda)}{c_{2,4}(\lambda)} u_1(\rho,\lambda).
\end{align*}
Of course, this only makes sense provided $c_{2,4}$ does not vanish, which necessitates the following Lemma.
\begin{lem}
Any $\lambda \in S$ is an eigenvalue of $\Lf$ if and only if $c_{2,4}(\lambda)=0$.
\end{lem}
\begin{proof}
This follows in the same way as Lemma 4.4 in \cite{DonWal22a}.
\end{proof}
To continue, we let $\varrho(\Lf)$ be the resolvent set of $\Lf$ and define $\widehat S:=S \cap \varrho(\Lf)$. Then, for any $\lambda\in \widehat{S}_\delta$, $u_0$ is well defined and satisfies
\begin{align*}
W(u_0(.,\lambda),u_1(.,\lambda))(\rho)= \rho^{1-d} (1-\rho^2)^{s-\lambda-\frac{3}{2}}.
\end{align*}
We also record the explicit forms of the solutions for convenience.
\begin{lem} \label{lem: form of solutions}
Let $\lambda\in \widehat{S}_\delta$ and set $a(\lambda)=i(s-\frac{1}{2}-\lambda)$. Then,
on the support of $\chi_\lambda$, one has that
\begin{align*}
u_0(\rho,\lambda)&= \rho^{\frac{1-d}{2}}(1-\rho^2)^{\frac{s}{2}-\frac{3}{4}-\frac{\lambda}{2}}\left[c_{1,3}(\lambda)-\frac{c_{2,3}(\lambda)}{c_{2,4}(\lambda)}c_{1,4}(\lambda)\right] \psi_1(\rho,\lambda)
\\
u_1(\rho,\lambda)&=\rho^{\frac{1-d}{2}}(1-\rho^2)^{\frac{s}{2}-\frac{3}{4}-\frac{\lambda}{2}}[c_{1,4}(\lambda)\psi_1(\rho,\lambda)+c_{2,4}(\lambda)\psi_2(\rho,\lambda)]
\\
u_2(\rho,\lambda)&=\rho^{\frac{1-d}{2}}(1-\rho^2)^{\frac{s}{2}-\frac{3}{4}-\frac{\lambda}{2}}[c_{1,3}(\lambda)\psi_1(\rho,\lambda)+c_{2,3}(\lambda)\psi_2(\rho,\lambda)]
\end{align*}
with $c_{i,j}(\lambda)=\O(\langle\omega\rangle^0)$
for $i=1,2$, $j=3,4$
and
\begin{align*}
\psi_1(\rho,\lambda)&=b_1(\rho,\lambda)[1+\rho^2 e_3(\rho,\lambda)] 
\\
&=\sqrt{(1-\rho^2)\varphi(\rho)}J_{\frac{d-2}{2}}( a(\lambda)\varphi(\rho))[1+\rho^2 e_3(\rho,\lambda)]
\\
\psi_2(\rho,\lambda)&= b_2(\rho,\lambda)[1+\rho^2 e_3(\rho,\lambda)]+ \psi_1(\rho,\lambda)e_4(\rho,\lambda).
\\
&=\sqrt{(1-\rho^2)\varphi(\rho)}Y_\frac{d-2}{2}( a(\lambda)\varphi(\rho))[1+\rho^2e_3(\rho,\lambda)]+ \psi_1(\rho,\lambda)e_4(\rho,\lambda)
\end{align*}
where $e_3$ satisfies
\begin{align*}
e_3(\rho,\lambda)=\widehat{e}_3(\rho,\lambda)+\O(\rho^{d-\frac{5}{2}}\langle\omega\rangle^{d-\frac{5}{2}})
\end{align*}
with
\begin{align}\label{esti:e12}
\partial_\rho^m\partial_\omega^n \widehat{e}_3(\rho,\lambda)\lesssim_{m,n} \langle\omega\rangle^{m-n}
\end{align}
and with
\begin{align*}
e_4(\rho,\lambda)=\O(\rho^{-1}\langle\omega\rangle^{-3})+\sum_{j=0}^{d-5}\rho^{4-d+j} \O(\langle\omega\rangle^{2-d+j})\widetilde{e_j}(\rho,\lambda)
\end{align*}
where all of the $\widetilde{e}_j$ also satisfy the estimate \eqref{esti:e12}.
Moreover, on the support of $(1-\chi_\lambda)$ one has that
\begin{align*}
u_0(\rho,\lambda)&=u_2(\rho,\lambda)-\frac{c_{2,3}(\lambda)}{c_{2,4}(\lambda)}u_1(\rho,\lambda)
\\
u_1(\rho,\lambda)&=
\rho^{\frac{1-d}{2}}(1-\rho^2)^{\frac{s}{2}-\frac{3}{4}-\frac{\lambda}{2}} \psi_4(\rho,\lambda)
\\
u_2(\rho,\lambda)&=\rho^{\frac{1-d}{2}}(1-\rho^2)^{\frac{s}{2}-\frac{3}{4}-\frac{\lambda}{2}}\psi_3(\rho,\lambda)
\end{align*}
with
\begin{align*}
\rho^{\frac{1-d}{2}}(1-\rho^2)^{\frac{s}{2}-\frac{3}{4}-\frac{\lambda}{2}}\psi_4(\rho,\lambda)&= \rho^{\frac{1-d}{2}}(1-\rho^2)^{\frac{s}{2}-\frac{3}{4}-\frac{\lambda}{2}}h_2(\rho,\lambda)[1+r_2(\rho,\lambda)]
\\
&=
\frac{\rho^{\frac{1-d}{2}}(1+\rho)^{s-\lambda-\frac{1}{2}}}{\sqrt{2s-2\lambda-1}}[1+e_2(\rho,\lambda)][1+r_2(\rho,\lambda)]
\\
\rho^{\frac{1-d}{2}}(1-\rho^2)^{\frac{s}{2}-\frac{3}{4}-\frac{\lambda}{2}}\psi_3(\rho,\lambda)&= \rho^{\frac{1-d}{2}}(1-\rho^2)^{\frac{s}{2}-\frac{3}{4}-\frac{\lambda}{2}}h_1(\rho,\lambda)[1+r_1(\rho,\lambda)]
\\
&=
\frac{ \rho^{\frac{1-d}{2}}(1-\rho)^{s-\lambda-\frac{1}{2}}}{\sqrt{2s-2\lambda-1}}[1+e_1(\rho,\lambda)][1+r_1(\rho,\lambda)]
\end{align*}
with 
$$
r_j(\rho,\lambda)=(1-\rho)\O(\rho^{0}\langle\omega\rangle^{-1}), \text{ and } e_j(\rho,\lambda)=(1-\rho)\O(\rho^{-1}\langle\omega\rangle^{-1})
$$
for $j=1,2$.
Moreover, in case $V=0$, one has that $r_1=r_2=0$.
\end{lem}

The goal, now, is to construct a solution $u(.,\lambda)\in H^k(\B^d_1)$ to the equation
\begin{equation}\label{eq:spectral equationagain}
  \begin{split}
&(\rho^2-1)u''(\rho)\left((d-2s+2\lambda+2)\rho-\frac{d-1}{\rho}\right)u'(\rho)
\\
&\quad +\frac{d-2s+2\lambda}{4}(2\lambda +d-2s+2)u(\rho)+V(\rho)u(\rho)=f(\rho),
\end{split}
\end{equation}
where $k=\ceil s$.
The first ansatz one would like to try is of course given by
\begin{align*}
u(\rho,\lambda)=-u_0(\rho,\lambda)\int_\rho^{1}\frac{u_1(t,\lambda)t^{d-1}f(t)}{(1-t^2)^{s-\lambda-\frac12}}dt-u_1(\rho,\lambda)\int_0^\rho\frac{u_0(t,\lambda)t^{d-1}f(t)}{(1-t^2 )^{s-\lambda-\frac12}} dt.
\end{align*}
However, as is easily visible, this is not a meaningful expression unless $s-\Re\lambda<\frac32$.
Thus, for $s>1$ we slightly modify the ansatz and instead start with 
\begin{align*}
u(\rho,\lambda)=c u_0(\rho,\lambda)-u_0(\rho,\lambda)\int_\rho^{\rho_1}\frac{u_1(t,\lambda)t^{d-1}f(t)}{(1-t^2)^{s-\lambda-\frac12}}dt-u_1(\rho,\lambda)\int_0^\rho\frac{u_0(t,\lambda)t^{d-1}f(t)}{(1-t^2 )^{s-\lambda-\frac12}} dt
\end{align*}
for $c\in \C$ and $\rho_1\in (0,1)$.
Let now 
\begin{align*}
U_{\ell,1}(\rho,\lambda)=\int_0^\rho \frac{u_{\ell}(t,\lambda)t^{d-1}}{(1-t^2)^{s-\lambda-\frac12}} dt
\end{align*} 
for $\ell=0,1,2$.
Integrating by parts in both integrals yields
\begin{align*}
u(\rho,\lambda)&=u_0(\rho,\lambda)\bigg[c+U_{1,1}(\rho,\lambda)f(\rho)-U_{1,1}(\rho_1,\lambda) f(\rho_1)+\int_\rho^{\rho_1} U_{1,1}(t,\lambda) f'(t) dt\bigg]
\\
&\quad+ u_1(\rho,\lambda)\bigg[-U_{0,1}(\rho,\lambda)f(\rho)+\int_0^{\rho} U_{0,1}(t,\lambda) f'(t) dt\bigg].
\end{align*}
Therefore, by setting $c=c_1+U_1(\rho_1,\lambda)f(\rho_1)$ with $c_1 \in \C$, we infer that
\begin{align*}
u(\rho,\lambda)&=u_0(\rho,\lambda)\bigg[c_1+U_{1,1}(\rho,\lambda)f(\rho)+\int_\rho^{\rho_1} U_{1,1}(t,\lambda) f'(t) dt\bigg]
\\
&\quad+ u_1(\rho,\lambda)\bigg[-U_{0,1}(\rho,\lambda)f(\rho)+\int_0^{\rho} U_{0,1}(t,\lambda) f'(t) dt\bigg].
\end{align*}
Iterating this procedure and letting $\rho_1$ tend to $1$ leads to
\begin{align*}
u(\rho,\lambda)&=u_0(\rho,\lambda)\bigg[c_k+\sum_{j=1}^{k-1} (-1)^{j+1}U_{1,j}(\rho,\lambda) f^{(j-1)}(\rho)
\\
&\quad+ (-1)^{k}\int_\rho^{1}\int_0^{t_1}\int_0^{t_2}\dots \int_0^{t_{k-1}}\frac{u_1(t_k,\lambda)t_k^{d-1}}{(1-t_k^2)^{s-\lambda-\frac12}}dt_k \dots d t_3 dt_2 f^{(k-1)}(t_1) dt_1\bigg]
\\
&\quad+u_1(\rho,\lambda)\bigg[\sum_{j=1}^{k-1}(-1)^{j} U_{0,j}(\rho,\lambda) f^{(j-1)}(\rho)
\\
&\quad+ (-1)^{k}\int_0^\rho\int_0^{t_1}\int_0^{t_2}\dots \int_0^{t_{k-1}}\frac{u_0(t_k,\lambda)t_k^{d-1}}{(1-t_k^2)^{s-\lambda-\frac12}}dt_k \dots d t_3 dt_2 f^{(k-1)}(t_1) dt_1\bigg]
\end{align*}
with $k=\ceil s$,  $c_k\in \C$ to be determined, and
\begin{align*}
U_{\ell,j}(\rho,\lambda)=\int_0^\rho\dots\int_0^{t_{j-1}} \frac{u_{\ell}(t_j,\lambda)t_j^{d-1}}{(1-t_j^2)^{s-\lambda-\frac12}} dt_j \dots dt_1
\end{align*} 
for $\ell=0,1,2$. The next three Lemmas will help us determine the right choice of $c_k$.
\begin{lem}
For $f\in C^\infty(\overline{\B^d_1})$ define $\kappa_1(f)$ as
\begin{align*}
\kappa_1(f)(\lambda):& =\frac{f(1)}{\sqrt{2s-2\lambda-1}}\int_0^1 \frac{(1-t)^{\ceil s -s- \frac12+\lambda}}{\prod_{j=1}^{\floor s}(\lambda+\frac{1}{2}+j-s)} 
\partial_t^{\ceil s-1}\left(\frac{t^{d-1}u_1(t,\lambda)}{(1+t)^{s-\lambda-\frac12}}\right) dt 
\\
&\quad+\frac{f(1)}{\sqrt{2s-2\lambda-1}}\sum_{j=1}^{\ceil s-2}\lim_{\rho \to 0}\partial_\rho^j\left(\frac{\rho^{d-1}u_1(\rho,\lambda)}{(1+\rho)^{s-\lambda-\frac12}}\right) \prod_{\ell=1}^{j+1}\frac{1}{\lambda+\frac{1}{2}+\ell-s}.
\end{align*}
Then $u_2(\rho,\lambda) U_{1,1}(\rho,\lambda) f(\rho)-u_2(\rho,\lambda)\kappa_1(f)(\lambda)\in H  ^{\ceil s}  ((0,1])$ for all $\lambda \in \widehat S$.
\end{lem}
\begin{proof}
Recall that
\begin{align*}
U_{1,1}(\rho,\lambda) &= \frac{1}{\sqrt{2s-2\lambda-1}}
\int_0^\rho \frac{ t^{d-1}u_1(t,\lambda)}{(1+t)^{s-\lambda-\frac12}}\frac{1}{(1-t)^{s-\lambda-\frac12}} dt
\end{align*}
which we integrate by parts to infer that 
\begin{align*}
U_{1,1}(\rho,\lambda) &=\frac{1}{\sqrt{2s-2\lambda-1}}
\frac{ \rho^{d-1}u_1(\rho,\lambda)}{(s-\lambda-\frac32)(1+\rho)^{s-\lambda-\frac12}}\frac{1}{(1-\rho)^{s-\lambda-\frac32}}
\\
&\quad
+\frac{1}{\sqrt{2s-2\lambda-1}}\int_0^\rho \frac{1}{(1-t)^{s-\lambda-\frac32}} 
\partial_t\left(\frac{t^{d-1}u_1(t,\lambda)}{(\lambda+\frac32-s)(1+t)^{s-\lambda-\frac12}}\right) dt 
\end{align*}
Now,
as
\begin{align*}
u_1(\rho,\lambda)&=\chi_\lambda(\rho)\rho^{\frac{1-d}{2}}(1-\rho^2)^{\frac{3}{4}-\frac \lambda 2 +\frac s 2}[ c_{1,4}(\lambda)\psi_1(\rho,\lambda)+c_{2,4}(\lambda)\psi_2(\rho,\lambda)]
\\
&\quad + (1-\chi_\lambda(\rho))
\frac{\rho^{\frac{1-d}{2}}(1+\rho)^{s-\lambda-\frac{1}{2}}}{\sqrt{2s-2\lambda-1}}[1+e_2(\rho,\lambda)][1+r_2(\rho,\lambda)]
\\
u_2(\rho,\lambda)&=\chi_\lambda(\rho)\rho^{\frac{1-d}{2}}(1-\rho^2)^{\frac{3}{4}-\frac \lambda 2 +\frac s 2}[ c_{1,3}(\lambda)\psi_1(\rho,\lambda)+c_{2,3}(\lambda)\psi_2(\rho,\lambda)]
\\
&\quad + (1-\chi_\lambda(\rho))
\frac{\rho^{\frac{1-d}{2}}(1-\rho)^{s-\lambda-\frac{1}{2}}}{\sqrt{2s-2\lambda-1}}[1+e_1(\rho,\lambda)][1+r_1(\rho,\lambda)]
\end{align*}
with $\psi_j$, $e_j$, and $r_j$ from Lemmas \ref{lem: form of solutions} and \ref{lem: fundi near 1}. Now, one readily checks that
\begin{align*}
\frac{u_2(\rho,\lambda)}{(1-\rho)^{s-\lambda-\frac{1}{2}}}\in C^k((0,1]).
\end{align*}
Therefore, we iterate this scheme to conclude that
\begin{align*}
u_2(\rho,\lambda)U_{1,1}(\rho,\lambda) &=
\frac{u_2(\rho,\lambda)}{\sqrt{2s-2\lambda-1}}\int_0^\rho \frac{(1-t)^{\ceil s-s-\frac12+\lambda}}{\prod_{j=1}^{\ceil s-1}(\lambda+\frac{1}{2}+j-s)} 
\partial_t^{\ceil s-1}\left(\frac{t^{d-1}u_1(t,\lambda)}{(1+t)^{s-\lambda-\frac12}}\right) dt 
\\
&\quad +\frac{u_2(\rho,\lambda)}{\sqrt{2s-2\lambda-1}}\sum_{j=1}^{\ceil s-2}\lim_{\rho \to 0}\partial_\rho^j\left(\frac{\rho^{d-1}u_1(\rho,\lambda)}{(1+\rho)^{s-\lambda-\frac12}}\right) \prod_{\ell=1}^{j+1}\frac{1}{\lambda+\frac{1}{2}+\ell-s}
\\
&\quad + r(\rho,\lambda)
\end{align*}
where $r(.,\lambda)$ is smooth at $\rho=1$ for all $\lambda\in \widehat{S}_\delta$.
From this, one can easily infer that
\begin{align*}
&u_2(\rho,\lambda) U_{1,1}(\rho,\lambda) f(\rho)- u_2(\rho,\lambda)\kappa_1(f)(\lambda)
\end{align*}
is an element of $C^{\ceil s}((0,1])$, by using a simple scaling argument.
\end{proof}
\begin{lem}
For $f\in C^\infty(\overline{\B^d_1})$ define $\kappa_2(f)$ as
\begin{align*}
\kappa_2(f)(\lambda)&:= \frac{f'(1)}{\sqrt{2s-2\lambda-1}}
\int_0^1 \int_0^{t_1} \frac{(1-t_2)^{\ceil s-s- \frac12+\lambda}}{\prod_{j=1}^{\ceil s-1}( \lambda+\frac{1}{2}+j-s)} 
\partial_{t_1}^{\ceil s-1}\left(\frac{t_2^{d-1}u_1(t_2,\lambda)}{(1+t_2)^{s-\lambda-\frac12}}\right) dt_1 dt_2
\\
&\quad -\frac{f'(1)}{\sqrt{2s-2\lambda-1}}(\ceil s-1)
\int_0^1 \frac{(1-t)^{\ceil s-s -\frac12+\lambda}}{\prod_{j=1}^{\ceil s-1}( \lambda+\frac{1}{2}+j-s)} 
\partial_t^{\ceil s-2}\left(\frac{t^{d-1}u_1(t,\lambda)}{(1+t)^{s-\lambda-\frac12}}\right) dt 
\\
&\quad + \frac{f'(1)}{\sqrt{2s-2\lambda-1}}\sum_{j=1}^{\ceil s-2}(j-1)\lim_{\rho \to 0}\partial_\rho^j\left(\frac{\rho^{d-1}u_1(\rho,\lambda)}{(1+\rho)^{s-\lambda-\frac12}}\right) \prod_{\ell=1}^{j+2}\frac{1}{\lambda+\frac{1}{2}+\ell-s}
\end{align*}
Then $u_2(\rho,\lambda) U_{1,2}(\rho,\lambda) f'(\rho)-u_2(\rho,\lambda)\kappa_2(f)  \in H  ^{\ceil s}  ((0,1])$ for all $\lambda \in \widehat S$.
\end{lem}

\begin{proof}
Since 
\begin{align*}
U_{1,2}(\rho,\lambda) &=\frac{1}{\sqrt{2s-2\lambda-1}}
\int_0^\rho \int_0^{t_1}\frac{ t_2^{\frac{d-1}{2}}u_1(t_2,\lambda)}{(1+t_2)^{s-\lambda-\frac12}}\frac{1}{(1-t_2)^{s-\lambda-\frac12}} dt_2 dt_1,
\end{align*}
we integrate by parts a number of times to arrive at
\begin{align*}
&\quad \int_0^\rho \int_0^{t_1}\frac{ t_2^{\frac{d-1}{2}}u_1(t_2,\lambda)}{(1+t_2)^{s-\lambda-\frac12}}\frac{1}{(1-t_2)^{s-\lambda-\frac12}} dt_2 dt_1
\\
&=
\int_0^\rho \int_0^{t_1} \frac{(1-t_2)^{\ceil s-s-\frac12+\lambda}}{\prod_{j=1}^{\ceil s-1}( \lambda+\frac{1}{2}+j-s)} 
\partial_{t_2}^{\ceil s -1}\left(\frac{t_2^{d-1}u_1(t_1,\lambda)}{(1+t_2)^{s-\lambda-\frac12}}\right) dt_2 dt_1
\\
&
\quad-\sum_{j=0}^{\ceil s-2} \int_0^\rho \partial_t^j\left( \frac{ t^{\frac{d-1}{2}}u_1(t,\lambda)}{(1+t)^{s-\lambda-\frac12}}\right)\frac{1}{\prod_{\ell=1}^{j+1}( \lambda+\frac{1}{2}+\ell-s)(1-t)^{s-\lambda-\frac32-j}} dt
\\
&
\quad+\sum_{j=1}^{\ceil s-2} \int_0^\rho \lim_{t\to 0} \partial_t^j\left( \frac{ t^{\frac{d-1}{2}}u_1(t,\lambda)}{(1+t)^{s-\lambda-\frac12}}\right)\frac{1}{\prod_{\ell=1}^{j+1}( \lambda+\frac{1}{2}+\ell-s)(1-t)^{s-\lambda-\frac32-j}} dt.
\end{align*}
Further,
\begin{align*}
&\quad-u_2(\rho,\lambda)\int_0^\rho \partial_\rho^j\left( \frac{ t^{\frac{d-1}{2}}u_1(t,\lambda)}{(1+t)^{s-\lambda-\frac12}}\right)\frac{1}{\prod_{\ell=1}^{j+1}( \lambda+\frac{1}{2}+\ell-s)(1-t)^{s-\lambda-\frac32-j}} dt
\\
&= 
-u_2(\rho,\lambda)
  \int_0^\rho \frac{(1-t)^{\ceil s-s -\frac12+\lambda}}{\prod_{j=1}^{\ceil s}( s-j-\lambda-\frac{1}{2})} 
\partial_t^{\ceil s-2}\left(\frac{t^{\frac{d-1}{2}}u_1(t,\lambda)}{(1+t)^{s-\lambda-\frac12}}\right) dt 
\\
&\quad -u_2(\rho,\lambda)\sum_{\ell=j}^{\ceil s -3}\lim_{\rho \to 0}\partial_\rho^\ell\left(\frac{\rho^{d-1}u_1(\rho,\lambda)}{(1+\rho)^{s-\lambda-\frac12}}\right) \prod_{m=1}^{\ell+2}\frac{1}{\lambda+\frac{1}{2}+m-s}+ r(\rho,\lambda)
\end{align*}
where $r$ is smooth at $\rho=1$.
Consequently, by choosing $\kappa_2(f)$ as stated in the Lemma the desired conclusion follows.
\end{proof}
Of course, the same procedure can be applied to all of the $U_{1,j}$.
\begin{lem}
For $f\in C^\infty(\overline{\B^d_1})$ and $1\leq j\leq \ceil s-1$ define $\kappa_j(f)$ as
\begin{align*}
\kappa_j(f)(\lambda)&:= 
f^{(j-1)}(1)\sum_{\ell=1}^j a_{j,\ell}
\int_0^1 \dots \int_0^s \frac{(1-t_{\ell})^{\floor s-s+\frac12+\lambda}}{\prod_{j=1}^{k}( k-j-\lambda-\frac{1}{2})} 
\\
&\quad \times \partial_{t_{\ell}}^{k-j+\ell}\left(\frac{t_{\ell}^{\frac{d-1}{2}}u_1(t_{\ell},\lambda)}{(1+t_{\ell})^{s-\lambda-\frac12}}\right) dt_{\ell} dt_{\ell-1}\dots dt_1
\\
&\quad
+\sum_{\ell=1}^{\floor s-j}b_{j,\ell}\lim_{\rho \to 0}\partial_\rho^\ell\left(\frac{u_1(\rho,\lambda)}{(1+\rho)^{s-\lambda-\frac12}}\right) \prod_{m=1}^{\ell+j}\frac{1}{\lambda+\frac{1}{2}+m-s}
\end{align*}
for appropriately chosen constants $a_{j,\ell}, b_{j,\ell} \in\mathbb{Z}$.
Then $$u_2(\rho,\lambda) U_{1,j}(\rho,\lambda) f^{(j)}(\rho)-u_2(\rho,\lambda)\kappa_j(f) H  ^{\ceil s}  ((0,1])$$ for all $\lambda \in \widehat S$.
Furthermore, for $\Re \lambda>\frac{1}{2}+s-\ceil s$, we can rewrite $\kappa_{\ceil s-1}$
 as 
\begin{align*}
\kappa_{\ceil s-1}(\lambda)= f^{(\ceil s -2)}(1)U_{1,\ceil s-1}(1,\lambda).
\end{align*}
\end{lem}
\begin{lem}\label{lem:reslambda}
Let $f\in C^\infty(\overline{\B^d_1})$ and $\lambda\in \widehat S$. Then, the unique solution $\mathcal{R}(f)(.,\lambda)\in H^{\ceil s}(\B^{d}_1)$ of the equation
\begin{equation}\label{eq:specodeagain}
  \begin{split}
&(\rho^2-1)u''(\rho)\left((d-2s+2\lambda+2)\rho-\frac{d-1}{\rho}\right)u'(\rho)
\\
&\quad +\frac{d-2s+2\lambda}{4}(2\lambda +d-2s+2)u(\rho)+V(\rho)u(\rho)=f(\rho)
\end{split}
\end{equation}
 is given by
\begin{align*}
\mathcal{R}(f)(\rho,\lambda)&:=u_0(\rho,\lambda)\bigg[\kappa(f)(\lambda)+\sum_{j=1}^{k-1} (-1)^{j+1}U_{1,j}(\rho,\lambda) f^{(j-1)}(\rho)
\\
&\quad+ (-1)^{k}\int_\rho^{1}\int_0^{t_1}\int_0^{t_2}\dots \int_0^{t_{k-1}}\frac{u_1(t_k,\lambda)t_k^{d-1}}{(1-t_k^2)^{s-\lambda-\frac12}}dt_k \dots d t_3 dt_2 f^{(k-1)}(t_1) dt_1\bigg]
\\
&\quad+u_1(\rho,\lambda)\bigg[\sum_{j=1}^{k-1}(-1)^j U_{0,j}(\rho,\lambda) f^{(j-1)}(\rho)
\\
&\quad+ (-1)^{k}\int_0^\rho\int_0^{t_1}\int_0^{t_2}\dots \int_0^{t_{k-1}}\frac{u_0(t_k,\lambda)t_k^{d-1}}{(1-t_k^2)^{s-\lambda-\frac12}}dt_k \dots d t_3 dt_2 f^{(k-1)}(t_1) dt_1\bigg]
\end{align*}
where $k=\ceil s$, $\kappa(f):=\sum_{j=1}^{k-1} (-1)^{j+1}\kappa_j(f)$, and
\begin{align*}
U_{\ell,j}(\rho,\lambda)=\int_0^\rho\dots\int_0^{t_{j-1}} \frac{u_{\ell}(t_j,\lambda)t_j^{d-1}}{(1-t_j^2)^{s-\lambda-\frac12}} dt_j \dots dt_1
\end{align*}
for $\ell=0,1,2$.
\end{lem}
\begin{proof}
Clearly $\mathcal{R}(f)\in C^k(\B^d_1\setminus \{0\})$ and so we only have to care about the endpoints.
We start with the endpoint $\rho=0$. By plugging in the explicit form of $u_0$ and $u_1$ near $\rho=0$ one easily concludes that
\begin{align*}
u_0(\rho,\lambda) U_{1,j}(\rho,\lambda)\in H^k(\B^d_{\frac{1}{2}})
\end{align*}
and
\begin{align*}
u_1(\rho,\lambda) U_{0,j}(\rho,\lambda)\in H^k(\B^d_{\frac{1}{2}})
\end{align*}
for all $1\leq j\leq k-1$.
Similarly, one infers that the remaining integral terms are fine at $\rho=0$, as well. Thus, we turn to the other endpoint $\rho=1$  and rewrite
\begin{align*}
-u_0(\rho,\lambda)U_{1,j}(\rho,\lambda) 
+u_1(\rho,\lambda)U_{0,j}(\rho,\lambda) = 
-u_2(\rho,\lambda)U_{1,j}(\rho,\lambda) 
+u_1(\rho,\lambda)U_{2,j}(\rho,\lambda).
\end{align*}
Then, by our choice of $\kappa(f)$ we infer that
\begin{align*}
u_2(\rho,\lambda)\left[\kappa(f)(\lambda)+\sum_{j=1}^{k-1} (-1)^{j+1}U_{1,j}(\rho,\lambda) f^{(j-1)}(\rho)\right]\in H^k((0,1]).
\end{align*}
Further, one easily computes that
$u_1(\rho,\lambda)U_{2,j}(\rho,\lambda) \in H^k((0,1])$.
Thus, we turn to the remaining integral terms which we rewrite as
\begin{align*}
&u_0(\rho,\lambda)(-1)^{k-1}\int_\rho^{1}\int_0^{t_1}\int_0^{t_2}\dots \int_0^{t_{k-1}}\frac{u_1(t_k,\lambda)t_k^{d-1}}{(1-t_k^2)^{s-\lambda-\frac12}}dt_k \dots d t_3 dt_2 f^{(k-1)}(t_1) dt_1
\\
&+u_1(\rho,\lambda) (-1)^{k-1}\int_0^\rho\int_0^{t_1}\int_0^{t_2}\dots \int_0^{t_{k-1}}\frac{u_0(t_k,\lambda)t_k^{d-1}}{(1-t_k^2)^{s-\lambda-\frac12}}dt_k \dots d t_3 dt_2 f^{(k-1)}(t_1) dt_1
\\
&=u_2(\rho,\lambda)(-1)^{k-1}\int_\rho^{1}\int_0^{t_1}\int_0^{t_2}\dots \int_0^{t_{k-1}}\frac{u_1(t_k,\lambda)t_k^{d-1}}{(1-t_k^2)^{s-\lambda-\frac12}}dt_k \dots d t_3 dt_2 f^{(k-1)}(t_1) dt_1
\\
&+u_1(\rho,\lambda) (-1)^{k-1}\int_0^\rho\int_0^{t_1}\int_0^{t_2}\dots \int_0^{t_{k-1}}\frac{u_2(t_k,\lambda)t_k^{d-1}}{(1-t_k^2)^{s-\lambda-\frac12}}dt_k \dots d t_3 dt_2 f^{(k-1)}(t_1) dt_1
\\
&-(-1)^{k-1}\frac{c_{2,3}(\lambda)}{c_{2,4}(\lambda)}u_1(\rho,\lambda) \int_0^1\int_0^{t_1}\int_0^{t_2}\dots \int_0^{t_{k-1}}\frac{u_1(t_k,\lambda)t_k^{d-1}}{(1-t_k^2)^{s-\lambda-\frac12}}dt_k \dots d t_3 dt_2 f^{(k-1)}(t_1) dt_1.
\end{align*}
In this form, on easily concludes that all of the above terms are $k$ times continuously differentiable at $\rho=1$ which establishes the assertion.
\end{proof}
For later reference, we denote the solution $\mathcal{R}$ by $\mathcal{R}_{\mathrm{f}} $ in the ``free'' case $V=0$. 
Furthermore, for $s \notin \mathbb{N}$ and $\Re \lambda>-\frac12+s-\floor s=-\frac{1}{2}+\theta$, we can recast $\mathcal{R}$ as follows.
\begin{lem}
Let $s \notin \mathbb{N}$ and $\lambda\in \widehat S$ be such that 
$$
\frac12+s-\ceil s<\Re \lambda.$$ Then, $\mathcal{R}$ can be recast as
\begin{align*}
\mathcal{R}(f)(\rho,\lambda)&=u_0(\rho,\lambda)\bigg[\kappa(f)(\lambda)+\sum_{j=1}^{k-2} (-1)^{j+1}U_{1,j}(\rho,\lambda) f^{(j-1)}(\rho)
\\
&\quad+ (-1)^{k-1}\int_\rho^{1}\int_0^{t_1}\int_0^{t_2}\dots \int_0^{t_{k-2}}\frac{u_1(t_{k-1},\lambda)t_{k-1}^{d-1}}{(1-t_{k-1}^2)^{s-\lambda-\frac12}}dt_{k-1} \dots d t_3 dt_2 f^{(k-2)}(t_1) dt_1\bigg]
\\
&\quad+u_1(\rho,\lambda)\bigg[\sum_{j=1}^{k-2}(-1)^j U_{0,j}(\rho,\lambda) f^{(j-1)}(\rho)
\\
&\quad+ (-1)^{k-1}\int_0^\rho\int_0^{t_1}\int_0^{t_2}\dots \int_0^{t_{k-2}}\frac{u_0(t_{k-1},\lambda)t_{k-1}^{d-1}}{(1-t_{k-1}^2)^{s-\lambda-\frac12}}dt_{k-1} \dots d t_3 dt_2 f^{(k-2)}(t_1) dt_1\bigg]
\end{align*}
where $k=\ceil s$ and $ \widehat \kappa(f):=\sum_{j=1}^{k-2} (-1)^j\kappa_j(f)$.
\end{lem}
\begin{proof}
This follows immediately by undoing one integration by parts in both integral terms.
\end{proof}
We also need one more Lemma in case $s\in \mathbb{N}$. 
\begin{lem}\label{lem:extendres}
Let $s\in \mathbb{N}$ and $\lambda\in \widehat{S}$. Then,
$\mathcal{R}(.,\lambda)\in H^{s+\frac{1}{100}}(\B^d_1)$. Furthermore, it is the unique solution of Eq.~ \eqref{eq:specodeagain} of that regularity.
\end{lem}
The proof of this Lemma can be found in Appendix B.\\
Before we can start deriving estimates on the semigroup, we describe the ranges of the involved parameter which we will need to use. We start with the easiest case, which is given in case $s\in \mathbb{N}$.
\begin{rem}\hfill
\begin{itemize}
\item $s \in \mathbb{N}:$\\
If $s$ happens to be an integer, we set the interpolation parameter $\theta=\frac12$ and pick $\delta <<1$ such that no eigenvalues lie in the strips $\{z\in \C: -\delta \leq \Re z <0\}$, $\{z\in \C: 0< \Re z \leq \delta\}$ and set $\mu_0=\delta$ and $\mu_1=-\delta $. Then, evidently, $(1-\theta)\mu_0+\theta \mu_1=0$.

\item $ s\notin \mathbb{N}$:\\
Since we also aim at deriving estimates at a fractional regularity level, we are going to establish estimates for the regularities $\ceil s$ and $\floor s$ and interpolate between these. For this we start by setting $\theta=  s-\floor s$, which implies that $\ceil s \theta +\floor s (1-\theta)=s.$ 
Furthermore, recall the estimate
\begin{align*}
\|\Sf_0(\tau)\|\lesssim_\varepsilon e^{(\varepsilon-\ceil s+s)\tau}=e^{(-1+\theta +\varepsilon)\tau}
\end{align*}
for any $\varepsilon>0$. This implies that the essential spectrum of $\Lf$ is contained in the set $\{z\in \C:\Re z \leq -1+\theta\}$ and we can push the contour integral \eqref{semirep} as close to the line $\Re z=-1+\theta$ as we like. Additionally, we can choose $1>>\delta>0$ such that no eigenvalues lie on the two lines $\{z\in \C: \Re z=-1+\theta+\tfrac{\delta}{\theta}\}$ and $\{z\in \C: \Re z=\theta-\tfrac{\delta}{1-\theta}\}$. We set $\mu_0:=\theta-\tfrac{\delta}{1-\theta}$ and $\mu_1 := -(1-\theta) +\frac\delta \theta$. We note there is hope that these are suitable for an interpolation argument, as $\mu_0$ and $\mu_1 $ are chosen  such that 
$$
\frac12+s-\ceil s<\mu_0<s-\floor s,$$
$$
s-\ceil s<\mu_1<s-\floor s ,
$$
and it holds that $$ (1-\theta)\mu_0+\mu_1\theta=0.$$

\end{itemize}
\end{rem}
Now, we define the set of unstable eigenvalues $\sigma_u(\Lf)$ as
\begin{align*}
\sigma_u(\Lf)=\{\lambda\in \sigma_p(\Lf): \Re \lambda \geq 0\}.
\end{align*}
Let $\Pf$ be the bounded finite rank spectral projection onto $\sigma_u (\Lf)$, $\Qf$ be the finite rank spectral projection on the remaining isolated eigenvalues that satisfy $\Re \lambda_i> \mu_1$, and set $(\I-\Qf)(\I-\Pf) \ff=\widetilde{\ff}$. Then,
we can explicitly write down the first component of our semigroup $\Sf(\I-\Qf)(\I-\Pf)$ for all $\ff \in C^\infty_{rad}\times C^\infty_{rad}(\overline{\B^d_1})$ as 
\begin{align}\label{semirep}
[\Sf(\tau) \widetilde{\ff}]_1(\rho)=[\Sf_0(\tau) \widetilde{\ff}]_1(\rho)
+\frac{1}{2\pi i}\lim_{N \to \infty} \int_{\mu_{a}-i N}^{\mu_{a}+ i N}e^{\lambda\tau}[\mathcal{R}(F_\lambda)(\rho,\lambda)-\mathcal{R}_{\mathrm{f}}(F_\lambda)(\rho,\lambda)] d\lambda,
\end{align}
with $a=0,1$ and $F_\lambda(\rho)= (\lambda +\frac{d}{2}-s+1)\widetilde f_1(\rho)+\rho \partial_\rho \widetilde f_1(\rho)+
\widetilde f_2(\rho)$, since $[\mathbf R_{\mathbf L}(\lambda)\widetilde
{\mathbf f}]_1=\mathcal R(F_\lambda)(.,\lambda)$ and $[\mathbf R_{\mathbf L_0}(\lambda)\widetilde
{\mathbf f}]_1=\mathcal R_{\mathrm f}(F_\lambda)(.,\lambda)$.
Hence, we are tasked with estimating the above integral term. 
For that, we recall some technical Lemmas and prove some ourselves.
\subsection{Preliminary and technical lemmas}
\begin{lem}\label{osci1}
	Let $\alpha >0$. Then 
	$$
	\left|\int_\R e^{i \omega a}\O(\langle\omega\rangle^{-(1+\alpha)}) d \omega \right|\lesssim \langle a\rangle^{-2},
	$$
	 for any $a\in \R$. 
\end{lem}

\begin{lem}\label{osci2}
	Let $\alpha \in (0,1)$. Then
	\begin{align*}
	\left|\int_\R e^{i \omega a}\O(\langle\omega\rangle^{-\alpha}) d \omega \right|\lesssim |a|^{\alpha-1}\langle a\rangle^{-2}
	\end{align*}
	holds for $a\in \R\setminus\{0\}$.
\end{lem}
\begin{proof}
	See Lemma 4.2 in \cite{DonRao20}.
\end{proof}
\begin{lem}\label{osci3}
We have
	\begin{align*}
	\left|\int_\R e^{i \omega a}(1-\chi_\lambda(\rho))\O\left(\rho^{-n}\langle\omega\rangle^{-(n+1)}\right) d \omega \right|\lesssim \langle a\rangle^{-2}
	\end{align*}
	for all $n\geq 1$, $\rho \in (0,1)$, and $a\in \R$.
\end{lem}
\begin{proof}
	This can be proven in the same manner as Lemma 4.3 in \cite{DonRao20}.
\end{proof}

\begin{lem}\label{osci4}
We have
	\begin{align*}
	\left|\int_\R e^{i \omega a}(1-\chi_\lambda(\rho))\O\left(\rho^{-n}\langle\omega\rangle^{-n}\right)d\omega\right|\lesssim|a|^{-1}\langle a\rangle^{-2}
	\end{align*}
	for any $n \geq 2$, $\rho \in (0,1)$, and $a\in \R \setminus\{0\}$.
\end{lem}
\begin{proof}
This can be proven as Lemma 4.4 in \cite{DonRao20}.
\end{proof}
Furthermore, by interpolating between Lemma \ref{osci3} and \ref{osci4} one obtains the following result.
\begin{lem}\label{osci5}
We have
	\begin{align*}
	\left|\int_\R e^{i \omega a}(1-\chi_\lambda(\rho))\O\left(\rho^{-n}\langle\omega\rangle^{-n}\right)d\omega\right|\lesssim\rho^{-\theta}|a|^{-(1-\theta)}\langle a\rangle^{-2}
	\end{align*}
	for any $n \geq 2$, $\rho \in (0,1)$, $\theta\in [0,1]$, and $a\in \R \setminus\{0\}$.
\end{lem}
Moving on, as $(1-\chi_\lambda(\rho))\O(\rho^{-\alpha}\langle\omega\rangle^{-\beta})=(1-\chi_\lambda(\rho))\O(\rho^{0}\langle\omega\rangle^{-\beta+\alpha})$ for all $\alpha,\beta>0$, one also readily obtains the next result.
\begin{lem}\label{osci6}
Let $0<\alpha<\beta<\infty$ be two fixed numbers. Then
\begin{align*}
	\left|\int_\R e^{i \omega a}(1-\chi_\lambda(\rho))\O\left(\rho^{-\alpha}\langle\omega\rangle^{-1-\beta}\right)d\omega\right|\lesssim\langle a\rangle^{-2}
	\end{align*}
	for all $a\in \R$.
	\end{lem}
Finally, our last lemma on oscillatory integrals reads as follows.
\begin{lem}\label{osci7}
Let $0<\alpha<1$ and $c>1$. Then, the estimate
\begin{align*}
\left|\rho^{-c} \int_\R e^{i\omega a} (1-\chi_\lambda(\rho)) \O(\langle\omega\rangle^{-c-\alpha}) d\omega \right| &\lesssim |a|^{-1+\alpha} \langle a\rangle^{-2}
\end{align*}
holds for all $a\in \R \setminus 0$.
\end{lem}
\begin{proof}
Without loss of generality, we assume $a>0$. First, let $a \in (0,1)$ and
let $kappa$ be smooth cutoff with support on $[-2,2]$ which equals $1$ on $[-1,1]$. Then,
\begin{align*}
\rho^{-c} \int_\R e^{i\omega a} (1-\chi_\lambda(\rho)) \O(\langle\omega\rangle^{-c-\alpha}) d\omega &=\rho^{-c} \int_\R \kappa(a\omega)e^{i\omega a}  (1-\chi_\lambda(\rho)) \O(\langle\omega\rangle^{-c-\alpha}) d\omega
\\
&\quad +\rho^{-c} \int_\R e^{i\omega a} (1-\kappa(a\omega))(1-\chi_\lambda(\rho)) \O(\langle\omega\rangle^{-c-\alpha}) d\omega
\\
&=:I(\rho,a)+I_2(\rho,a).
\end{align*}
Then,
\begin{align*}
|I_1 (\rho,a) |&= \left|\int_\R \kappa(a\omega)e^{i\omega a}  (1-\chi_\lambda(\rho)) \O(\rho^0\langle\omega\rangle^{-\alpha}) d\omega\right|
\\
&= a^{-1} \left| \int_\R \kappa(\omega)e^{i\omega}  (1-\chi_\lambda(a^{-1}\rho)) \O(\rho^{0} a^{\alpha}\omega^{-\alpha}) d\omega\right|
\lesssim a^{\alpha-1}.
\end{align*}
Further, 
\begin{align*}
I_2(\rho,a)&= \rho^{-c} a^{-1} \int_\R e^{i\omega a} (1-\kappa(a\omega))(1-\chi_\lambda(\rho)) \O(\langle\omega\rangle^{-1-c-\alpha}) d\omega
\\
&\quad + \rho^{-c} \int_\R e^{i\omega a}\kappa'(a\omega)(1-\chi_\lambda(\rho)) \O(\langle\omega\rangle^{-c-\alpha}) d\omega
\\
&\quad + \rho^{-c} a^{-1} \int_\R e^{i\omega a}(1-\kappa(a\omega))\partial_\omega\chi_\lambda(\rho) \O(\langle\omega\rangle^{-c-\alpha}) d\omega
\\
& =:I_{2,1}(\rho,a)+I_{2,2}(\rho,a)+I_{2,3}(\rho,a)
\end{align*}
and one computes that
\begin{align*}
|I_{2,1}(\rho,a)|&=\left|a^{-1} \int_\R e^{i\omega a} (1-\kappa(a\omega))(1-\chi_\lambda(\rho)) \O(\rho^0\langle\omega\rangle^{-1-\alpha}) d\omega\right|
\\
&=\left|a^{-2} \int_\R e^{i\omega } (1-\kappa(\omega))(1-\chi_\lambda(a^{-1}\rho)) \O(\rho^0\omega^{-1-\alpha}a^{1+\alpha}) d\omega\right|
\\
&\lesssim a^{\alpha-1},
\end{align*}   
\begin{align*}
|I_{2,2}(\rho,a)|&=\left|a^{-1} \int_\R e^{i\omega} (1-\kappa(\omega))(1-\chi_{\lambda a^{-1}}(\rho)) \O(\rho^{-c}\langle\omega a^{-1}\rangle^{-c-\alpha}) d\omega\right|
\\
&=\left|a^{-1} \int_\R e^{i\omega} (1-\kappa(\omega))(1-\chi_{\lambda a^{-1}}(\rho)) \O(\rho^{0}|\omega a^{-1}|^{-\alpha}) d\omega\right|
\\
&\lesssim a^{\alpha-1},
\end{align*}  
and
\begin{align*}
|I_{2,3}(\rho,a)|&=\left|a^{-1} \int_\R e^{i\omega a} (1-\kappa(a\omega))(1-\chi_\lambda(\rho)) \O(\rho^0\langle\omega\rangle^{-1-\alpha}) d\omega\right|
\\
&=\left|a^{-1} \int_\R e^{i\omega a}(1-\kappa(a\omega))\partial_\omega\chi_\lambda(\rho) \O(\rho^0\langle\omega\rangle^{-\alpha}) d\omega\right|
\\
&=\left|a^{-2} \int_\R e^{i\omega}(1-\kappa(\omega))\partial_{a^{-1}\omega}\chi_{a^{-1}\lambda}(\rho) \O(\rho^0\langle\omega a^{-1}\rangle^{-\alpha}) d\omega\right|
\\
&\lesssim\left|a^{-2} \int_\R (1-\kappa(\omega))|\omega a^{-1}|^{-1-\alpha} d\omega\right|
\\
&\lesssim a^{\alpha-1}.
\end{align*}  
\end{proof}	
We will also rely on the following technical result.
\begin{lem}\label{teclem4}
Let $\alpha\in (0,1)$ and $\beta \in [0,1)$. Then we have the estimate
	\begin{align*}
	\int_0^1 t^{-\beta}|a+\log(1\pm s)|^{-\alpha} dt\lesssim |a|^{-\alpha}
	\end{align*}
	for all $a\in \R\setminus\{0\}$.
\end{lem}
\begin{proof}
	We only prove the - case, as the + case can be shown analogously. 
For $a<0$ the estimate
\begin{align*}
\left|a+ \log(1-t) \right|^{-\alpha} \leq |a|^{-\alpha}
\end{align*}  
holds for all $s\in [0,1]$ and so the claim follows.
For $a>0$  we change variables according to $s=1-e^{ax}$ and compute
	\begin{align*}
	\int_0^1s^{-\beta}|a+\log(1-t)|^{-\alpha} dt &= \int_{-\infty}^0(1-e^{a x})^{-\beta}|a+ax|^{-\alpha} a e^{a x} dx
	\\
	&\lesssim |a|^{1-\alpha}\int_{-\frac{1}{2}}^0(1-e^{a x})^{-\beta}e^{ax} dx 
	\\
	&\quad +|a|^{1-\alpha}(1-e^{-\frac a2})^{-\beta}e^{-\frac a2}\int_{-2}^{-\frac{1}{2}} |1+x|^{-\alpha}dx 
	\\
	&\quad +|a|^{1-\alpha}\int^{-2}_{-\infty}(1-e^{a x})^{-\beta}e^{a x} dx.
	\end{align*}
	The claimed estimate is now an immediate consequence of the two identities
\[
\partial_x \frac{(1-e^{ax})^{1-\beta}}{a(1-\beta)}=(1-e^{ax})^{-\beta}e^{ax} 
\]
and
\[(1-e^{-\frac{a}{2}})^{-\beta}e^{-\frac{a}{2}} \lesssim a^{-\beta}.
\]
\end{proof}
Similarly, one can show the next technical Lemma 
\begin{lem} \label{teclem5}
Let $\alpha\in (0,1)$ and $\beta \in [0,1)$. Then the estimate
	\begin{align*}
	\int_0^1 s^{-\beta}\left|a\pm \frac{1}{2}\log(1-t^2) \right|^{-\alpha} dt\lesssim |a|^{-\alpha}
	\end{align*}
	holds for all $a\in \R\setminus\{0\}$.
\end{lem}
Lastly, we will also require results on weighted norms.
For this purpose we first borrow the following two results
\begin{lem}[Item 1 of Theorem 2.3 in \cite{DjaEdeOli11}]\label{simple char}
Let $d\in \mathbb{N}$ be fixed. Then
\begin{align*}
\sum_{j=0}^k \int_0^1 |f^{(j)}(\rho)|^p \rho^{d-1} d\rho\lesssim \|f\|_{W^{k,p}(\B^d_1)}^p
\end{align*}
for all $f\in C^\infty(\overline{\B^d_1})$ and all $p>1$ and $k\in \mathbb{N}$.
\end{lem}
\begin{lem}[Item 1 of Lemma 3 in \cite{Ost22a}]\label{teclem1}
Let $p,r,s \in \R$ with $p\geq 1$, $r>0$ and $s>-\frac1p$.
The estimate 
$$
\int_0^r \rho^{ps}|f(\rho)|^p d\rho\lesssim |f(r)|^p+\int_0^r \rho^{p(s+1)}|f'(\rho)|^p d\rho
$$
holds for all $f\in C^1([0,r])$.
\end{lem}
\begin{lem}\label{teclem2}
Let  $p>1$, $k \in \mathbb N$ with $\frac{d-1}p +\frac{1}{p}> k\geq 2$,  and $\varepsilon>0$ be fixed number. Then, the estimates
\begin{align*}
\| |.|^{j+\frac{d-1}{p}+\frac1p-k+\varepsilon}f^{(j-1)}\|_{L^\infty(\B^d_1)}\lesssim \|f\|_{W^{k-1,p}(\B^d_1)}
\end{align*}
and
\begin{align*}
\| |.|^{j+\frac{d-1}{p}+\frac1p-1-k+\varepsilon}f^{(j-1)}\|_{L^\infty(\B^d_1)}\lesssim \|f\|_{W^{k,p}(\B^d_1)}
\end{align*}
holds for all $f\in C_{rad}^\infty(\overline{\B^d_1})$ and all $1\leq j\leq k-1$.
\end{lem}
\begin{proof}
By the one dimensional Sobolev embedding one has that
\begin{align*}
\|  |.|^{j+\frac{d-1}{p}+\frac1p -k+\varepsilon}f^{(j-1)}\|_{L^\infty(\B^d_1)}&\lesssim_\delta \| |.|^{j+\frac{d-1}{p}+\frac1p-k+\varepsilon}f^{(j-1)}\|_{W^{1,1+\delta}((0,1))}
\end{align*}
for any $\delta >0$.
Moreover, by Hölder's inequality
\begin{align*}
&\quad\| |.|^{j+\frac{d-1}{p}+\frac1p-k+\varepsilon}f^{(j-1)}\|_{W^{1,1+\delta}((0,1))}
\\
&\lesssim \| |.|^{j-1+\frac{d-1}{p}+\frac1p-k+\varepsilon}f^{(j-1)}\|_{L^{1+\delta}((0,1))}+\| |.|^{j+\frac{d-1}{p}+\frac1p-k+\varepsilon}f^{(j)}\|_{L^{1+\delta}((0,1))}
\\
&\leq \| |.|^{-\frac{p-1}{p}+\varepsilon} \|_{L^{\frac{p+p\delta}{p-(1+\delta)}}((0,1))}
\\
&\quad \times \left(\| |.|^{j+\frac{d-1}{p}-k}f^{(j-1)}\|_{L^p((0,1))}+\| |.|^{j+\frac{d-1}{p}+1-k}f^{(j)}\|_{L^p((0,1))}\right).
\end{align*}
Since an appropriate choice of $\delta>0$ ensures that 
\begin{align*}
\| |.|^{-\frac{p-1}{p}+\varepsilon} \|_{L^{\frac{p+p\delta}{p-(1+\delta)}}((0,1))}\lesssim 1,
\end{align*}
we obtain that 
\begin{align*}
\|  |.|^{j+\frac{d-1}{p}+\frac1p -k+\varepsilon}f^{(j-1)}\|_{L^\infty(\B^d_1)}&\lesssim_\varepsilon \| |.|^{j+\frac{d-1}{p}+1-k}f^{(j-1)}\|_{L^p((0,1))}
 \\
 &\quad+\| |.|^{j+\frac{d-1}{p}+1-k}f^{(j)}\|_{L^p((0,1))}.
\end{align*}
Now, a repeated application of Lemma \ref{teclem1} shows 
\begin{align*}
&\quad \| |.|^{j+\frac{d-1}{p}-k}f^{(j-1)}\|_{L^p((0,1))}^p+\| |.|^{j+\frac{d-1}{p}+1-k}f^{(j)}\|_{L^p((0,1))}^p 
\\&=\int_0^1 \rho^{p(j+\frac{d-1}{p}+2-k)}|f^{(j+1)}(\rho)|^p d\rho +\int_0^1 \rho^{p(j+1\frac{d-1}{p}+2-k)}|f^{(j)}(\rho)|^p d\rho
\\
&\lesssim \sum_{\ell=j}^{k-2} |f^{(\ell)}(1)|+ \int_0^1 |f^{(k-1)}(\rho)|^p \rho^{d-1} d\rho \lesssim \|f\|_{W^{k-1,p}(\B^d_1)}^p
\end{align*}
where the last inequality follows form \ref{simple char} and the estimate $|f(1)|\lesssim \|f\|_{H^1(\B^d_1)}$.
\end{proof}
\begin{lem}\label{teclem3}
Let $d \in \mathbb{N}$, $k\in \mathbb{N}$ with $ k \leq \frac{d}{2}$, and $r<1$ satisfy $|r-\frac12|<\varepsilon$ with $\varepsilon$ sufficiently small.
Then, provided that $dr +n-k-r \neq 0$, the estimates
\begin{align*}
\||.|^{-n+j+1-\varepsilon}f^{(j-1)}\|_{L^{\frac{d}{dr+n-k-r}}(\B^d_1)}\lesssim \|f\|_{W^{k-1,\frac{1}{r}}(\B^d_1)}
\end{align*}
and
\begin{align*}
\||.|^{-n+j-\varepsilon}f^{(j-1)}\|_{L^{\frac{d}{dr+n-k-r}}(\B^d_1)}\lesssim \|f\|_{W^{k,\frac{1}{r}}(\B^d_1)}
\end{align*}
hold for all $n,j \in \mathbb{N}_0$ with $1\leq j \leq k-1 $, $0\leq n\leq k-1$ and $n+j\leq k-1$.
\end{lem}
\begin{proof}
Set $m:=\lfloor \frac{p}{pr-1}\rfloor $ with $p=\frac{d}{dr+n-k-r}$ and note that an appropriately small choice of $\varepsilon$ ensures that $m\geq 2$.
Then $m$ is chosen such that the Sobolev inequality
\begin{align*}
\|.\|_{L^{p}(\B^m_1)} \lesssim \|.\|_{W^{1,\frac{1}{r}}(\B^m_1)}
\end{align*} 
holds. So,
\begin{align*}
\||.|^{-n+j+1-\varepsilon}f^{(j-1)}\|_{L^{\frac{d}{dr+n-k-r}}(\B^d_1)}&=\||.|^{-n+j+1-\varepsilon+\frac{d-m}{p}}f^{(j-1)}\|_{L^{\frac{d}{dr+n-k-r}}(\B^m_1)}
\\
&\lesssim \||.|^{-n+j+1-\varepsilon+\frac{d-m}{p}}f^{(j-1)}\|_{W^{1,\frac 1r}(\B^m_1)}\\
&\lesssim \||.|^{-n+j+1-\varepsilon+\frac{d-m}{p}}f^{(j)}\|_{W^{1,\frac 1r}(\B^m_1)}
\\
&\quad+\||.|^{-n+j-\varepsilon+\frac{d-m}{p}}f^{(j-1)}\|_{W^{1,\frac 1r}(\B^m_1)}.
\end{align*}
As we would like to invoke Lemma \ref{teclem1}, we first need to establish that
\begin{align*}
(-n+j-\varepsilon+\frac{d-m}{p})\frac{1}{r}+m-1 >-1.
\end{align*}
To check this, we note that, by assumption, $r$ and $p$ satisfy
\begin{align*}
r=\frac{1}{2}+\varepsilon_r
\end{align*}
\begin{align*}
p=\frac{d}{dr-k+n-r}=\frac{d}{\frac{d}{2}-k+n-\frac{1}{2}+(d-1)\varepsilon_r}
\end{align*}
for some small (not necessary positive) $\varepsilon_r$. 
Hence,
\begin{align*}
-n+j-\varepsilon+\frac{d}{p}= d-k+j-\frac12+\widetilde \varepsilon \geq \frac12+\widetilde\varepsilon
\end{align*}
for some $ \widetilde \varepsilon$ that is very small in absolute value. 
Plugging this in yields
\begin{align*}
(-n+j-\varepsilon+\frac{d-m}{p})\frac{1}{r}+m-1&\geq m-\frac12 -\frac{m}{rp}+\widehat \varepsilon>-1
 \end{align*}
where $\widehat \varepsilon$ is again very small in absolute value (and tends to $0$ as $\varepsilon$ does so). Thus, provided $\varepsilon$ is chosen small enough, one can apply Lemma \ref{teclem1} several times to arrive at 
\begin{align*}
\||.|^{-n+j+1-\varepsilon}f^{(j-1)}\|_{L^{\frac{d}{dr+n-k-r}}(\B^d_1)}^{\frac{1}{r}}&\lesssim \sum_{\ell=j-1}^{k-2} |f^{(\ell)}(1)|^{\frac{1}{r}}+ \int_0^1  \rho^{a}|f^{(k-1)}(\rho)|^{\frac{1}{r}} d\rho
\end{align*}
with
\begin{align*}
a=(-n-\varepsilon+k+\frac{d-m}{p})\frac{1}{r}+m-1=d-1+m-1-\frac{m}{pr}-\varepsilon.
\end{align*}
Now, by construction, $m=2$ implies $p\geq 6$ which directly yields
$$
m-1-\frac{m}{pr}-\varepsilon\geq \frac{1}{3}-\varepsilon>0.
$$
Similarly, for $m\geq 3$ we have that
\begin{align*}
m-1-\frac{m}{pr}-\varepsilon\geq 2-\frac{3}{2}-\varepsilon>0
\end{align*}
and the first of the claimed estimates follows. For the second one, we can use the same arguments, provided that 
\begin{align}\label{eq: needs to hold}
(-n+j-1-\varepsilon+\frac{d-m}{p})\frac{1}{r}+m-1 >-1.
\end{align} 
For this, we compute that
\begin{align*}
(-n+j-1-\varepsilon+\frac{d-m}{p})\frac{1}{r}+m-1 \geq m-2-\varepsilon-\frac{m}{rp}. 
\end{align*}
So, one readily infers the validity of \eqref{eq: needs to hold} by using previously employed arguments and we therefore conclude this proof. 
\end{proof}

Finally, the last technical result, which is in the same spirit as the one above, is the following.
\begin{lem}\label{teclem6}
Let $d \in \mathbb{N}$, $k\in \mathbb{N}$ with $ k \leq \frac{d}{2}$, and $r\leq\frac 12$ satisfy $|r-\frac12|<\varepsilon$ with $\varepsilon$ sufficiently small.
Then, provided that $dr +n-k+1-r \neq 0$, the estimates
\begin{align*}
\||.|^{-n+j+1-\varepsilon}f^{(j-1)}\|_{L^{\frac{d}{dr+n-k+1-r}}(\B^d_1)}\lesssim \|f\|_{W^{k-2,\frac{1}{r}}(\B^d_1)}
\end{align*}
and
\begin{align*}
\||.|^{-n+j-\varepsilon}f^{(j-1)}\|_{L^{\frac{d}{dr+n-k+1-r}}(\B^d_1)}\lesssim \|f\|_{W^{k-1,\frac{1}{r}}(\B^d_1)}
\end{align*}
hold for all $n,j \in \mathbb{N}_0$ with $1\leq j \leq k-2 $, $0\leq n\leq k-2$ and $n+j\leq k-2$.
\end{lem}
\begin{proof}
The proof is essentially the same as the one of Lemma \ref{teclem5}.
\end{proof}

Recall, that our aim is to derive estimates of the form
\begin{align}
\left\|\lim_{N \to \infty} \int_{\varepsilon-i N}^{\varepsilon+ i N}e^{\lambda\tau}[\mathcal{R}(F_\lambda)(\rho,\lambda)-\mathcal{R}_{\mathrm{f}}(F_\lambda)(\rho,\lambda)] d\lambda \right\|_{L^p(\R_+)\dot{W}^{n,p}(\B^d_1)}\lesssim \|\ff\|_{\mathcal{H}}
\end{align}
for appropriate choices of $p,q\in \R$ and $n\in \mathbb{Z}$ with $0\leq n\leq k$.
For this, it is necessary to take a closer look at $\Rm(f)$.
\begin{lem}
One has that
\begin{align*}
\partial_\rho\mathcal{R}(f)(\rho,\lambda)&=\partial_\rho u_0(\rho,\lambda)\bigg[\kappa(f)(\lambda)+\sum_{j=1}^{k-1} (-1)^{j+1}U_{1,j}(\rho,\lambda) f^{(j-1)}(\rho)
\\
&\quad+ (-1)^{k}\int_\rho^{1}\int_0^{t_1}\int_0^{t_2}\dots \int_0^{t_{k-1}}\frac{u_1(t_k,\lambda)t_k^{d-1}}{(1-t_k^2)^{s-\lambda-\frac12}}dt_k \dots d t_3 dt_2 f^{(k-1)}(t_1) dt_1\bigg]
\\
&\quad+\partial_\rho u_1(\rho,\lambda)\bigg[\sum_{j=1}^{k-1}(-1)^j U_{0,j}(\rho,\lambda) f^{(j-1)}(\rho)
\\
&\quad+ (-1)^{k}\int_0^\rho\int_0^{t_1}\int_0^{t_2}\dots \int_0^{t_{k-1}}\frac{u_0(t_k,\lambda)t_k^{d-1}}{(1-t_k^2)^{s-\lambda-\frac12}}dt_k \dots d t_3 dt_2 f^{(k-1)}(t_1) dt_1\bigg] 
\end{align*}
with $k=\ceil s$.
\end{lem}
\begin{proof}
Observe that \begin{align*}
\left[\partial_\rho \sum_{j=1}^{k-1} (-1)^{j+1}U_{1,j}(\rho,\lambda) f^{(j-1)}(\rho)\right]=f(\rho)\partial_\rho U_{1,1}(\rho,\lambda)+ (-1)^k U_{1,k-1}(\rho,\lambda)f^{(k-1)}(\rho)
\end{align*}
as well as 
\begin{align*}
&\quad\partial_\rho(-1)^{k}\int_\rho^{1}\int_0^{t_1}\int_0^{t_2}\dots \int_0^{t_{k-1}}\frac{u_1(t_k,\lambda)t_k^{d-1}}{(1-t_k^2)^{s-\lambda-\frac12}}dt_k \dots d t_3 dt_2 f^{(k-1)}(t_1) dt_1
\\
&=-(-1)^k U_{1,k-1}(\rho,\lambda)f^{(k-1)}(\rho).
\end{align*}
Similarly,
 \begin{align*}
\left[\partial_\rho \sum_{j=1}^{k-1} (-1)^{j}U_{0,j}(\rho,\lambda) f^{(j-1)}(\rho)\right]=f(\rho)\partial_\rho U_{0,1}(\rho,\lambda)+ (-1)^{k-1} U_{0,k-1}(\rho,\lambda)f^{(k-1)}(\rho)
\end{align*}
as well as 
\begin{align*}
&\quad\partial_\rho(-1)^{k}\int_0^\rho\int_0^{t_1}\int_0^{t_2}\dots \int_0^{t_{k-1}}\frac{u_0(t_k,\lambda)t_k^{d-1}}{(1-t_k^2)^{s-\lambda-\frac12}}dt_k \dots d t_3 dt_2 f^{(k-1)}(t_1) dt_1
\\
&=(-1)^k U_{0,k-1}(\rho,\lambda)f^{(k-1)}(\rho).
\end{align*}
Therefore, 
\begin{align*}
\partial_\rho\mathcal{R}(f)(\rho,\lambda)&=\partial_\rho u_0(\rho,\lambda)\bigg[\kappa(f)(\lambda)+\sum_{j=1}^{k-1} (-1)^{j+1}U_{1,j}(\rho,\lambda) f^{(j-1)}(\rho)
\\
&\quad+ (-1)^{k}\int_\rho^{1}\int_0^{t_1}\int_0^{t_2}\dots \int_0^{t_{k-1}}\frac{u_1(t_k,\lambda)t_k^{d-1}}{(1-t_k^2)^{s-\lambda-\frac12}}dt_k \dots d t_3 dt_2 f^{(k-1)}(t_1) dt_1\bigg]
\\
&\quad+\partial_\rho u_1(\rho,\lambda)\bigg[\sum_{j=1}^{k}(-1)^j U_{0,j}(\rho,\lambda) f^{(j-1)}(\rho)
\\
&\quad+ (-1)^{k}\int_0^\rho\int_0^{t_1}\int_0^{t_2}\dots \int_0^{t_{k-1}}\frac{u_0(t_k,\lambda)t_k^{d-1}}{(1-t_k^2)^{s-\lambda-\frac12}}dt_k \dots d t_3 dt_2 f^{(k-1)}(t_1) dt_1\bigg] 
\\
&\quad +u_0(\rho,\lambda)\partial_\rho U_{1,1}(\rho,\lambda) f(\rho)-u_1(\rho,\lambda)\partial_\rho U_{0,1}(\rho,\lambda) f(\rho)
\end{align*}
and since
\begin{align*}
+u_0(\rho,\lambda)\partial_\rho U_{1,1}(\rho,\lambda) -u_1(\rho,\lambda)\partial_\rho U_{0,1}(\rho,\lambda) =0
\end{align*}
the claim follows.
\end{proof} 
Moreover, by the same logic one derives an analogous result for higher derivative of $\Rm(f)$.
\begin{lem}
One has that
\begin{align*}
\partial_\rho^n \mathcal{R}(f)(\rho,\lambda)&=\partial_\rho^n u_0(\rho,\lambda)\bigg[\kappa(f)(\lambda)+\sum_{j=n}^{k-1} (-1)^{j+1}U_{1,j}(\rho,\lambda) f^{(j-1)}(\rho)
\\
&\quad+ (-1)^{k}\int_\rho^{1}\int_0^{t_1}\int_0^{t_2}\dots \int_0^{t_{k-1}}\frac{u_1(t_k,\lambda)t_k^{d-1}}{(1-t_k^2)^{s-\lambda-\frac12}}dt_k \dots d t_3 dt_2 f^{(k-1)}(t_1) dt_1\bigg]
\\
&\quad+\partial_\rho^n u_1(\rho,\lambda)\bigg[\sum_{j=n}^{k-1}(-1)^j U_{0,j}(\rho,\lambda) f^{(j-1)}(\rho)
\\
&\quad+ (-1)^{k}\int_0^\rho\int_0^{t_1}\int_0^{t_2}\dots \int_0^{t_{k-1}}\frac{u_0(t_k,\lambda)t_k^{d-1}}{(1-t_k^2)^{s-\lambda-\frac12}}dt_k \dots d t_3 dt_2 f^{(k-1)}(t_1) dt_1\bigg] 
\\
&\quad+\sum_{j=1}^{n-1} (-1)^{j+1}
\partial_\rho^{n-j-1}\left[\partial_\rho \left(\partial_\rho^j u_0(\rho,\lambda)U_{1,j}(\rho,\lambda)\right)f^{(j-1)
}(\rho)\right]
\\
&\quad+\sum_{j=1}^{n-1} (-1)^{j}
\partial_\rho^{n-j-1}\left[\partial_\rho \left(\partial_\rho^j u_1(\rho,\lambda)U_{0,j}(\rho,\lambda)\right)f^{(j-1)
}(\rho)\right]
\end{align*}
for all $2\leq n\leq k:=\ceil s$.
\end{lem}
The analogue of course holds for our second representation of $\Rm$.
\begin{lem}
Let $\Re =\mu_0$. Then,
one has that
\begin{align*}
\partial_\rho\mathcal{R}(f)(\rho,\lambda)&=\partial_\rho u_0(\rho,\lambda)\bigg[\widehat \kappa(f)(\lambda)+\sum_{j=1}^{k-2} (-1)^{j+1}U_{1,j}(\rho,\lambda) f^{(j-1)}(\rho)
\\
&\quad+ (-1)^{k-1}\int_\rho^{1}\int_0^{t_1}\int_0^{t_2}\dots \int_0^{t_{k-2}}\frac{u_1(t_{k-1},\lambda)t_{k-1}^{d-1}}{(1-t_{k-1}^2)^{s-\lambda-\frac12}}dt_{k-1} \dots d t_3 dt_2 f^{(k-2)}(t_1) dt_1\bigg]
\\
&\quad+\partial_\rho u_1(\rho,\lambda)\bigg[\sum_{j=1}^{k-2}(-1)^j U_{0,j}(\rho,\lambda) f^{(j-1)}(\rho)
\\
&\quad+ (-1)^{k-1}\int_0^\rho\int_0^{t_1}\int_0^{t_2}\dots \int_0^{t_{k-2}}\frac{u_0(t_{k-1},\lambda)t_{k-1}^{d-1}}{(1-t_{k-1}^2)^{s-\lambda-\frac12}}dt_{k-1} \dots d t_3 dt_2 f^{(k-2)}(t_1) dt_1\bigg] 
\end{align*}
and
\begin{align*}
\partial_\rho^n \mathcal{R}(f)(\rho,\lambda)&=\partial_\rho^n u_0(\rho,\lambda)\bigg[\widehat \kappa(f)(\lambda)+\sum_{j=n}^{k-2} (-1)^{j+1}U_{1,j}(\rho,\lambda) f^{(j-1)}(\rho)
\\
&\quad+ (-1)^{k-1}\int_\rho^{1}\int_0^{t_1}\int_0^{t_2}\dots \int_0^{t_{k-2}}\frac{u_1(t_{k-1},\lambda)t_{k-1}^{d-1}}{(1-t_{k-1}^2)^{s-\lambda-\frac12}}dt_{k-1} \dots d t_3 dt_2 f^{(k-2)}(t_1) dt_1\bigg]
\\
&\quad+\partial_\rho^n u_1(\rho,\lambda)\bigg[\sum_{j=n}^{k-2}(-1)^j U_{0,j}(\rho,\lambda) f^{(j-1)}(\rho)
\\
&\quad+ (-1)^{k-1}\int_0^\rho\int_0^{t_1}\int_0^{t_2}\dots \int_0^{t_{k-2}}\frac{u_0(t_{k-1},\lambda)t_{k-1}^{d-1}}{(1-t_{k-1}^2)^{s-\lambda-\frac12}}dt_{k-1} \dots d t_3 dt_2 f^{(k-1)}(t_1) dt_1\bigg] 
\\
&\quad+\sum_{j=1}^{n-1} (-1)^{j+1}
\partial_\rho^{n-j-1}\left[\partial_\rho \left(\partial_\rho^j u_0(\rho,\lambda)U_{1,j}(\rho,\lambda)\right)f^{(j-1)
}(\rho)\right]
\\
&\quad+\sum_{j=1}^{n-1} (-1)^{j}
\partial_\rho^{n-j-1}\left[\partial_\rho \left(\partial_\rho^j u_1(\rho,\lambda)U_{0,j}(\rho,\lambda)\right)f^{(j-1)
}(\rho)\right]
\end{align*}
for all $2\leq n\leq k-1=\ceil s-1$.
\end{lem}
Next, we take a closer look at the function $\kappa(f)$.
\begin{lem} \label{lem:kappa}
The functions $\kappa_j(f)(\lambda)$ and $\kappa_{\mathrm{f}_j}(f)(\lambda) $ satisfy
\begin{align*}
\kappa_j(f)(\lambda)= f^{(j-1)}(1) \O(\langle\omega\rangle^{-\frac{d}{2}-1})
\end{align*}
and
\begin{align*}
\kappa_j(f)(\lambda)-\kappa_{\mathrm{f}_j}(f)(\lambda)=f^{(j-1)}(1)\O(\langle\omega\rangle^{-\frac{d}{2}-2})
\end{align*}
for all $0\leq j\leq k-1$.
\end{lem}
\begin{proof}
Recall that
\begin{align*}
\kappa_1(f)(\lambda)&=\frac{f(1)}{\sqrt{2s-2\lambda-1}}\int_0^1 \frac{(1-t)^{\ceil s -s- \frac12+\lambda}}{\prod_{j=1}^{\floor s}(\lambda+\frac{1}{2}+j-s)} 
\partial_t^{\ceil s-1}\left(\frac{t^{d-1}u_1(t,\lambda)}{(1+t)^{s-\lambda-\frac12}}\right) dt 
\\
&\quad+\frac{f(1)}{\sqrt{2s-2\lambda-1}}\sum_{j=1}^{\ceil s-2}\lim_{\rho \to 0}\partial_\rho^j\left(\frac{\rho^{d-1}u_1(\rho,\lambda)}{(1+\rho)^{s-\lambda-\frac12}}\right) \prod_{\ell=1}^{j+1}\frac{1}{\lambda+\frac{1}{2}+\ell-s}.
\end{align*}
We decompose $\kappa_1(f)$ as
\begin{align*}
\kappa_1(f)(\lambda)&=\frac{f(1)}{\sqrt{2s-2\lambda-1}}\int_0^1 \frac{\chi_\lambda(t)(1-t)^{\ceil s -s- \frac12+\lambda}}{\prod_{j=1}^{\floor s}(\lambda+\frac{1}{2}+j-s)} 
\\
&\quad \times
\partial_t^{\ceil s-1}\left(\frac{t^{\frac{d-1}{2}}[c_{1,4}(\lambda)\psi_1(t,\lambda)+c_{2,4}(\lambda)\psi_2(t,\lambda)]}{(1+t)^{s-\lambda-\frac12}}\right) dt
\\
&\quad+
\frac{f(1)}{\sqrt{2s-2\lambda-1}}\int_0^1 \frac{(1-\chi_\lambda(t))(1-t)^{\ceil s -s- \frac12+\lambda}}{\prod_{j=1}^{\floor s}(\lambda+\frac{1}{2}+j-s)}
\\
&\quad \times \partial_t^{\ceil s-1} \left(t^{\frac{d-1}{2}}[1+e_2(t,\lambda)][1+r_2(t,\lambda)]\right) dt
\\
&\quad+ \frac{f(1)}{\sqrt{2s-2\lambda-1}}\sum_{j=1}^{\ceil s-2}\lim_{\rho \to 0}\partial_\rho^j\left(\frac{\rho^{d-1}u_1(\rho,\lambda)}{(1+\rho)^{s-\lambda-\frac12}}\right) \prod_{\ell=1}^{j+1}\frac{1}{\lambda+\frac{1}{2}+\ell-s}.
\end{align*}
Now, as $\rho^{d-2}u_1(\rho,\lambda)\in C^{\ceil{ \frac d2} }([0,1])$, one makes use of Lemma \ref{lem: form of solutions} and readily computes that
\begin{align*}
&\quad\frac{1}{\sqrt{2s-2\lambda-1}}\sum_{j=1}^{\ceil s-2}\lim_{\rho \to 0}\partial_\rho^j\left(\frac{\rho^{d-1}u_1(\rho,\lambda)}{(1+\rho)^{s-\lambda-\frac12}}\right) \prod_{\ell=1}^{j+1}\frac{1}{\lambda+\frac{1}{2}+\ell-s}
\\
&=\sum_{j=1}^{\ceil s -2}j(d-1)\lim_{\rho \to 0}\partial_\rho^{j-1}\left(\frac{\rho^{d-2}u_1(\rho,\lambda)}{(1+\rho)^{s-\lambda-\frac12}}\right) \prod_{\ell=1}^{j+1}\frac{1}{\lambda+\frac{1}{2}+\ell-s}=\O(\langle\omega\rangle^{-\frac{d}{2}-1}).
\end{align*}
So, we only need to investigate the integral terms. For the first of the above integral terms, we use the identity 
$$
\chi_\lambda(\rho)\O(\rho^\alpha\langle\omega\rangle^\beta)=
\chi_\lambda(\rho)\O(\rho^{\alpha-\gamma}\langle\omega\rangle^{\beta-\gamma})$$
which holds for all $\alpha,\beta \in \R$ and all $\gamma\geq0$ 
to conclude that 
\begin{align*}
 \frac{\chi_\lambda(t)(1-t)^{\ceil s -s- \frac12+\lambda}}{\prod_{j=1}^{\floor s}(\lambda+\frac{1}{2}+j-s)} 
\partial_t^{\ceil s -1}\left(\frac{t^{\frac{d-1}{2}}[c_{1,4}(\lambda)\psi_1(t,\lambda)+c_{2,4}(\lambda)\psi_2(t,\lambda)]}{(1+t)^{k-\lambda-\frac12}}\right) =\chi_\lambda(t)\O(t^{0}\langle\omega\rangle^{-\frac{d}{2}}).
\end{align*} 
For the latter, repeated integrations by parts show that
\begin{align*}
&\quad\int_{0}^1
 \frac{(1-\chi_\lambda(t))(1-t)^{\ceil s -s- \frac12+\lambda}}{\prod_{j=1}^{\floor s}(\lambda+\frac{1}{2}+j-s)} 
\partial_t^{\ceil s -1}\left(t^{\frac{d-1}{2}}\frac{[1+e_2(t,\lambda)][1+r_2(t,\lambda)]}{\sqrt{2s-2\lambda-1}}\right) dt
\\
&=
\int_{0}^1
(1-\chi_\lambda(t))(1-t)^{\ceil s -s- \frac12+\lambda}
\O(t^{\frac{d+1}{2}-\ceil s }\langle\omega\rangle^{-\ceil s +\frac{1}{2}}) dt
\\
&=
\int_{0}^1
(1-t)^{\frac{d}{2} -s+ \frac32+\lambda}
\partial_t^{\frac{d}{2}-\ceil s +2}[(1-\chi_\lambda(t)) \O(t^{\frac{d+1}{2}-\ceil s }\langle\omega\rangle^{-\frac{d}{2} -\frac{3}{2}}) ]dt
=
\O(\langle\omega\rangle^{-\frac{d}{2}-1}).
\end{align*}
Hence,
\begin{align*}
\kappa_1(f)=f(1)\O(\langle\omega\rangle^{-\frac{d}{2}-1})
\end{align*}
and similarly, one establishes $\kappa_j(f)=f^{(j-1)}(1)\O(\langle\omega\rangle^{-\frac{d}{2}-1})$ for $j\geq 1$. Finally, to estimate the differences of $\kappa_j(f)$ and $\kappa_{\mathrm{f}_j}(\lambda)$, we record the following identities:
\begin{align*}
c_{i,j}(\lambda)-c_{\mathrm{f}_{i,j}}(\lambda)= \O(\langle\omega\rangle^{-1})
\\
\psi_1(\rho,\lambda)-\psi_{\mathrm{f}_1}(\rho,\lambda)= \O(\rho^{\frac{d+3}{2}}\langle\omega\rangle^{\frac{d-2}{2}}).
\end{align*}
Further, analogous identities hold for the remaining differences. Therefore, subtracting $\kappa_{\mathrm{f}_j}$  from $\kappa_j$ gains us one additional order of decay in $\omega$ and the claim follows.
\end{proof}

With this result, we end this section and move on to the oscillatory integrals.
\section{Strichartz estimates}
To start this section, we remark that a short inspection shows 
$\mathcal{R}(f)(\rho,\mu_a+i\omega)-\mathcal{R}_{\mathrm{f}}(f)(\rho,\mu_a+i \omega)$ decays like $\langle\omega\rangle^{-2}$, for $a=0,1$. Hence, by dominated convergence, we obtain that
\begin{align*}
&\quad\lim_{N \to \infty} \int_{\mu_a-i N}^{\mu_a+i N}e^{\lambda \tau}[\mathcal{R}(f)(\rho,\lambda)-\mathcal{R}_{\mathrm{f}}(f)(\rho,\lambda)] d\lambda 
\\
&=i e^{\mu_a\tau} \int_\R e^{i \omega \tau}[\mathcal{R}(f)(\rho,\mu_a+ i \omega)-\mathcal{R}_{\mathrm{f}}(f)(\rho,\mu_a+i \omega)] d\omega
\end{align*}
for $a=0,1$. However, we also need to move $\rho$ derivatives into the oscillatory integral, i.e. we are required to establish identities of the form
\begin{align*}
&\quad\partial_\rho^n \lim_{N \to \infty} \int_{\mu_a-i N}^{\mu_a+i N}e^{\lambda \tau}[\mathcal{R}(f)(\rho,\lambda)-\mathcal{R}_{\mathrm{f}}(f)(\rho,\lambda)] d\lambda 
\\
&=i e^{\mu_a\tau} \int_\R e^{i \omega \tau}\partial_\rho^n[\mathcal{R}(f)(\rho,\mu_a+ i \omega)-\mathcal{R}_{\mathrm{f}}(f)(\rho,\mu_a+i \omega)] d\omega.
\end{align*}

\begin{lem}
The identities
\begin{align*}
&\quad \partial_\rho^n\lim_{N \to \infty} \int_{\mu_a-i N}^{\mu_a+ i N}e^{\lambda\tau}[\mathcal{R}(f)(\rho,\lambda)-\mathcal{R}_{\mathrm{f}}(f)(\rho,\lambda)] d\lambda
\\
&=ie^{\mu_a \tau}\int_\R e^{i \omega\tau}\partial_\rho^n [\mathcal{R}(f)(\rho,\mu_a+i\omega)-\mathcal{R}_{\mathrm{f}}(f)(\rho,\mu_a+i\omega)] d\omega
\end{align*}
and
\begin{align*}
&\quad\partial_\rho^n \lim_{N \to \infty} \int_{\mu_a-i N}^{\mu_a+ i N}\lambda e^{\lambda\tau}[\mathcal{R}(f)(\rho,\lambda)-\mathcal{R}_{\mathrm{f}}(f)(\rho,\lambda)] d\lambda
\\
&=ie^{\mu_a \tau}\int_\R (\mu_a+i\omega) e^{i \omega \tau}\partial_\rho^n  [\mathcal{R}(f)(\rho,\mu_a+i\omega)-\mathcal{R}_{\mathrm{f}}(f)(\rho,\mu_a+i\omega)] d\omega
\end{align*}
hold for all integers $n$ with $0\leq n\leq k=\ceil s$, $a=0,1$, $\tau\geq 0$, $\rho\in (0,1)$, and $f\in C^\infty_{rad}(\overline{\B^d_1})$.
\end{lem}
We remark in passing, that the need to investigate terms of the form  $\omega \partial_\rho^{n}\Rm(f)(\rho,\lambda)$ arises from the simple observation that $F_\lambda$ also contains the term $\lambda f_1$.
\begin{proof}
To show this Lemma, we first take a look at 
\begin{align*}
\kappa(f)(\lambda) u_0(\rho,\lambda)-\kappa_{\mathrm{f}}(f)(\lambda) u_{\mathrm{f}_0}(\rho,\lambda)
\end{align*}
and claim that 
\begin{align*}
&\quad \partial_\rho^n  \lim_{N \to \infty} \int_{\mu_a-i N}^{\mu_a+ i N}\lambda e^{\lambda\tau}[\kappa(f)(\lambda) u_0(\rho,\lambda)-\kappa_{\mathrm{f}}(f)(\lambda) u_{\mathrm{f}_0}(\rho,\lambda)] d\lambda
\\
&=ie^{\mu_a \tau}\int_\R (\mu_a+i\omega) e^{i \omega \tau}\partial_\rho^n  [\kappa(f)(\mu_a+ i\omega) u_0(\rho,\mu_a+i \omega)-\kappa_{\mathrm{f}}(f)(\mu_a+i \omega) u_{\mathrm{f}_0}(\rho,\mu_a+i\omega )] d\omega
\end{align*}
for all stated $n,j,\tau,\rho,$ and all $f\in C^\infty_{rad}(\overline{\B^d_1})$.
Indeed, by Lemma \ref{lem:kappa} the difference $ \kappa(f)(\lambda) u_0(\rho,\lambda)-\kappa_{\mathrm{f}}(f)(\lambda) u_{\mathrm{f}_0}(\rho,\lambda)$ decays of order $\frac{d+5}{2}$ in $\omega$ and since differentiating with respect to $\rho$ lowers that decay by 1 degree, the claim follows thanks to dominated convergence.
Next, we aim to show that 
\begin{equation}\label{eq:interchange 1}
\begin{split}
&\quad\partial_\rho^n\lim_{\varepsilon\to 0} \lim_{N \to \infty} \int_{\mu_a-i N}^{\mu_a+ i N}\lambda e^{\lambda\tau}f(\rho)[u_0(\rho,\lambda)U_1(\rho,\lambda)-u_{\mathrm{f}_0}(\rho,\lambda)U_{\mathrm{f}_1}(\rho,\lambda)]d\lambda
\\
&=ie^{\mu_a \tau}\int_\R (\mu_a+i\omega) e^{i \omega \tau}
\\
&\quad \times\partial_\rho^n \left(f(\rho)[u_0(\rho,\mu_a+i\omega )U_1(\rho,\mu_a+i\omega )-u_{\mathrm{f}_0}(\rho,\mu_a+i\omega)U_{\mathrm{f}_1}(\rho,\mu_a+i\omega)]\right)d\omega.
\end{split}
\end{equation}

For this, we rewrite $u_0 U_1  -u_1 U_0$ as $$ u_0 U_1  -u_1 U_0 =u_2 U_1 - u_1 U_2 $$ and decompose $ u_2 U_1$ as
\begin{align*}
u_2(\rho,\lambda) U_1(\rho,\lambda)&=\chi_\lambda(\rho)[c_{1,3}(\lambda)\psi_1(\rho,\lambda)+c_{1,4}(\lambda) \psi_2(\rho,\lambda)] U_1(\rho,\lambda) 
\\
&\quad+\rho^{\frac{1-d}{2}}(1-\chi_\lambda(\rho))(1-\rho)^{s-\frac{1}{2}-\lambda}[1+e_1(\rho,\lambda)][1+r_1(\rho,\lambda)]
\\
&\quad \times \int_0^\rho \frac{(1-\chi_\lambda(t)) t^{\frac{d-1}{2}}[1+e_2(t,\lambda)][1+r_2(t,\lambda)]}{(s-2\lambda-1)(1-t)^{s-\lambda-\frac12}} dt
\\
&\quad
+ u_2(\rho,\lambda) \int_0^\rho \chi_\lambda(t)t^{\frac{d-1}{2}}\frac{\left[c_{2,3}(\lambda)\psi_1(t,\lambda)+c_{2,4}(\lambda)\psi_1(t,\lambda) \right]}{(1-t^2)^{s-\lambda-\frac12}}
dt. 
\end{align*}
Now, on the support of $\chi_\lambda(\rho)$, we can exchange powers of $\rho$ for decay in $\omega$. Thus, the interchanging and limiting operations can all be done for this term. Analogously, also in the second of the above integral terms, we can (after subtracting the corresponding free analogue) exchange enough powers of $t$ for decay in $\omega$ such that we are allowed to move up to $\ceil s$ $\rho$-derivatives inside the integral.
Further, an integration by parts shows
\begin{align*}
&\quad \rho^{\frac{1-d}{2}}(1-\chi_\lambda(\rho))(1-\rho)^{s-\frac{1}{2}-\lambda}[1+e_1(\rho,\lambda)][1+r_1(\rho,\lambda)]
\\
&\quad \times \int_0^\rho \frac{(1-\chi_\lambda(t)) t^{\frac{d-1}{2}}[1+e_2(t,\lambda)][1+r_2(t,\lambda)]}{(2s-2\lambda-1)(1-t)^{s-\lambda-\frac12}} dt 
\\
&=(1-\chi_\lambda(\rho))^2\frac{(1-\rho)[1+e_1(\rho,\lambda)][1+r_1(\rho,\lambda)][1+e_2(\rho,\lambda)][1+r_2(t,\lambda)]}{(2s-2\lambda-1)(s-\lambda- \frac32)}
\\
&\quad\rho^{\frac{1-d}{2}}(1-\chi_\lambda(\rho))(1-\rho)^{s-\frac{1}{2}-\lambda}[1+e_1(\rho,\lambda)][1+r_1(\rho,\lambda)]
\\
&\quad \times \int_0^\rho \frac{(1-\chi_\lambda(t)) t^{\frac{d-1}{2}}[1+e_2(t,\lambda)][1+r_2(t,\lambda)]}{(2s-2\lambda-1)(s-\lambda- \frac32)(1-t)^{s-\lambda-\frac12}} dt
\\
=&: B(\rho,\lambda)+I(\rho,\lambda).
\end{align*}
Observe now, that if we subtract the free part from this expression, we obtain that schematically, $B-B_{\mathrm{f}}$ is of the form
\begin{align*}
 B(\rho,\lambda)- B_{\mathrm{f}}(\rho,\lambda)= (1-\chi_\lambda(\rho))^2[\O(\rho^0\langle\omega\rangle^{-3})+\O(\rho^{-1}\langle\omega\rangle^{-4})].
\end{align*}
Hence, also this boundary part is no obstruction to our desired interchanging identities.
Moreover, we can integrate by parts further, until we obtain enough decay in the integral term while all resulting boundary terms will be of the same form as $B(\rho,\lambda)$. Likewise, one decomposes $u_1U_2$ and 
\eqref{eq:interchange 1} follows. 
\\
Finally, as one can argue in the same fashion for the remaining terms, we conclude this proof.
\end{proof}
\subsection{Oscillatory integrals}
To now estimate all the oscillatory integrals, we split the difference of $\partial_\rho^n [\Rm(f)(\rho,\lambda)-\Rm_{\mathrm{f}}(f)(\rho,\lambda)]$ for $0\leq n\leq \ceil s$  into different smaller parts, once with $\Re \lambda=\mu_1$ and once with $\Re \lambda=\mu_0$.
In case $s$ is not an integer, the first parts we look at are given by

$$W_{0,1}(f)(\rho,\lambda):=\kappa(f)(\lambda)u_0(\rho,\lambda) -\kappa_{\mathrm{f}}(f)(\lambda)u_{\mathrm{f}_0}(\rho,\lambda) \qquad\qquad \text{for } \Re \lambda=\mu_1$$ 
and
$$
W_{0,0}(f)(\rho,\lambda):=\widehat \kappa(f)(\lambda)u_0(\rho,\lambda) -\widehat \kappa_{\mathrm{f}}(f)(\lambda)u_{\mathrm{f}_0}(\rho,\lambda) \qquad \qquad \text{for } \Re \lambda=\mu_0,$$
while for integer $s$ we only need to look at $W_{0,1}$ for both $\Re \lambda=\mu_1$ and $\Re \lambda=\mu_0$.
\begin{lem}\label{lem:decomp1}
 We can decompose $W_{0,1}(f)(\rho,\lambda)$ as 
$$
W_{0,1}(f)(\rho,\lambda)=\sum_{j=1}^{k-1}f^{(j-1)}(1)\sum_{l=1}^3 H_{j,\ell}(\rho,\lambda)$$ 
where 
\begin{align*}
H_{j,1}(\rho,\lambda):&=\chi_\lambda(\rho)(1-\rho^2)^{\frac{s}{2}-\frac{3}{4}-\frac{\lambda}{2}} \rho^{\frac{1-d}{2}}b_1(\rho,\lambda)
\\
&\quad\times\big[\O(\langle\omega\rangle^{-1})[1+\rho^2e_3(\rho,\lambda)]+\rho^2(e_3(\rho,\lambda)-e_{\mathrm{f}_3}(\rho,\lambda))\big]\O(\langle\omega\rangle^{-\frac{d}{2}-1})
\\
H_{j,2}(\rho,\lambda):&=(1-\chi_\lambda(\rho))\rho^{\frac{1-d}{2}}(1+\rho)^{s-\lambda-\frac{1}{2}}[1+e_2(\rho,\lambda)]\O(\langle\omega\rangle)^{-\frac{d+5}{2}}
\\
&\quad+(1-\chi_\lambda(\rho))\rho^{\frac{1-d}{2}}(1+\rho)^{s-\lambda-\frac{1}{2}}[1+e_2(\rho,\lambda)]r_2(\rho,\lambda)\O(\langle\omega\rangle)^{-\frac{d+3}{2}}
\\
H_{j,3}(\rho,\lambda):&=(1-\chi_\lambda(\rho))\rho^{\frac{1-d}{2}}(1-\rho)^{s-\lambda-\frac{1}{2}}[1+e_1(\rho,\lambda)]\O(\langle\omega\rangle)^{-\frac{d+5}{2}}
\\
&\quad+(1-\chi_\lambda(\rho))\rho^{\frac{1-d}{2}}(1-\rho)^{s-\lambda-\frac{1}{2}}[1+e_1(\rho,\lambda)]r_1(\rho,\lambda)\O(\langle\omega\rangle)^{-\frac{d+3}{2}}.
\end{align*}
Similarly, we can decompose $W_{0,0}$ as 
$$
W_{0,0}(f)(\rho,\lambda)=\sum_{j=1}^{k-2}f^{(j-1)}(1)\sum_{l=1}^3 H_{j,\ell}(\rho,\lambda).$$ 
\end{lem}
\begin{proof}
Recall from Lemma \ref{lem: form of solutions}
\begin{align*}
u_0(\rho,\lambda)&= \chi_\lambda(\rho)\rho^{\frac{1-d}{2}}(1-\rho^2)^{\frac{s}{2}-\frac{1}{4}-\frac{\lambda}{2}}\widehat{c}(\lambda)\sqrt{\varphi(\rho)}J_{\frac{d-2}{2}}( a(\lambda)\varphi(\rho))[1+\rho^2 e_3(\rho,\lambda)]
\\
&\quad + (1-\chi_\lambda(\rho))\widehat{c}(\lambda)\frac{\rho^{\frac{1-d}{2}}(1+\rho)^{s-\lambda-\frac{1}{2}}}{\sqrt{2s-2\lambda-1}}[1+e_2(\rho,\lambda)][1+r_2(\rho,\lambda)]
\\
&\quad + (1-\chi_\lambda(\rho)) \frac{ \rho^{\frac{1-d}{2}}(1-\rho)^{s-\lambda-\frac{1}{2}}}{\sqrt{2s-2\lambda-1}}[1+e_1(\rho,\lambda)][1+r_1(\rho,\lambda)]
\end{align*}
with $$\widehat c(\lambda) =\left[c_{1,3}(\lambda)-\frac{c_{2,3}(\lambda)}{c_{2,4}(\lambda)}c_{1,4}(\lambda)\right]=\O(\langle\omega\rangle^0).$$ Thus, desired decomposition is a direct consequence of splitting $u_{\mathrm{f}_0}$ likewise and a direct computation.
\end{proof}
Motivated by above decompositions, we define a family of operators $T_{j,\ell}^{n}$ corresponding to 
$$\int_\R e^{i \omega\tau}\partial_\rho^n W_{0,1}(f)(\rho,\mu_0+i\omega) d\omega \quad \text{and} \quad \int_\R e^{i \omega\tau}\partial_\rho^n W_{0,0}(f)(\rho,\mu_1+i\omega) d\omega$$ 
  as follows. In case $s\in \mathbb{N}$ or $a=1$, we define 
\begin{align*}
T_{j,\ell}^{n,a}(f)(\tau,\rho):=f^{(j-1)}(1) \int_\R e^{i \omega\tau} \partial_\rho^{n}H_{j,\ell}(\rho, \mu_a+i \omega) d\omega
\end{align*}
for $ 1\leq j\leq k-1$, $0\leq n\leq k-1$, $\ell=1,2,3$, $ a=0,1$, and $f\in C^\infty_{rad} (\overline{\B^d_1}).$
Analogously, we define 
\begin{align*}
\dot{T}_{j,\ell}^{n,a}(f)(\tau,\rho):=f^{(j-1)}(1)  \int_\R \omega e^{i \omega\tau} \partial_\rho^{n}H_{j,\ell}(\rho,\mu_a+i \omega) d\omega.
\end{align*}
In case $s\notin \mathbb N$ and $a=0$, we define the operators $T_{j,\ell}^{n,a}$ and $\dot T_{j,\ell}^{n,a}$ in the same fashion, with the only difference being that $j $ only ranges from $0$ to $k-2$.
\begin{lem} \label{lem:strich1a}
Let $ 3 \leq d \in \mathbb{N}$ and $1\leq s \notin \mathbb{N}$ with $1\leq k=\ceil s \leq \frac{d}{2}$.
Then, the estimates 
\begin{align*}
\||.|^{-m}T_{j,\ell}^{n-m,1}(f)\|_{L^p(\R_+)L^q(\B^d_1)}\lesssim \| f\|_{W^{j,\frac{2}{1+2\frac{\delta}{\theta}}}(\B^d_1)}
\end{align*}
and 
\begin{align*}
\||.|^{-m}\dot{T}_{j,\ell}^{n-m,1}(f)\|_{L^p(\R_+)L^q(\B^d_1)}\lesssim \| f\|_{W^{j+1,\frac{2}{1+2\frac{\delta}{\theta}}}(\B^d_1)}
\end{align*}
hold for all  $j,n,m\in \mathbb{Z}$ with $1\leq j\leq k-1$, with $0\leq n \leq k$, $0\leq m<n$, $\ell=1,2,3$, and
 $ p,q \in [\frac 1{\frac{1}{2}+\frac{\delta}{\theta}},\infty]$, such that the scaling relation $$\frac{1}{p}+\frac{d}{q}=d(\frac{1}{2}+\frac{\delta}{\theta})-k+n$$ 
is satisfied, as well as all $f\in C^\infty_{rad}(\overline{\B^d_1})$.
Moreover, 
if  $ s>\frac{d-1}{2}$, then, also the estimates
\begin{align*}
\|T_{j,\ell}^{0,1}(f)\|_{L^{\frac{2}{1+2\frac{\delta}{\theta}}}(\R_+)L^\infty(\B^d_1)}\lesssim \| f\|_{W^{j,\frac{2}{1+2\frac{\delta}{\theta}}}(\B^d_1)}
\end{align*}
and 
\begin{align*}
\|\dot{T}_{j,\ell}^{0,1}(f)\|_{L^{\frac{2}{1+2\frac{\delta}{\theta}}}(\R_+)L^\infty(\B^d_1)}\lesssim \| f\|_{W^{j+1,\frac{2}{1+2\frac{\delta}{\theta}}}(\B^d_1)}
\end{align*}
hold.
Similarly, the estimates 
\begin{align*}
\||.|^{-m}T_{j,\ell}^{n-m,0}(f)\|_{L^p(\R_+)L^q(\B^d_1)}\lesssim \| f\|_{W^{j,\frac{2}{1-\frac{2\delta}{1-\theta}}}(\B^d_1)}
\end{align*}
and 
\begin{align*}
\||.|^{-m}\dot{T}_{j,\ell}^{n-m,0}(f)\|_{L^p(\R_+)L^q(\B^d_1)}\lesssim \| f\|_{W^{j+1,\frac{2}{1 -\frac{2\delta}{1-\theta}}}(\B^d_1)}
\end{align*}
hold for all  $j,n,m\in \mathbb{Z}$ with $1\leq j\leq k-1$, $0\leq n \leq k-1$, $0\leq m<n$, $\ell=1,2,3$, and
 $ p,q \in [\frac{2}{1-\frac{2\delta}{1-\theta}},\infty]$, such that the scaling relation  $$\frac{1}{p}+\frac{d}{q}=d(\frac12 -\frac{\delta}{1-\theta})-k+1+n$$ 
is satisfied, as well as all $f\in C^\infty_{rad}(\overline{\B^d_1})$.
Lastly, for $s >\frac{d-1}{2}$ the estimates
 \begin{align*}
\|T_{j,\ell}^{0,0}(f)\|_{L^{\frac{2}{1-\frac{2\delta}{1-\theta}}}(\R_+)L^\infty(\B^d_1)}\lesssim \| f\|_{W^{j,\frac{2}{1-\frac{2\delta}{1-\theta}}}(\B^d_1)}
\end{align*}
and 
\begin{align*}
\|\dot{T}_{j,\ell}^{0,0}(f)\|_{L^{\frac{2}{1-\frac{2\delta}{1-\theta}}}(\R_+)L^\infty(\B^d_1)}\lesssim \| f\|_{W^{j+1,\frac{2}{ 1-\frac{2\delta}{1-\theta}}}(\B^d_1)}
\end{align*}
hold for $j=1,\dots ,k-1$ and $\ell=1,2,3$, and all $f\in C^\infty_{rad}(\overline{\B^d_1})$. 
\end{lem}
\begin{rem}
Before we come to the proof of this result, we want to make some remarks on the choices of the parameters. First, we note that $m$ and $n$ are chosen such that 
\begin{align*}
\|f\|_{\dot{W}^{n,p}(\B^d_1)}\lesssim \sum_{0\leq m <n}\||.|^{-m} f^{(n)}\|_{L^p(\B^d_1)}.
\end{align*}
Furthermore the pairs $(p,q)$ that satisfy the imposed scaling relations are chosen such that interpolating between the involved Strichartz norms will lead to the desired ones. For this interpolation argument to work, we need the additional estimate in case $s$ is close to $\frac{d}{2}$, since this implies that there will be pairs $(p,q)$ such that $L^pL^q$ is not admissible space, even though $p\geq 2$ as well as pairs of the form $p,\infty)$. Furthermore, it suffices to prove estimates for the endpoints, i.e. for the pairs $(p,q)=(\infty,q)$ and the one where $p$ is the lowest admissible value. The intermediate ones then follow by a simple interpolation argument.
\end{rem}
\begin{proof}
As a warm up, we consider, $ T_{j,1}^{0,a}$ and $\dot T_{j,1}^{0,a}$, for both $a=0,1$.
Then, since $$ 
\chi_\lambda(\rho)\rho^{\frac{1-d}{2}}b_1(\rho,\lambda)=\chi_\lambda(\rho)\O(\rho^0\langle\omega\rangle^{\frac{d-2}{2}}),$$
we see that
\begin{align*}
T_{j,1}^{0,a}(f)(\tau,\rho)=f(1) \int_\R e^{i \omega\tau}\chi_{\mu_a+i \omega}(\rho)(1-\rho^2)^{\frac{s}{2}-\frac{3}{4}-\frac{\mu_a+i\omega}{2}}\O(\rho^0\langle\omega\rangle^{-3} )d\omega
\end{align*}
as well as
\begin{align*}
\dot T_{j,1}^{0,a}(f)(\tau,\rho)=f(1) \int_\R e^{i \omega\tau}\chi_{\mu_a+i \omega}(\rho)(1-\rho^2)^{\frac{s}{2}-\frac{3}{4}-\frac{\mu_a+i \omega}{2}}\O(\rho^0\langle\omega\rangle^{-2} )d\omega
\end{align*}
which allows us to apply Lemma \ref{osci1} to derive that
\begin{align*}
| T_{j,1}^{0,a}(f)|+|\dot T_{j,1}^{0,a}(f)|&\lesssim |f^{(j-1)}(1)|\langle\tau-\log(1-\rho^2)\rangle^{-2}(1-\rho^2)^{\frac{s}{2}-\frac34-\frac{\mu_a}{2}}1_{[0,\rho_1]}(\rho)
\end{align*}
for some $\rho_1<1$ and where $1_{[0,\rho_1]}$ is the indicator function of the interval $[0,\rho_1]$.
So,
\begin{align*}
\| T_{j,1}^{0,a}(f)\|_{L^p(\R_+)L^q(\B^d_1)}+\|\dot T_{j,1}^{0,a}(f)\|_{L^p(\R_+)L^q(\B^d_1)}&\lesssim  |f^{(j-1)}(1)|.
\end{align*}
for all $p,q\in [1,\infty]$.
Further, one computes that for any smooth function $g$, one has that
\begin{align*}
|g(1)|= \left|\int_0^1 (\rho^{d}g(\rho))' d\rho \right|&\lesssim \|g\|_{W^{1,1}(\B^d_1)},
\end{align*}
from which one readily infers 
\begin{align*}
\| T_{j,1}^{0,a}(f)\|_{L^p(\R_+)L^q(\B^d_1)}+\|\dot T_{j,1}^{0,a}(f)\|_{L^p(\R_+)L^q(\B^d_1)}&\lesssim  \|f\|_{W^{j,1}(\B^d_1)}.
\end{align*}
So, we turn to $T_{j,1}^{n,1}$ for $n\geq 1$. In this case, the endpoint pairs $(p,q)$ are given by

$$
(\frac{2}{1+2\frac{\delta}{\theta}},\frac{d}{d(\frac{1}{2}+\frac{\delta}{\theta})-k+n-\frac{1}{2}+\frac{\delta}{\theta}}) \text{ and }(\infty,\frac{d}{d(\frac{1}{2}+\frac{\delta}{\theta})-k+n}).$$ Now, from the structure of $H_{j,1}$, one sees that differentiating it with respect to $\rho$ amounts to multiplying with $\O(\langle\omega\rangle)$. Consequently, by exchanging powers of $\rho$ for decay in $\omega$, we see that 
\begin{align*}
 T_{j,1}^{n,1}(f)(\rho,\lambda)=f^{(j-1)}(1)\int_\R e^{i \omega\tau} \chi_{\mu_1+ i\omega}(\rho)(1-\rho^2)^{\frac{s}{2}-\frac{3}{4}-\frac{\mu_1+i\omega}{2}} \O(\rho^{-n+c }\langle\omega\rangle^{-3+c}) d\omega
\end{align*}
for any $c>0$ fixed. Set $r=\frac{1}{2}+\frac{\delta}{\theta}>\frac12$. The goal here is to choose a $c>0$ small enough such that $-3+c<-1$ and
\begin{align*}
\||.|^{-n+c}\|_{L^{\frac{d}{dr +n-k-r}}(\B^d_1)}<\infty.
\end{align*} 
To this end, we compute that
\begin{equation}\label{eq:right power}
\begin{split}
d+\frac{d(-n+c)}{dr+n-k-r}&=\frac{d^2r-dk-d+dc}{dr n-k-r}\\
&\geq dr\frac{d(r-\frac{1}{2})-1+c}{dr+n-k-r}.
\end{split}
\end{equation}
Thus, if we set $c=1-\frac d2(r-\frac 12)<1$, one obtains that $|.|^{-n+c}\in L^{\frac{d}{dr +n-k-r}}(\B^d_1)$ since by \eqref{eq:right power}
\begin{align*}
\||.|^{-n+1-\frac d2(r-\frac 12)}\|_{L^{\frac{d}{dr +n-k-r}}(\B^d_1)}^{\frac{d}{dr +n-k-r}}
&=\int_0^1 \rho^{(-n+1-\frac d2(r-\frac 12 )){\frac{d}{dr +n-k-r}}}\rho^{d-1} d\rho 
\\
&\leq  \int_0^1 \rho^{-1+\frac d2 r\frac{d(r-\frac{1}{2})}{dr+n-k-r }} d\rho < \infty.
\end{align*}
Consequently, Lemma \ref{osci1} shows 
\begin{align*}
&\quad\|\rho^{-m}T_{j,1}^{n-m,1}(f)(\tau,\rho)\|_{L^p_\tau(\R_+)L^q_\rho(\B^d_1)}+\|\rho^{-m}\dot T_{j,1}^{n-m,1}(f)(\tau,\rho)\|_{L^p_\tau(\R_+)L^q_\rho(\B^d_1)}
\\
&\lesssim \|\langle\tau \rangle^{-2} \rho^{-n+c}\|_{L^p_\tau(\R_+)L^q_\rho(\B^d_1)} \| f\|_{W^{j,\frac 1r}(\B^d_1)}
 \end{align*}
for all desired $m,n,p,q$ as stated in the Lemma.
 Likewise, one derives the desired estimates on the operators $T_{j,1}^{n-m,0}$ and $\dot T_{j,1}^{n-m,0}$, so, we turn to
 $T^{n,a}_{j,2}$. Note that differentiating $H_{j,0}(\rho,\lambda)$ leads to two different scenarios. Either the derivative hits the cutoff function in which case all estimate are established as for $T^{n,1}_{j,1}$, or, one of the other terms get differentiated,
  which boils down to a multiplication with $[\rho^{-1}+\O(\langle\omega\rangle)]$.
For these terms, an application of Lemma \ref{osci3} shows
\begin{align*}
|\rho^{-m}T_{j,2}^{n-m,a}(f)(\tau,\rho)|+|\rho^{-m}\Dot T_{j,2}^{n-m,a}(f)(\tau,\rho)|\lesssim \|f\|_{W^{j,r}(\B^d_1)} \langle\tau\rangle^{-2} \rho^{1-n},
\end{align*}
for all $0\leq n \leq k$ and $a=0,1$. Hence, the claimed estimates follow and we turn to $\ell=3$ and $a=1$.
Now, for $n\leq k-1$ we can control $|.|^{-m}T_{j,3}^{n-m}$ and $|.|^{-m}\dot{T}_{j,3}^{n-m}$ in the same manner as $T_{j,3}^{n}$. For $n=k$ we study  the technically most involved case which is given by $\dot T_{j,3}^{k,1}$. Here, the most tricky term required to be bounded is 
\begin{align*}
&\quad \left|\int_\R e^{i\omega\tau}(1-\chi_{\mu_1+i\omega}(\rho))\rho^{\frac{1-d}{2}}(1-\rho)^{s-\ceil s-\frac 12-\mu_1 -i\omega}\O(\langle\omega\rangle^{-\frac32})\right|
\\
&\lesssim \left|\rho^{\frac{1-d}{2}}(1-\rho)^{s-\ceil s-\frac 12-\mu_1}\langle\tau-\log(1-\rho)\rangle^{-2}\right|
\\
&=\left|\rho^{\frac{1-d}{2}}(1-\rho)^{\theta-\frac 32-\mu_1}\langle\tau-\log(1-\rho)\rangle^{-2}\right|
\\
&= \left|\rho^{\frac{1-d}{2}}(1-\rho)^{-\frac12-\frac{\delta}{\theta}}\langle\tau-\log(1-\rho)\rangle^{-2}\right|
\end{align*}
for any $\varepsilon>0$ fixed by Lemma \ref{osci1}. Thus, changing variables according to $\rho=(1-e^{-y})$ shows
\begin{align*}
&\quad\left\|\int_\R e^{i\omega\tau}(1-\chi_{\mu_1+i\omega}(\rho))\rho^{\frac{1-d}{2}}(1-\rho)^{s-\ceil s-\frac 12-\mu_1 -i\omega}\O(\langle\omega\rangle^{-\frac{3}{2}})\right\|_{L^\infty_\tau(\R_+)L^{\frac{2}{1 +2\frac{\delta}{\theta}}}_\rho(\B^d_1)}^{\frac{2}{1 +2\frac{\delta}{\theta}}}
\\
&\lesssim \left\|\rho^{\frac{1-d}{2}}(1-\rho)^{-\frac12-\frac{\delta}{\theta}}\langle\tau-\log(1-\rho)\rangle^{-2}\right\|_{L^\infty_\tau(\R_+)L^{\frac{2}{1 +2\frac{\delta}{\theta}}}_\rho(\B^d_1)}^{\frac{2}{1 +2\frac{\delta}{\theta}}}
\\
&\leq\left\|\int_0^1 (1-\rho)^{-1}\langle\tau-\log(1-\rho)\rangle^{-\frac32} d\rho \right\|_{L^\infty_\tau(\R_+)}
\\
&=\left\|\int_0^\infty \langle\tau-y\rangle^{-\frac32} dy\right\|_{L^\infty_\tau(\R_+)} \leq \left\|\int_{-\infty}^\infty \langle y \rangle^{-\frac32} dy\right\|_{L^\infty_\tau(\R_+)}<\infty.
\end{align*}
So, we conclude that the desired bound on $\dot T_{j,3}^{k,1}$, and since the terms $|.|^{-m}  T_{j,3}^{k-m,1}$ and $ |.|^{-m}\dot T_{j,3}^{k-m,1}$ can be bounded likewise, only the claimed estimates on $ T_{j,3}^{k,0}$ and $\dot T_{j,3}^{k,0}$ are left. However, one readily establishes these with the help of Lemma \ref{osci5} and the same reasoning.
\end{proof}
We move on to the integer regularity case.
\begin{lem} \label{lem:strich1b}
Let $3 \leq d \in \mathbb{N}$ and  $ s  \in \mathbb{N}$ with $1\leq s <\frac{d}{2}$.
Then, the estimates 
\begin{align*}
\||.|^{-m}T_{j,\ell}^{n-m,1}(f)\|_{L^p(\R_+)L^q(\B^d_1)}\lesssim \| f\|_{W^{j,\frac{2}{1+2\delta}}(\B^d_1)}
\end{align*}
and 
\begin{align*}
\||.|^{-m}\dot{T}_{j,\ell}^{n-m,1}(f)\|_{L^p(\R_+)L^q(\B^d_1)}\lesssim \| f\|_{W^{j+1,\frac{2}{1+2\delta}}(\B^d_1)}
\end{align*}
hold for all  $j,n,m\in \mathbb{Z}$ with $1\leq j\leq s-1$, with $0\leq n \leq s$, $0\leq m<n$,  $\ell=1,2,3$, and
 $ p,q\in [\frac{2}{1+2\delta},\infty]$ such that the scaling relation $$\frac{1}{p}+\frac{d}{q}=d(\frac{1}{2}+\delta)-s+n,$$ 
is satisfied,  as well as all $f\in C^\infty_{rad}(\overline{\B^d_1})$.
 Furthermore, the estimates 
\begin{align*}
\||.|^{-m}T_{j,\ell}^{n-m,0}(f)\|_{L^p(\R_+)L^q(\B^d_1)}\lesssim \| f\|_{W^{j,\frac{2}{1-2\delta}}(\B^d_1)}
\end{align*}
and 
\begin{align*}
\||.|^{-m}\dot{T}_{j,\ell}^{n-m,0}(f)\|_{L^p(\R_+)L^q(\B^d_1)}\lesssim \| f\|_{W^{j+1,\frac{2}{1-2\delta}}(\B^d_1)}
\end{align*}
hold for the same $j,n,m, \ell $ as well as 
 $ p,q\in [\frac{2}{1-2\delta},\infty]$ such that the scaling relation  $$\frac{1}{p}+\frac{d}{q}=d(\frac12 -\delta)-s+n$$ is satisfied, as well as all $f\in C^\infty_{rad}(\overline{\B^d_1})$. 
 Lastly, in case $s=\frac{d-1}{2}$, the estimates
  \begin{align*}
\|T_{j,\ell}^{0,1}f\|_{L^{\frac{2}{1+2\delta}}(\R_+)L^\infty(\B^d_1)}&\lesssim \| f\|_{W^{s-1,\frac{2}{1+2\delta}}(\B^d_1)}
\\
\| \dot{T}_{j,\ell}^{0,1}f\|_{L^{\frac{2}{1+2\delta}}(\R_+)L^\infty(\B^d_1)}&\lesssim \| f\|_{W^{s,\frac{2}{1+2\delta}}(\B^d_1)}
\end{align*}
and
 \begin{align*}
\|T_{j,\ell}^{0,0}f\|_{L^{\frac{2}{1-2\delta}}(\R_+)L^\infty(\B^d_1)}&\lesssim \| f\|_{W^{s-1,\frac{2}{1-2\delta}}(\B^d_1)}
\\
\| \dot{T}_{j,\ell}^{0,0}f\|_{L^{\frac{2}{1-2\delta}}(\R_+)L^\infty(\B^d_1)}&\lesssim \| f\|_{W^{s,\frac{2}{1-2\delta}}(\B^d_1)}
\end{align*}
hold for $j=1,\dots ,k-1$ and $\ell=1,2,3$, and all $f\in C^\infty_{rad}(\overline{\B^d_1})$. 
\end{lem}
\begin{proof}
Up to minor details, this result can be shown ad verbatim as Lemma \ref{lem:strich1a}.
\end{proof}
We continue by defining
\begin{align*}
W_j^n(f)(\rho,\lambda):&=(-1)^{j+1}[\partial_\rho^n u_0(\rho,\lambda) U_{1,j}(\rho,\lambda) f^{(j-1)}(\rho) -\partial_\rho^n u_{\mathrm{f}_0}(\rho,\lambda) U_{\mathrm{f}_{1,j}}(\rho,\lambda) f^{(j-1)}(\rho)]
\\
&\quad +(-1)^j[\partial_\rho^n u_1(\rho,\lambda) U_{0,j}(\rho,\lambda) f^{(j-1)}(\rho) -\partial_\rho^n u_{\mathrm{f}_1}(\rho,\lambda) U_{\mathrm{f}_{0,j}}(\rho,\lambda) f^{(j-1)}(\rho)]
\end{align*}
for $0 \leq n \leq j, 1 \leq j \leq k-1$  and
\begin{align*}
W_{j}^n(f)(\rho,\lambda):&=(-1)^{j+1}
\partial_\rho^{n-j-1}\left[\partial_\rho \left(\partial_\rho^j u_0(\rho,\lambda)U_{1,j}(\rho,\lambda)-\partial_\rho^j u_{\mathrm{f}_0}(\rho,\lambda)U_{\mathrm{f}_{1,j}}(\rho,\lambda)\right)f^{(j-1)}(\rho)\right]
\\
&\quad-(-1)^{j} 
\partial_\rho^{n-j-1}\left[\partial_\rho \left(\partial_\rho^j u_1(\rho,\lambda)U_{1,0}(\rho,\lambda)-\partial_\rho^j u_{\mathrm{f}_1}(\rho,\lambda)U_{\mathrm{f}_{0,j}}(\rho,\lambda)\right)f^{(j-1)}(\rho)\right]
\end{align*}
for $1\leq j< n\leq k$.
\begin{lem}\label{lem:decomp2}
Let $\Re \lambda= \mu_1$ and $k= \ceil s$.
For $0 \leq n \leq j$ and $ 1 \leq j \leq k-1$, we can decompose $W_j^n(f)$ as
\begin{align*}
W_j^n(f)(\rho,\lambda)=(-1)^{j+1}\sum_{\ell=4}^8 H_{j,\ell}^n(f)(\rho,\lambda)
\end{align*}
where
\begin{align*}
H_{j,4}^n(\rho,\lambda):&=f^{(j-1)}(\rho)\chi_\lambda(\rho)(1-\rho^2)^{\frac{s}{2}-\frac{3}{4}-\frac{\lambda}{2}}\int_0^\rho\int_0^{t_1}\dots \int_0^{t_{j-1}} \frac{\O(\rho^0 t_j\langle \omega\rangle^{-1 +n })}{(1-t_j^2)^{\frac{s}{2}
+\frac{1}{4}-\frac{\lambda}{2}}} dt_j\dots dt_2 dt_1
\\
H_{j,5}^n(\rho,\lambda):&=f^{(j-1)}(\rho)(1-\chi_\lambda(\rho))\int_0^\rho\int_0^{t_1}\dots \int_0^{t_{j-1}}\frac{\chi_\lambda(t)\O(t_j\langle\omega\rangle^{-\frac{d-1}{2}})}{(1-t_j^2)^{\frac{s}{2}
+\frac{1}{4}-\frac{\lambda}{2}}}  
\\
&\quad \times \partial_\rho^n\left[\rho^{\frac{1-d}{2}}(1+\rho)^{s-\lambda-\frac{1}{2}}[1+(1-\rho)\O(\rho^{-1}\langle\omega\rangle^{-1})]\beta_1(\rho,t_j,\lambda) \right] dt_j\dots dt_2 dt_1,
\end{align*}
and
\begin{align*}
H_{j,6}^n(\rho,\lambda):&=f^{(j-1)}(\rho)(1-\chi_\lambda(\rho))\int_0^\rho\int_0^{t_1}\dots \int_0^{t_{j-1}}\frac{\chi_\lambda(t_j)\O(t_j\langle\omega\rangle^{-\frac{d-1}{2}})}{(1-t^2)^{\frac{s}{2}
+\frac{1}{4}-\frac{\lambda}{2}}}  
\\
&\quad \times \partial_\rho^n\left[\rho^{\frac{1-d}{2}}(1-\rho)^{s-\lambda-\frac{1}{2}}[1+(1-\rho)\O(\rho^{-1}\langle\omega\rangle^{-1})]\beta_2(\rho,t_j,\lambda) \right] dt_j\dots dt_2 dt_1
\\
H_{j,7}^n(\rho,\lambda):&= f^{(j-1)}(\rho)(1-\chi_\lambda(\rho))  \partial_\rho^n\left[\rho^{\frac{1-d}{2}}(1-\rho)^{s-\lambda-\frac{1}{2}}[1+e_1(\rho,\lambda)][1+r_1(\rho,\lambda)]\right]
\\
&\quad  \times 
\int_0^\rho\int_0^{t_1}\dots \int_0^{t_{j-1}}\frac{t^{\frac{d-1}{2}}(1-\chi_\lambda(t_j))[1+e_2(t_j,\lambda)]r_2(t_j,\lambda)}{(2s-2\lambda-1)(1-t_j)^{s-\lambda-\frac{1}{2}}}
dt_j\dots dt_2 dt_1
\\
&\quad f^{(j-1)}(\rho)(1-\chi_\lambda(\rho)) \partial_\rho^n\left[\rho^{\frac{1-d}{2}}(1-\rho)^{s-\lambda-\frac{1}{2}}[1+e_1(\rho,\lambda)]r_1(\rho,\lambda)\right]
\\
&\quad \times \int_0^\rho\int_0^{t_1}\dots \int_0^{t_{j-1}}\frac{t_j^{\frac{d-1}{2}}(1-\chi_\lambda(t_j))[1+e_2(t_j,\lambda)]}{(2s-2\lambda-1)(1-t_j)^{s-\lambda-\frac{1}{2}}}
dt_j\dots dt_2 dt_1
\\
H_{j,8}^n(\rho,\lambda):&=-f^{(j-1)}(\rho)(1-\chi_\lambda(\rho))\partial_\rho^n\left[\rho^{\frac{1-d}{2}}(1+\rho)^{s-\lambda-\frac{1}{2}}[1+e_2(\rho,\lambda)]\right]
\\
&\quad \times \int_0^\rho\int_0^{t_1}\dots \int_0^{t_{j-1}}\frac{t_j^{\frac{d-1}{2}}(1-\chi_\lambda(t_j))[1+e_1(t_j,\lambda)]r_1(t_j,\lambda)}{(2s-2\lambda-1)(1+t_j)^{s-\lambda-\frac{1}{2}}}dt_j\dots dt_2 dt_1
\\
&\quad -f^{(j-1)}(\rho)(1-\chi_\lambda(\rho)) \partial_\rho^n\left[\rho^{\frac{1-d}{2}}(1+\rho)^{s-\lambda-\frac{1}{2}}[1+e_2(\rho,\lambda)]r_2(\rho,\lambda)\right]
\\
&\quad \times \int_0^\rho\int_0^{t_1}\dots \int_0^{t_{j-1}}\frac{t_j^{\frac{d-1}{2}}(1-\chi_\lambda(t_j))[1+e_1(t_j,\lambda)][1+r_1(t_j,\lambda)]}{(2s-2\lambda-1)(1+t_j)^{s-\lambda-\frac{1}{2}}}
dt_j\dots dt_2 dt_1
\end{align*}
with $$
\beta_j(\rho,s,\lambda)=\O(\langle\omega\rangle^{-1})+\O(s^2\langle\omega\rangle^0)+(1-\rho)\O(\rho^0\langle\omega\rangle^{-1})
+(1-\rho)\O(\rho^0 s^2\langle\omega\rangle^{-1}).
$$ 
The same decomposition holds in case $\Re \lambda=\mu_0$ (where $1\leq j\leq k-2$ in case $s\notin \mathbb{N}$). 
\begin{proof}
To prove this result, one decomposes the functions $u_j$ as in the proof of Lemma \ref{lem:decomp1} by making use of Lemma \ref{lem: form of solutions}. The decomposition then follows from a straightforward calculation.
\end{proof}
\end{lem}
\begin{lem}
Let $\Re \lambda= \mu_1$ and $k= \ceil s$. Then, for $1\leq j< n \leq k$, we can decompose  $W_j^n(f)$ as
\begin{align*}
W_j^n(f)(\rho,\lambda)=\sum_{\ell=4}^8 H_{j,\ell}^n(f)(\rho,\lambda)
\end{align*}
where
\begin{align*}
H_{j,4}^n(\rho,\lambda):&=\chi_\lambda(\rho)\partial_\rho^{n-j-1}\bigg[f^{(j-1)}(\rho)\partial_\rho\bigg((1-\rho^2)^{\frac{s}{2}-\frac{3}{4}-\frac{\lambda}{2}}
\\
&\quad \times \int_0^\rho\int_0^{t_1}\dots \int_0^{t_{j-1}} \frac{\O(\rho^0 t_j\langle \omega\rangle^{-1 +j })}{(1-t_j^2)^{\frac{s}{2}
+\frac{1}{4}-\frac{\lambda}{2}}} dt_{j}\dots dt_2 dt_1\bigg)\bigg]
\\
H_{j,5}^n(\rho,\lambda):&=(1-\chi_\lambda(\rho))\partial_\rho^{n-j-1}\bigg[f^{(j-1)}(\rho)\partial_\rho\bigg(\int_0^\rho\int_0^{t_1}\dots \int_0^{t_{j-1}}\frac{\chi_\lambda(t_j)\O(t_j\langle\omega\rangle^{-\frac{d-1}{2}})}{(1-t_j^2)^{\frac{s}{2}
+\frac{1}{4}-\frac{\lambda}{2}}}  
\\
&\quad \times \partial_\rho^j\left[\rho^{\frac{1-d}{2}}(1+\rho)^{s-\lambda-\frac{1}{2}}[1+(1-\rho)\O(\rho^{-1}\langle\omega\rangle^{-1})]\beta_1(\rho,t_j,\lambda) \right] dt_j\dots dt_2 dt_1\bigg)\bigg]
\end{align*}
and
\begin{align*}
H_{j,6}^n(\rho,\lambda):&=(1-\chi_\lambda(\rho))\partial_\rho^{n-j-1}\bigg[f^{(j-1)}(\rho)\partial_\rho\bigg(\int_0^\rho\int_0^{t_1}\dots \int_0^{t_{j-1}}\frac{\chi_\lambda(t_j)\O(t_j\langle\omega\rangle^{-\frac{d-1}{2}})}{(1-t_j^2)^{\frac{s}{2}
+\frac{1}{4}-\frac{\lambda}{2}}}  
\\
&\quad \times \partial_\rho^j\left[\rho^{\frac{1-d}{2}}(1-\rho)^{s-\lambda-\frac{1}{2}}[1+(1-\rho)\O(\rho^{-1}\langle\omega\rangle^{-1})]\beta_2(\rho,t_j,\lambda) \right] dt_j\dots dt_2 dt_1\bigg)\bigg]
\\
H_{j,7}^n(\rho,\lambda):&=(1-\chi_\lambda(\rho)) \partial_\rho^{n-j-1}\bigg[f^{(j-1)}(\rho)
\\
&\quad \times \partial_\rho\bigg( \partial_\rho^j\left[\rho^{\frac{1-d}{2}}(1-\rho)^{s-\lambda-\frac{1}{2}}[1+e_1(\rho,\lambda)][1+r_1(\rho,\lambda)]\right]
\\
&\quad  \times 
\int_0^\rho\int_0^{t_1}\dots \int_0^{t_{j-1}}\frac{t^{\frac{d-1}{2}}(1-\chi_\lambda(t_j))[1+e_2(t_j,\lambda)]r_2(t_j,\lambda)}{(2s-2\lambda-1)(1-t_j)^{s-\lambda-\frac{1}{2}}}
dt_j\dots dt_2 dt_1\bigg)\bigg]
\\
&\quad +(1-\chi_\lambda(\rho)) \partial_\rho^{n-j-1}\bigg[f^{(j-1)}(\rho)\partial_\rho
\\
&\quad \times \bigg(\partial_\rho^j\left[\rho^{\frac{1-d}{2}}(1-\rho)^{s-\lambda-\frac{1}{2}}[1+e_1(\rho,\lambda)]r_1(\rho,\lambda)\right]
\\
&\quad \times \int_0^\rho\int_0^{t_1}\dots \int_0^{t_{j-1}}\frac{t_j^{\frac{d-1}{2}}(1-\chi_\lambda(t_j))[1+e_2(t_j,\lambda)]}{(2s-2\lambda-1)(1-t_j)^{s-\lambda-\frac{1}{2}}}
dt_j\dots dt_2 dt_1\bigg)\bigg],
\end{align*}
and
\begin{align*}
H_{j,8}^n(\rho,\lambda):&=-(1-\chi_\lambda(\rho)) \partial_\rho^{n-j-1}\bigg[f^{(j-1)}(\rho)
\\
&\quad \times \partial_\rho\bigg(\partial_\rho^j\left[\rho^{\frac{1-d}{2}}(1+\rho)^{s-\lambda-\frac{1}{2}}[1+e_2(\rho,\lambda)]\right]
\\
&\quad \times \int_0^\rho\int_0^{t_1}\dots \int_0^{t_{j-1}}\frac{t_j^{\frac{d-1}{2}}(1-\chi_\lambda(t_j))[1+e_1(t_j,\lambda)]r_1(t_j,\lambda)}{(2s-2\lambda-1)(1+t_j)^{s-\lambda-\frac{1}{2}}}dt_j\dots dt_2 dt_1\bigg)\bigg]
\\
&\quad -(1-\chi_\lambda(\rho)) \partial_\rho^{n-j-1}\bigg[f^{(j-1)}(\rho)
\\
&\quad \times \partial_\rho\bigg(\partial_\rho^j\left[\rho^{\frac{1-d}{2}}(1+\rho)^{s-\lambda-\frac{1}{2}}[1+e_2(\rho,\lambda)]r_2(\rho,\lambda)\right]
\\
&\quad \times \int_0^\rho\int_0^{t_1}\dots \int_0^{t_{j-1}}\frac{t_j^{\frac{d-1}{2}}(1-\chi_\lambda(t_j))[1+e_1(t_j,\lambda)][1+r_1(t_j,\lambda)]}{(2s-2\lambda-1)(1+t_j)^{s-\lambda-\frac{1}{2}}}
dt_j\dots dt_2 dt_1\bigg)\bigg]
\end{align*}
with $\beta_j$ as in Lemma \ref{lem:decomp2}.
The same decomposition holds in case $\Re \lambda=\mu_0$ (where $1\leq j<n\leq k-2$ in case $s\notin \mathbb{N}$).
\end{lem}
We continue as above and define the operators $T_{j,\ell}^{n,a}$ and $\dot{T}_{j,\ell}^{n,a}$ 
associated to $ W_j$ in the non integer case as
\begin{align*}
T_{j,\ell}^{n,a}f(\tau,\rho)&= \int_\R e^{i \omega\tau}f(\rho)H_{j,\ell}^n(\rho,\mu_a+ i \omega) d\omega
\end{align*}
and
\begin{align*}
\dot{T}_{j,\ell}^{n,a}f(\tau,\rho)&= \int_\R \omega e^{i \omega\tau}f(\rho)H_{j,\ell}^n(\rho,\mu_a +i \omega) d\omega
\end{align*}
for $\ell=0,1,\dots 8$, all $f\in C^\infty_{rad}(\overline{\B^d_1})$ as well as $0\leq n\leq k$ and $1\leq j \leq k-1$ in case $a=1$ and $0\leq n\leq k-1$, $1\leq j \leq k-2$ in case $a=0$.\\
If $s\in \mathbb{N}$, then we define the operators $T_{j,\ell}^{n,a}f(\tau,\rho)$ and $\dot T_{j,\ell}^{n,a}f(\tau,\rho)$ for the same parameters as in the case $a=1$ above, for both $a=0,1$.
\begin{lem}\label{lem:strichartz2a}
Let $3 \leq d \in \mathbb{N}$ and $1\leq s \notin \mathbb{N}$ with $1\leq k=\ceil s \leq \frac{d}{2}$.
Then, the estimates 
\begin{align*}
\left\||.|^{-m}\sum_{\ell=4}^8 T_{j,\ell}^{n-m,1}f\right\|_{L^p(\R_+)L^q(\B^d_1)}\lesssim \| f\|_{W^{k-1,\frac{2}{1+2\frac{\delta}{\theta}}}(\B^d_1)}
\end{align*}
and 
\begin{align*}
\left\||.|^{-m} \sum_{\ell=4}^8 \dot{T}_{j,\ell}^{n-m,1}f\right\|_{L^p(\R_+)L^q(\B^d_1)}\lesssim \| f\|_{W^{k,\frac{2}{1+2\frac{\delta}{\theta}}}(\B^d_1)}
\end{align*}
hold for all $j,n,m\in \mathbb{Z}$ with $1\leq j\leq k-1$, $0\leq n \leq k$, $0\leq m<n$, and
 $p,q\in [\frac{2}{1+2\frac{\delta}{\theta}},\infty]$ such that the scaling relation $$\frac{1}{p}+\frac{d}{q}=d(\frac{1}{2}+\frac{\delta}{\theta})-k+n$$ 
is satisfied, as well as all $f\in C^\infty_{rad}(\overline{\B^d_1})$.
Moreover, if $s>\frac{d-1}{2}$, then, also the estimates
\begin{align*}
\|\sum_{\ell=4}^8 T_{j,\ell}^{0,1}(f)\|_{L^{\frac{2}{1+2\frac{\delta}{\theta}}}(\R_+)L^\infty(\B^d_1)}\lesssim \| f\|_{W^{k-1,\frac{2}{1+2\frac{\delta}{\theta}}}(\B^d_1)}
\end{align*}
and 
\begin{align*}
\|\sum_{\ell=4}^8 \dot{T}_{j,\ell}^{0,1}(f)\|_{L^{\frac{2}{1+2\frac{\delta}{\theta}}}(\R_+)L^\infty(\B^d_1)}\lesssim \| f\|_{W^{k,\frac{2}{1+2\frac{\delta}{\theta}}}(\B^d_1)}
\end{align*}
hold. Similarly, the estimates 
\begin{align*}
\||.|^{-m} \sum_{\ell=4}^8 T_{j,\ell}^{n-m,0}f\|_{L^p(\R_+)L^q(\B^d_1)}\lesssim \| f\|_{W^{k-2,\frac{2}{1-\frac{2\delta}{1-\theta}}}(\B^d_1)}
\end{align*}
and 
\begin{align*}
\||.|^{-m} \sum_{\ell=4}^8 \dot{T}_{j,\ell}^{n-m,0}f\|_{L^p(\R_+)L^q(\B^d_1)}\lesssim \| f\|_{W^{k-1 ,\frac{2}{1-\frac{2\delta}{1-\theta}}}(\B^d_1)}
\end{align*}
hold for all  $j,n,m\in \mathbb{Z}$ with $1\leq j\leq k-2$, $0\leq n \leq k-1$, $0\leq m<n$, $p,q\in [\frac{2}{1-\frac{2\delta}{1-\theta}},\infty]$ such that the scaling relation  $$\frac{1}{p}+\frac{d}{q}=d(\frac12 -\frac{\delta}{1-\theta})-k+1+n$$ 
is satisfied, as well as all $f\in C^\infty_{rad}(\overline{\B^d_1})$.
\end{lem}
Lastly, for $s >\frac{d-1}{2}$, the estimates
 \begin{align*}
\|T_{j,\ell}^{0,0}(f)\|_{L^{\frac{2}{1-\frac{2\delta}{1-\theta}}}(\R_+)L^\infty(\B^d_1)}\lesssim \| f\|_{W^{k-2,\frac{2}{1-\frac{2\delta}{1-\theta}}}(\B^d_1)}
\end{align*}
and 
\begin{align*}
\|\dot{T}_{j,\ell}^{0,0}(f)\|_{L^{\frac{2}{1-\frac{2\delta}{1-\theta}}}(\R_+)L^\infty(\B^d_1)}\lesssim \| f\|_{W^{k-1,\frac{2}{ 1-\frac{2\delta}{1-\theta}}}(\B^d_1)}
\end{align*}
hold for $j=1,\dots,k-2$, $\ell=4,\dots 8$, and all $f\in C^\infty_{rad}(\overline{\B^d_1})$.
\begin{proof}
We only exhibit the proofs of the estimates on the operators $|.|^{-m}  \dot{T}_{j,\ell}^{n-m,1}$ as, aside from the last two estimates stated in the Lemma, the bounds on the operators $|.|^{-m} T_{j,\ell}^{n-m,0}$ and $|.|^{-m}  \dot{T}_{j,\ell}^{n-m,0}$ follow by the same means. The $L^p L^\infty$ bounds on the operators $T_{j,\ell}^{0,0}$ and $\dot T_{j,\ell}^{0,0}$ will be established at the end of this proof.
We start with the range $0 \leq n \leq j$ and $ 1 \leq j \leq k-1$. Then,  $T_{j,4}^{n-m,1}$
satisfies
\begin{align*}
\rho^{-m}T_{j,4}^{n-m,1}f(\tau,\rho)&=f^{(j-1)}(\rho)\int_0^\rho\int_0^{t_1}\dots \int_0^{t_{j-1}}\int_\R e^{i \omega\tau}\\\
&\quad \times (1-\rho^2)^{\frac{s}{2}-\frac{3}{4}-\frac{\mu_1+i\omega}{2}}\chi_{\mu_1+i\omega}(\rho) \frac{\O(\rho^{-n} t_j^{1-\varepsilon}\langle \omega\rangle^{-1-\varepsilon})}{(1-t_j^2)^{\frac{s}{2}
+\frac{1}{4}-\frac{\mu_1+i\omega}{2}}} d\omega
dt_j \dots dt_1
\end{align*}
for any $\varepsilon >0$ and all $0\leq m<n$. Hence,
\begin{align*}
|\rho^{-m}    T_{j,4}^{n-m,1}f(\tau,\rho)|\lesssim_\varepsilon \langle\tau\rangle^{-2} \rho^{-n+j+1-\varepsilon}|f^{(j-1)}(\rho)|
\end{align*}
and the claimed estimates on $|.|^{-m}T_{j,4}^{n-m}$ follow from Lemmas \ref{teclem2}, \ref{teclem3}, and \ref{teclem4}. Likewise, one concludes
\begin{align*}
\||.|^{-m}\dot{T}_{j,4}^{n-m,1}f\|_{L^p(\R_+)L^q(\B^d_1)}\lesssim  \|f\|_{W^{k,\frac 1r}(\B^d_1)}
\end{align*}
for all admissible sets of $n,m,p,q$ and with $r=\frac{1}{2}+\frac{\delta}{\theta}
$ as well as all $f\in C^\infty(\overline{\B^d_1})$.
Thus, we turn to $|.|^{-m}T_{1,5}^{n-m,1}$.
An application of Lemma \ref{osci3} shows 
\begin{align*}
|\rho^{-m}T_{j,5}^{n-m,1}f(\tau,\rho)|&\lesssim \langle\tau-\log(1+\rho)\rangle^{-2} \sum_{j=1}^n   |f^{(j-1)}(\rho)|\rho^{1+j-n}\\
&\lesssim  \langle\tau\rangle^{-2} \sum_{j=1}^n   |f^{(j-1)}(\rho)|\rho^{1+j-n}
\end{align*}
and
\begin{align*}
|\rho^{-m}\dot{T}_{j,5}^{n-m,1}f(\tau,\rho)|&\lesssim \langle\tau\rangle^{-2}  \sum_{j=1}^n  |f^{(j-1)}(\rho)|\rho^{j-n}
\end{align*}
and the desired bounds on $|.|^{-m}T_{j,5}^{n-m,1}$ and $|.|^{-m}\dot{T}_{j,5}^{n-m,1}$ follow again from Lemmas \ref{teclem2}, \ref{teclem3} and \ref{teclem4} and  we turn to $ \ell=6$. For $n <k$ (or $n=k$ and $m>0$), one can bound $|.|^{-m}T_{j,6}^{n-m,1}$, and $|.|^{-m}\dot{T}_{j,6}^{n-m,1}$ just as $|.|^{-m} T_{j,5}^{n-m,1}$ and $|.|^{-m} \dot T_{j,5}^{n-m,1}$. So, we study 
 $T_{j,6}^{k,1}$.
Employing Lemma \ref{osci3} yields
\begin{align*}
|T_{j,6}^{k,1}(f)(\tau)(\rho)|&\lesssim \sum_{\ell=0}^{k-j-1}\rho^{1+j-k+\ell}|f^{(j+\ell-1)}(\rho)|(1-\rho)^{-r}\langle\tau-\log(1-\rho)\rangle^{-2}
\end{align*}
with $r=\frac{1}{2}+\frac{\delta}{\theta}$. By arguing as in the proof of Lemma \ref{lem:strich1a}, one deduces
\begin{align*}
&\quad \left\|\rho^{1+j-k+\ell}f^{(j+\ell-1)}(\rho)(1-\rho)^{-r}\langle\tau-\log(1-\rho)\rangle^{-2}\right\|_{L^\infty_\tau(\R_+)L^{\frac{1}{r}}\rho(\B^d_1)}^{\frac{1}{r}}
\\
&\lesssim \left\|\int_0^1\rho^{d-1+\frac 1r(1+j-k+\ell)}|f^{(j+\ell-1)}(\rho)|^{\frac{1}{r}}(1-\rho)^{-1}\langle\tau-\log(1-\rho)\rangle^{-4}d \rho\right\|_{L^\infty_\tau(\R_+)}
\\
&= \left\|\int_0^\infty (1-e^{-y})^{d-1+\frac 1r(1+j-k+\ell)}|f^{(j+\ell-1)}(1-e^{-y})|^{\frac{1}{r}}\langle\tau-y\rangle^{-4}dy \right\|_{L^\infty_\tau(\R_+)}
\\
& \lesssim \| |.|^{r(d-1)+r+j-k+\ell+(1-r)}f^{(j+\ell-1)}\|_{L^\infty(\B^d_1)}^{\frac{1}{r}} \lesssim
\| f\|_{W^{\frac 1r, k-1}(\B^d_1)}^{\frac 1r}
\end{align*}
where the last inequality is thanks to Lemma \ref{teclem2}. In the same fashion one estimates
$ \dot T_{j,6}^{k,1}$ and we move on to $T_{1,7}^{0,1}$. Here, utilizing Lemma \ref{osci3} leads to the estimate
\begin{align*}
|T_{1,7}^{0,1}f(\tau,\rho)|&\lesssim |f(\rho)|\rho^{\frac{d-1}{2}}(1-\rho)^{s-\mu_1-\frac{1}{2}}\int_0^\rho \frac{s \langle\tau-\log(1-\rho)+\log(1-t)\rangle^{-2}}{(1-t)^{s-\mu_1-\frac{1}{2}}} dt
\\
&\lesssim |f(\rho)|\rho^{2}\langle\tau\rangle^{-2}
\end{align*}
and once more, the desired estimates follow from our technical Lemmas. In the same way one bounds $T_{j,7}^{0,1}$ for the remaining $j$.
Now, when derivatives are involved, we need to take a closer look at the integral terms in the operators $T_{j,7}^{n}$, which are given by
\begin{align*}
\int_0^\rho\int_0^{t_1}\dots \int_0^{t_{j-1}}\frac{t^{\frac{d-1}{2}}(1-\chi_\lambda(t_j))[1+e_2(t_j,\lambda)]r_2(t_j,\lambda)}{(2s-2\lambda-1)(1-t)^{s-\lambda-\frac{1}{2}}}
dt_j\dots dt_2 dt_1
\end{align*}
and
\begin{align*}
\int_0^\rho\int_0^{t_1}\dots \int_0^{t_{j-1}}\frac{t^{\frac{d-1}{2}}(1-\chi_\lambda(t_j))[1+e_2(t_j,\lambda)]}{(2s-2\lambda-1)(1-t_j)^{s-\lambda-\frac{1}{2}}}
dt_j\dots dt_2 dt_1.
\end{align*}
Integrating by parts $j $ times shows that 
\begin{align*}
&\quad \int_0^\rho\int_0^{t_1}\dots \int_0^{t_{j-1}}\frac{t^{\frac{d-1}{2}}(1-\chi_\lambda(t_j))[1+e_2(t_j,\lambda)]r_2(t,\lambda)}{(2s-2\lambda-1)(1-t_j)^{s-\lambda-\frac{1}{2}}}
dt_j\dots dt_2 dt_1 
\\
&= -\sum_{\ell=1}^j   \frac{1}{\prod_{b=1}^{j-\ell+1} (s-\lambda-\frac{1}{2}-b)(2s-2\lambda-1)}
\\
&\quad \times \int_0^\rho\int_0^{t_\ell}\dots \int_0^{t_{\ell-1}}\frac{\partial_{t_\ell}\left(t_\ell^{\frac{d-1}{2}}(1-\chi_\lambda(t_\ell))[1+e_2(t_\ell,\lambda)]r_2(t_\ell,\lambda)\right)}{(1-t_\ell)^{s-\lambda-\frac{3}{2}-j+\ell}}
dt_\ell\dots dt_2 dt_1 
\\
&\quad+ \frac{\rho^{\frac{d-1}{2}}(1-\chi_\lambda(\rho))[1+e_2(\rho,\lambda)]r_2(\rho,\lambda)}{\prod_{\ell=1}^{j} (s-\lambda-\frac{1}{2}-\ell) (2s-2\lambda-1)(1-\rho)^{s-\lambda-\frac{1}{2}-j}}
\\
&=:\sum_{\ell=1}^j I_{7,\ell,1}^j(\rho,\lambda)+B_{7,1}^j(\rho,\lambda).
\end{align*}
Likewise,
\begin{align*}
&\quad \int_0^\rho\int_0^{t_1}\dots \int_0^{t_{j-1}}\frac{t^{\frac{d-1}{2}}(1-\chi_\lambda(t_j))[1+e_2(t_j,\lambda)]}{(2s-2\lambda-1)(1-t_j)^{s-\lambda-\frac{1}{2}}}
dt_j\dots dt_2 dt_1 
\\
&= -\sum_{\ell=1}^j \int_0^\rho\int_0^{t_\ell}\dots \int_0^{t_{\ell-1}}\frac{\partial_{t_\ell}\left(t_\ell^{\frac{d-1}{2}}(1-\chi_\lambda(t_\ell))[1+e_2(t_\ell,\lambda)]\right)}{\prod_{b=1}^{j-\ell+1} (s-\lambda-\frac{1}{2}-b)(2s-2\lambda-1)(1-t_\ell)^{s-\lambda-\frac{3}{2}-j+\ell}}
dt_\ell\dots dt_2 dt_1 
\\
&\quad+ \frac{\rho^{\frac{d-1}{2}}(1-\chi_\lambda(\rho))[1+e_2(\rho,\lambda)]}{\prod_{\ell=1}^{j} (s-\lambda-\frac{1}{2}-\ell) (2s-2\lambda-1)(1-\rho)^{s-\lambda-\frac{1}{2}-j}}
\\
&=:\sum_{\ell=1}^j I_{7,\ell,2}^j(\rho,\lambda)+B_{7,2}^j(\rho,\lambda)
\end{align*}
which implies that
\begin{align*}
H_{j,7}^n&\quad=f^{(j-1)}(\rho)(1-\chi_\lambda(\rho))^2  \partial_\rho^n\left[\rho^{\frac{1-d}{2}}(1-\rho)^{s-\lambda-\frac{1}{2}}[1+e_1(\rho,\lambda)][1+r_1(\rho,\lambda)]\right]
\\
&\quad \times \left[\sum_{\ell=1}^j I_{7,\ell,1}^j(\rho,\lambda)+B_{7,1}^j(\rho,\lambda)\right]
 \\
&\quad+f^{(j-1)}(\rho)(1-\chi_\lambda(\rho))^2  \partial_\rho^n\left[\rho^{\frac{1-d}{2}}(1-\rho)^{s-\lambda-\frac{1}{2}}[1+e_1(\rho,\lambda)]r_1(\rho,\lambda)\right]
\\
&\quad \times \left[\sum_{\ell=1}^j(\rho,\lambda) I_{7,\ell,2}^j+B^j_{7,2}(\rho,\lambda)\right]
\end{align*}
for $0\leq n\leq j$.
By using this expression, one easily manages to bound
$|.|^{-m} T_{j,7}^{n-m,1}$ and $|.|^{-m} \dot T_{j,7}^{n-m,1}$ for as long as $n-m <j, 0 \leq m<n $, as the boundary parts exhibit enough decay in $\omega$, while one can perform additional integrations by parts in the integral terms to make them decay fast enough such that one can apply Lemma \ref{osci3}.\\ For $n-m=j$ (and $n-m >j)$, we need a cancellation. Therefore we note that in the same fashion, one computes that
\begin{align*}
H_{j,8}^j&\quad=-f^{(j-1)}(\rho)(1-\chi_\lambda(\rho)) \partial_\rho^j\left[\rho^{\frac{1-d}{2}}(1+\rho)^{s-\lambda-\frac{1}{2}}[1+e_2(\rho,\lambda)]r_2(\rho,\lambda)\right]
\\
&\quad \times \left[\sum_{\ell=1}^j I_{8,\ell,1}^j(\rho,\lambda)+B_{8,1}(\rho,\lambda)\right]
 \\
&\quad-f^{(j-1)}(\rho)(1-\chi_\lambda(\rho)) \partial_\rho^n\left[\rho^{\frac{1-d}{2}}(1+\rho)^{s-\lambda-\frac{1}{2}}[1+e_2(\rho,\lambda)]\right]
\\
&\quad \times \left[\sum_{\ell=1}^j(\rho,\lambda) I_{8,\ell,1}^j+B_{8,2}^j(\rho,\lambda)\right],
\end{align*}
where 
\begin{align*}
 I_{8,\ell,1}^j(\rho,\lambda)&= \frac{(-1)^{j-\ell+1}}{\prod_{b=1}^{j-\ell+1} (s-\lambda-\frac{1}{2}-b)(2s-2\lambda-1)}
 \\
 &\quad \times  \int_0^\rho\int_0^{t_\ell}\dots \int_0^{t_{\ell-1}}\frac{\partial_{t_\ell}\left(t_\ell^{\frac{d-1}{2}}(1-\chi_\lambda(t_\ell))[1+e_1(t_\ell,\lambda)][1+r_1(t_\ell,\lambda)]\right)}{(1+t_\ell)^{s-\lambda-\frac{3}{2}-j+\ell}}
dt_\ell\dots dt_2 dt_1 
\end{align*}
and 
\begin{align*}
B^j_{8,1}(\rho,\lambda)= (-1)^{j}\frac{\rho^{\frac{d-1}{2}}(1-\chi_\lambda(\rho))[1+e_1(\rho,\lambda)][1+r_1(\rho,\lambda)]}{\prod_{\ell=1}^{j} (s-\lambda-\frac{1}{2}-\ell) (2s-2\lambda-1)(1+\rho)^{s-\lambda-\frac{1}{2}-j}}
\end{align*}
and where $I_{8,\ell,2}^j$ and $B^j_{8,2}(\rho,\lambda)$ are defined analogously. The reason as to why we go through all of this trouble to spell involved terms out explicitly, is the following. Terms involving any of the integral terms $I^j_{a,\ell,b}$  
can be further integrated by parts to yield boundary terms which are schematically of the form  
$$
(1-\chi_\lambda(\rho))(1+\rho)^{\lambda}\O(\rho^{\frac{d-1}{2}-\ell}\langle\omega\rangle^{-1-\ell-j})
$$
or
$$
\partial_\rho\chi_\lambda(\rho)(1+\rho)^{\lambda}\O(\rho^{\frac{d+1}{2}-\ell}\langle\omega\rangle^{-1-\ell-j})
$$
and integral terms with high decay in $\omega$ that allows us to bound the corresponding oscillatory integrals in the well established fashion. However, the terms involving any of the $B^j_{i,\ell}$ only decay of order 2. Hence, we need to show an explicit cancellation. For this, we compute that
\begin{align*}
&\quad \partial_\rho^j\left[\rho^{\frac{1-d}{2}}(1-\rho)^{s-\lambda-\frac{1}{2}}[1+e_1(\rho,\lambda)][1+r_1(\rho,\lambda)]\right]
 \\
 &= (-1)^j\rho^{\frac{1-d}{2}}(1-\rho)^{s-\lambda-\frac{1}{2}-j}[1+e_1(\rho,\lambda)][1+r_1(\rho,\lambda)]\prod_{\ell=0}^{j-1}(s-\lambda-\tfrac{1}{2}-\ell)
 \\
 &\quad + j\partial_\rho^{j-1}\left((1-\rho)^{s-\lambda-\frac{1}{2}}\partial_\rho\left[\rho^{\frac{1-d}{2}}[1+e_2(\rho,\lambda)][1+r_2(\rho,\lambda)]\right]\right).
\end{align*}
Hence, 
\begin{align*}
&\quad(1-\chi_\lambda(\rho))  \partial_\rho^j\left[\rho^{\frac{1-d}{2}}(1-\rho)^{s-\lambda-\frac{1}{2}}[1+e_1(\rho,\lambda)][1+r_1(\rho,\lambda)]\right]B^j_{7,1}(\rho,\lambda)
\\
&=(-1)^{j}\frac{\prod_{\ell=0}^{j-1} (s-\lambda-\frac{1}{2}-\ell)}{(2s-2\lambda-1)\prod_{\ell=1}^{j} (s-\lambda-\frac{1}{2}-\ell) }
\\
&\quad \times (1-\chi_\lambda(\rho))^2[1+e_1(\rho,\lambda)][1+r_1(\rho,\lambda)][1+e_2(\rho,\lambda)]r_2(\rho,\lambda)
\\
&\quad +(1-\chi_\lambda(\rho))^2\sum_{\ell=1}^j \O(\rho^{-\ell}\langle\omega\rangle^{-2-\ell})+\partial_\rho\chi_\lambda(\rho)\sum_{\ell=1}^j \O(\rho^{-\ell+1}\langle\omega\rangle^{-2-\ell}).
\end{align*}
In the same way we conclude that
\begin{align*}
&\quad-(1-\chi_\lambda(\rho))  \partial_\rho^j\left[\rho^{\frac{1-d}{2}}(1+\rho)^{s-\lambda-\frac{1}{2}}[1+e_2(\rho,\lambda)]r_2(\rho,\lambda)\right]
B_{8,2}^j(\rho,\lambda)
\\
&=\frac{(-1)^{j+1}\prod_{\ell=0}^{j-1} (s-\lambda-\frac{1}{2}-\ell)}{(2s-2\lambda-1)\prod_{\ell=1}^{j} (s-\lambda-\frac{1}{2}-\ell) }
\\
&\quad \times (1-\chi_\lambda(\rho))^2[1+e_1(\rho,\lambda)][1+r_1(\rho,\lambda)][1+e_2(\rho,\lambda)]r_2(\rho,\lambda)
\\
&\quad +(1-\chi_\lambda(\rho))^2\sum_{\ell=1}^j \O(\rho^{-\ell}\langle\omega\rangle^{-2-\ell})+\partial_\rho\chi_\lambda(\rho)\sum_{\ell=1}^j \O(\rho^{-\ell+1}\langle\omega\rangle^{-2-\ell}).
\end{align*}
Consequently, 
\begin{align*}
&\quad(1-\chi_\lambda(\rho))  \partial_\rho^j\left[\rho^{\frac{1-d}{2}}(1-\rho)^{s-\lambda-\frac{1}{2}}[1+e_1(\rho,\lambda)][1+r_1(\rho,\lambda)]\right]B^j_{7,1}(\rho,\lambda)
\\
&\quad -(1-\chi_\lambda(\rho))  \partial_\rho^j\left[\rho^{\frac{1-d}{2}}(1+\rho)^{s-\lambda-\frac{1}{2}}[1+e_2(\rho,\lambda)]r_2(\rho,\lambda)\right]
B_{8,2}^j(\rho,\lambda)
\\
&=(1-\chi_\lambda(\rho))^2\sum_{\ell=1}^j \O(\rho^{-\ell}\langle\omega\rangle^{-2-\ell})+\partial_\rho\chi_\lambda(\rho)\sum_{\ell=1}^j \O(\rho^{-\ell+1}\langle\omega\rangle^{-2-\ell}).
\end{align*} 
By applying the same cancellation argument to remaining boundary terms and performing further integrations by parts, one deduces that $H^{n}_{j,7}+H^{n}_{j,8}$ decays of cubic order for all $n\geq j$. Thus, the desired estimates on $ |.|^{-m}T_{j,7}^{n-m,1}+ |.|^{-m}T_{j,8}^{n-m,1}$ and $  |.|^{-m}\dot T_{j,7}^{n-m,1}+ |.|^{-m}\dot T_{j,8}^{n-m,1}$ for all remaining $n,m$ follow from Lemmas \ref{osci3} and the technical Lemmas on weighted norms \ref{teclem2} and \ref{teclem3}.
So, we, at last, come to the estimates  \begin{align*}
\|T_{j,\ell}^{0,0}(f)\|_{L^{\frac{2}{1-\frac{2\delta}{1-\theta}}}(\R_+)L^\infty(\B^d_1)}\lesssim \| f\|_{W^{k-2,\frac{2}{1-\frac{2\delta}{1-\theta}}}(\B^d_1)}
\end{align*}
and 
\begin{align*}
\|\dot{T}_{j,\ell}^{0,0}(f)\|_{L^{\frac{2}{1-\frac{2\delta}{1-\theta}}}(\R_+)L^\infty(\B^d_1)}\lesssim \| f\|_{W^{k-1,\frac{2}{ 1-\frac{2\delta}{1-\theta}}}(\B^d_1)}.
\end{align*}
For $\ell=0$, another application of Lemma \ref{osci1} yields
\begin{align*}
|T_{j,\ell}^{0,0}(f)(\rho,\tau)|&\lesssim_\varepsilon \langle\tau\rangle^{-2} \rho^{j+1-\varepsilon}|f^{(j-1)}(\rho)|.
\end{align*}
Moreover, if $\varepsilon$ is chosen small enough, then
$$
\|(.)^{-\frac{1}{2}-\varepsilon}\|_{L^{\frac{2(1-\theta)}{1-\theta+2\delta}}((0,1))},
$$
which implies that
\begin{align*}
\rho^{j+1-\varepsilon}|f^{(j-1)}(\rho)|&\lesssim \left|\int_0^\rho f^{(j)}(\rho) \rho^{j+1-\varepsilon} d\rho+ \int_0^\rho f^{(j)}(\rho) \rho^{j-\varepsilon} d\rho\right|
\\
&\lesssim \|f^{(j)}(\rho) \rho^{j+\frac{3}{2}}\|_{L^\frac{2}{1-2\frac{\delta}{1\theta}}((0,1))}+ \|f^{(j-1)}(\rho) \rho^{j+\frac{1}{2}}\|_{L^\frac{2}{1-2\frac{\delta}{1\theta}}((0,1))}
\\
&\lesssim
\|f\|_{W^{k-2,\frac{2}{ 1-\frac{2\delta}{1-\theta}}}(\B^d_1)}+\|f^{(k-2)}(\rho) \rho^{k -\frac{1}{2}}\|_{L^\frac{2}{1-2\frac{\delta}{1\theta}}((0,1))}
\\
&=\|f\|_{W^{k-2,\frac{2}{ 1-\frac{2\delta}{1-\theta}}}(\B^d_1)}+\|f^{(k-2)}(\rho) \rho^{\frac{d-1}{2}}\|_{L^\frac{2}{1-2\frac{\delta}{1\theta}}((0,1))} \lesssim \|f\|_{W^{k-2,\frac{2}{ 1-\frac{2\delta}{1-\theta}}}(\B^d_1)}.
\end{align*}
In the same fashion, one proves the estimate on $\dot{T}_{j,\ell}^{0,0}$ and by combining this reasoning with the strategies employed in this proof, one bounds the remaining operators.
\end{proof}
Again, we state the corresponding Lemma for the integer regularity case, but omit the proof as it is essentially the same.   
\begin{lem}\label{lem:strichartz2b}
Let $3\leq d \in \mathbb{N}$ and $ s \in \mathbb{N}$ with $1\leq s < \frac{d}{2}$.
Then, the estimates 
\begin{align*}
\left\||.|^{-m}\sum_{\ell=4}^8 T_{j,\ell}^{n-m,1}f\right\|_{L^p(\R_+)L^q(\B^d_1)}\lesssim \| f\|_{W^{s-1,\frac{2}{1 +2\delta}}(\B^d_1)}
\end{align*}
and 
\begin{align*}
\left\||.|^{-m} \sum_{\ell=4}^8 \dot{T}_{j,\ell}^{n-m,1}f\right\|_{L^p(\R_+)L^q(\B^d_1)}\lesssim \| f\|_{W^{s,\frac{2}{1 +2\delta}}(\B^d_1)}
\end{align*}
hold for all  $j,n,m\in \mathbb{Z}$ with $1\leq j\leq s-1$, with $0\leq n \leq s$, $0\leq m<n$,
 $p,q\in [\frac{2}{1 +2\delta},\infty]$ such that the scaling relation  $$\frac{1}{p}+\frac{d}{q}=d(\frac{1}{2} +\delta)-s+n$$ 
is satisfied, as well as all $f\in C^\infty_{rad}(\overline{\B^d_1})$.
 Furthermore, the estimates 
\begin{align*}
\||.|^{-m} \sum_{\ell=4}^8 T_{j,\ell}^{n-m,0}f\|_{L^p(\R_+)L^q(\B^d_1)}\lesssim \| f\|_{W^{s-1,\frac{2}{1-2\delta}}(\B^d_1)}
\end{align*}
and 
\begin{align*}
\||.|^{-m} \sum_{\ell=4}^8 \dot{T}_{j,\ell}^{n-m,0}f\|_{L^p(\R_+)L^q(\B^d_1)}\lesssim \| f\|_{W^{s,\frac{2}{1-2\delta}}(\B^d_1)}
\end{align*}
hold for the same range of $j,n,m$ and all $ p,q \in [\frac{2}{1-2\delta}, \infty ]  $ such that the scaling relation  $$\frac{1}{p}+\frac{d}{q}=d(\frac12-\delta)-s+n$$ is satisfied, as well as all $f\in C^\infty_{rad}(\overline{\B^d_1})$. Lastly, in case $s=\frac{d-1}{2}$, the estimates
 \begin{align*}
\|T_{j,\ell}^{0,1}f\|_{L^{\frac{2}{1+2\delta}}(\R_+)L^\infty(\B^d_1)}&\lesssim \| f\|_{W^{s-1,\frac{2}{1+2\delta}}(\B^d_1)}
\\
\| \dot{T}_{j,\ell}^{0,1}f\|_{L^{\frac{2}{1+2\delta}}(\R_+)L^\infty(\B^d_1)}&\lesssim \| f\|_{W^{s,\frac{2}{1+2\delta}}(\B^d_1)}
\end{align*}
and
 \begin{align*}
\|T_{j,\ell}^{0,0}f\|_{L^{\frac{2}{1-2\delta}}(\R_+)L^\infty(\B^d_1)}&\lesssim \| f\|_{W^{s-1,\frac{2}{1-2\delta}}(\B^d_1)}
\\
\| \dot{T}_{j,\ell}^{0,0}f\|_{L^{\frac{2}{1-2\delta}}(\R_+)L^\infty(\B^d_1)}&\lesssim \| f\|_{W^{s,\frac{2}{1-2\delta}}(\B^d_1)}
\end{align*}
hold for $j=1,\dots ,k-1$ and $\ell=3,\dots,8$.
\end{lem}

Finally,  we come to 
\begin{align*}
&\quad W_{k}^n(f)(\rho,\lambda)
\\
&:= (-1)^k\partial_\rho^n  u_0(\rho,\lambda)\int_\rho^{1}\int_0^{t_1}\int_0^{t_2}\dots \int_0^{t_{k-1}}\frac{u_1(t_k,\lambda)t_k^{d-1}}{(1-t_k^2)^{s-\lambda-\frac12}}dt_k \dots d t_3 dt_2 f^{(k-1)}(t_1) dt_1
\\
&\quad -  (-1)^k \partial_\rho^nu_{\mathrm{f}_0}(\rho,\lambda)\int_\rho^{1}\int_0^{t_1}\int_0^{t_2}\dots \int_0^{t_{k-1}}\frac{u_{\mathrm{f}_1}(t_k,\lambda)t_k^{d-1}}{(1-t_k^2)^{s-\lambda-\frac12}}dt_k \dots d t_3 dt_2 f^{(k-1)}(t_1) dt_1
\\
&\quad+ (-1)^k\partial_\rho^nu_1(\rho,\lambda)\int_0^\rho\int_0^{t_1}\int_0^{t_2}\dots \int_0^{t_{k-1}}\frac{u_0(t_k,\lambda)t_k^{d-1}}{(1-t_k^2)^{s-\lambda-\frac12}}dt_k \dots d t_3 dt_2 f^{(k-1)}(t_1) dt_1
\\
&\quad- (-1)^k\partial_\rho^nu_{\mathrm{f}_1}(\rho,\lambda)\int_0^\rho\int_0^{t_1}\int_0^{t_2}\dots \int_0^{t_{k-1}}\frac{u_{\mathrm{f}_0}(t_k,\lambda)t_k^{d-1}}{(1-t_k^2)^{s-\lambda-\frac12}}dt_k \dots d t_3 dt_2 f^{(k-1)}(t_1) dt_1
\end{align*}
for all $n\in \mathbb Z  $ with $0\leq n \leq k$
and
\begin{align*}
&\quad W_{k-1}^n(f)(\rho,\lambda)
\\
&:=(-1)^{k-1}\partial_\rho^n u_0(\rho,\lambda)\int_\rho^{1}\int_0^{t_1}\int_0^{t_2}\dots \int_0^{t_{k-2}}\frac{u_1(t_{k-1},\lambda)t_{k-1}^{d-1}}{(1-t_{k-1}^2)^{s-\lambda-\frac12}}dt_{k-1} \dots d t_3 dt_2 f^{(k-2)}(t_1) dt_1
\\
&\quad - (-1)^{k-1} \partial_\rho^nu_{\mathrm{f}_0}(\rho,\lambda)\int_\rho^{1}\int_0^{t_1}\int_0^{t_2}\dots \int_0^{t_{k-2}}\frac{u_{\mathrm{f}_1}(t_{k-1},\lambda)t_{k-1}^{d-1}}{(1-t_{k-1}^2)^{s-\lambda-\frac12}}dt_{k-1} \dots d t_3 dt_2 f^{(k-2)}(t_1) dt_1
\\
&\quad+(-1)^{k-1}\partial_\rho^nu_1(\rho,\lambda)\int_0^\rho\int_0^{t_1}\int_0^{t_2}\dots \int_0^{t_{k-2}}\frac{u_0(t_{k-1},\lambda)t_{k-1}^{d-1}}{(1-t_{k-1}^2)^{s-\lambda-\frac12}}dt_{k-1} \dots d t_3 dt_2 f^{(k-2)}(t_1) dt_1
\\
&\quad-(-1)^{k-1}\partial_\rho^nu_{\mathrm{f}_1}(\rho,\lambda)\int_0^\rho\int_0^{t_1}\int_0^{t_2}\dots \int_0^{t_{k-2}}\frac{u_{\mathrm{f}_0}(t_{k-1},\lambda)t_{k-1}^{d-1}}{(1-t_{k-1}^2)^{s-\lambda-\frac12}}dt_{k-1} \dots d t_3 dt_2 f^{(k-2)}(t_1) dt_1
\end{align*}
in case $ s \notin \mathbb{N} $ and $\Re \lambda= \mu_0 $.
Bounding the associated oscillatory integrals will require cancellations similar to the ones above. Hence, we decompose it as follows.
\begin{lem}
We can decompose $W_{k}^n (f)(\rho,\lambda)$ as
\begin{align*}
W_{k}^n(f)(\rho,\lambda)=(-1)^k\sum_{\ell=9}^{17}H_{k,l}^n(f)(\rho,\lambda)
\end{align*}
with
\begin{align*}
H_{k,9}^n(f)(\rho,\lambda)&=\chi_\lambda(\rho)\partial_\rho^n\left((1-\rho^2)^{\frac{s}{2}-\frac{3}{4}-\frac{\lambda}{2}}\rho^{\frac{1-d}{2}}b_1(\rho,\lambda)[1+\rho^2 e_3(\rho,\lambda)]\right)
\\
&\quad \times
\int_\rho^{1}\int_0^{t_1}\int_0^{t_2}\dots \int_0^{t_{k-1}}\chi_\lambda(t_k)\frac{\O(t_k\langle\omega\rangle^{-\frac{d}{2}})}{(1-t_k^2)^{\frac{s}{2}+\frac{1}{4}-\frac{\lambda}{2}}}dt_k \dots d t_3 dt_2 f^{(k-1)}(t_1) dt_1
\\
&\quad + P_{k+1,9}^n(f)(\rho,\lambda)
\\
H_{k,10}^n(f)(\rho,\lambda)&=\chi_\lambda(\rho)\partial_\rho^n\left((1-\rho^2)^{\frac{s}{2}-\frac{3}{4}-\frac{\lambda}{2}}\rho^{\frac{1-d}{2}}b_1(\rho,\lambda)[1+\rho^2 e_3(\rho,\lambda)]\right)
\\
&\quad \times
\int_\rho^{1}\int_0^{t_1}\int_0^{t_2}\dots \int_0^{t_{k-1}}(1-\chi_\lambda(t_k))
\\
&\quad \times \frac{t_k^{\frac{d-1}{2}}\O(\langle\omega\rangle^{-\frac{3}{2}})[1+e_2(t_k,\lambda)]}{(1-t_k)^{s-\lambda-\frac12}}dt_k \dots d t_3 dt_2 f^{(k-1)}(t_1) dt_1
 + P_{k+1,10}(f)(\rho,\lambda)
\\
H_{k,11}^n(f)(\rho,\lambda)&=(1-\chi_\lambda(\rho))\partial_\rho^n\left(\rho^{\frac{1-d}{2}}(1-\rho)^{s-\lambda-\frac{1}{2}}[1+(1-\rho)\O(\rho^{-1}\langle\omega\rangle^{-1})]\right)
\\
&\quad \times
\int_\rho^{1}\int_0^{t_1}\int_0^{t_2}\dots \int_0^{t_{k-1}}\chi_\lambda(t_k)\frac{\O(t_k\langle\omega\rangle^{-\frac{d+1}{2}})}{(1-t_k^2)^{\frac{s}{2}+\frac{1}{4}-\frac{\lambda}{2}}}dt_k \dots d t_3 dt_2 f^{(k-1)}(t_1) dt_1
\\
&\quad + P_{k,11}^n(f)(\rho,\lambda),
\end{align*}
and
\begin{align*}
H_{k,12}^n(f)(\rho,\lambda)&=(1-\chi_\lambda(\rho))\partial_\rho^n \left(\rho^{\frac{1-d}{2}}(1+\rho)^{s-\lambda-\frac{1}{2}}[1+(1-\rho)\O(\rho^{-1}\langle\omega\rangle^{-1})]\right)
\\
&\quad \times
\int_\rho^{1}\int_0^{t_1}\int_0^{t_2}\dots \int_0^{t_{k-1}}\chi_\lambda(t_k)\frac{\O(t_k\langle\omega\rangle^{-\frac{d+1}{2}})}{(1-t_k^2)^{\frac{s}{2}+\frac{1}{4}-\frac{\lambda}{2}}}dt_k \dots d t_3 dt_2 f^{(k-1)}(t_1) dt_1
\\
&\quad + P_{k,12}^n(f)(\rho,\lambda)
\\
H_{k,13}^n(f)(\rho,\lambda)&=(1-\chi_\lambda(\rho))\partial_\rho^n\left(\rho^{\frac{1-d}{2}}(1-\rho)^{s-\lambda-\frac{1}{2}}[1+e_1(\rho,\lambda)][1+r_1(\rho,\lambda)]\right)
\\
&\quad \times
\int_\rho^{1}\int_0^{t_1}\int_0^{t_2}\dots \int_0^{t_{k-1}}(1-\chi_\lambda(t_k))
\\
&\quad \times \frac{t_k^{\frac{d-1}{2}}[1+e_2(t,\lambda)]r_2(t,\lambda)}{(2s-2\lambda-1)(1-t_k)^{s-\lambda-\frac12}}dt_k \dots d t_3 dt_2 f^{(k-1)}(t_1) dt_1
\\
&\quad +
(1-\chi_\lambda(\rho))\partial_\rho^n \left(\rho^{\frac{1-d}{2}}(1-\rho)^{s-\lambda-\frac{1}{2}}[1+e_1(\rho,\lambda)]r_1(\rho,\lambda)\right)
\\
&\quad \times
\int_\rho^{1}\int_0^{t_1}\int_0^{t_2}\dots \int_0^{t_{k-1}}(1-\chi_\lambda(t_k))
\\
&\quad \times \frac{t_k^{\frac{d-1}{2}}[1+e_2(t,\lambda)]}{(2s-2\lambda-1)(1-t_k)^{s-\lambda-\frac12}}dt_k \dots d t_3 dt_2 f^{(k-1)}(t_1) dt_1
\end{align*}
and
\begin{align*}
H_{k,14}^n(f)(\rho,\lambda)&=\chi_\lambda(\rho)\partial_\rho^n\left((1-\rho^2)^{\frac{s}{2}-\frac{3}{4}-\frac{\lambda}{2}}\O(\rho^0\langle\omega\rangle^0)\right)
\\
&\quad \times
\int_0^\rho\int_0^{t_1}\int_0^{t_2}\dots \int_0^{t_{k-1}}\frac{\O (t_k^{d-1}\langle\omega\rangle^{d-3})}{(1-t_k^2)^{\frac{s}{2}+\frac{1}{4}-\frac{\lambda}{2}}}dt_k \dots d t_3 dt_2 f^{(k-1)}(t_1) dt_1
\\
H_{k,15}^n(f)(\rho,\lambda)&=(1-\chi_\lambda(\rho))\partial_\rho^n \left(\rho^{\frac{1-d}{2}}(1+\rho)^{s-\lambda-\frac{1}{2}}[1+(1-\rho)\O(\rho^{-1}\langle\omega\rangle^{-1})]\right)
\\
&\quad \times
\int_0^\rho\int_0^{t_1}\int_0^{t_2}\dots \int_0^{t_{k-1}}\chi_\lambda(t_k)\frac{\O(t_k^{d-1}\langle\omega\rangle^{\frac{d-5}{2}})}{(1-t_k^2)^{\frac{s}{2}+\frac{1}{4}-\frac{\lambda}{2}}}dt_k \dots d t_3 dt_2 f^{(k-1)}(t_1) dt_1
\\
&\quad + P_{k,16}^n(f)(\rho,\lambda)
\\
H_{k,16}^n(f)(\rho,\lambda)&=(1-\chi_\lambda(\rho))\partial_\rho^n\left(\rho^{\frac{1-d}{2}}(1+\rho)^{s-\lambda-\frac{1}{2}}[1+e_2(\rho,\lambda)]r_2(\rho,\lambda)\right)
\\
&\quad \times
\int_0^\rho\int_0^{t_1}\int_0^{t_2}\dots \int_0^{t_{k-1}}(1-\chi_\lambda(t_k))
\\
&\quad\frac{t_k^{\frac{d-1}{2}}[1+e_1(t_k,\lambda)][1+r_1(t_k,\lambda)]}{(2s-2\lambda-1)(1+t_k)^{s-\lambda-\frac12}}dt_k \dots d t_3 dt_2 f^{(k-1)}(t_1) dt_1
\\
&\quad +(1-\chi_\lambda(\rho))\partial_\rho^n\left(\rho^{\frac{1-d}{2}}(1+\rho)^{s-\lambda-\frac{1}{2}}[1+e_2(\rho,\lambda)]\right)
\\
&\quad \times
\int_0^\rho\int_0^{t_1}\int_0^{t_2}\dots \int_0^{t_{k-1}}(1-\chi_\lambda(t_k))
\\
&\quad\times \frac{t_k^{\frac{d-1}{2}}[1+e_1(t_k,\lambda)]r_1(t_k,\lambda)}{(2s-2\lambda-1)(1+t_k)^{s-\lambda-\frac12}}dt_k \dots d t_3 dt_2 f^{(k-1)}(t_1) dt_1,
\end{align*}
and
\begin{align*}
H_{k,17}^n(f)(\rho,\lambda)&=(1-\chi_\lambda(\rho))\partial_\rho^n\left(\rho^{\frac{1-d}{2}}(1+\rho)^{s-\lambda-\frac{1}{2}}[1+(1-\rho)\O(\rho^{-1}\langle\omega\rangle^{-1})]\right)
\\
&\quad \times
\int_0^1\int_0^{t_1}\int_0^{t_2}\dots \int_0^{t_{k-1}}(1-\chi_\lambda(t_k))
\\
&\frac{t_k^{\frac{d-1}{2}}\O(\langle\omega\rangle^{-2})[1+e_3(t_k,\lambda)]}{(1-t_k)^{s-\lambda-\frac12}}dt_k \dots d t_3 dt_2 f^{(k-1)}(t_1) dt_1
\\
&\quad + P_{k,17}^n(f)(\rho,\lambda)
\end{align*}
where the terms $P_{k,j}^n(f)$ are of the same form as the corresponding leading order term, but with better behavior in either $\rho$ or $\omega$. 
Furthermore, in case $s \notin \mathbb{N}$ and $\Re \lambda=\mu_0$, one can decompose $W^n_{k-1}$ in the same way with the only difference being that $k$ gets replaced by $k-1$.
\end{lem}
\begin{proof}
One arrives at this decomposition in the same way as the previous ones. The only notable difference are the kernels $H_{k,17}^n$, due to the integral from $0$  to $1$ contained in them. One obtains these kernels, as we split one integral term from $0$ to $\rho$ and one from $\rho$ to $1$. 
\end{proof}
Some remarks on this decomposition are in order. First, one notices that  
for the kernel $H_{k,15}$ we employ the slightly more crude symbol representation then for e.g. in the kernel $H_{k,9}$. This stems from the fact that symbol representation alone is not good enough, when we are unable to exchange positive powers of $t_1$ for powers of $\rho$. Furthermore, for the kernels $H_{k,13}^k$ and $H_{k,16}^k$ we need this explicit form, to once more obtain a cancellation.
As above, we define operators $T_{k,\ell}^{n,a}$ and $\dot{T}_{k,\ell}^{n,a}$ 
associated to $ W_k$ as
\begin{align*}
T_{k,\ell}^{n,a}f(\tau,\rho)&= \int_\R e^{i \omega\tau}f(\rho)H_{k,\ell}^n(\rho,\mu_a+ i \omega) d\omega
\end{align*}
and
\begin{align*}
\dot{T}_{k,\ell}^{n,a}f(\tau,\rho)&= \int_\R \omega e^{i \omega\tau}f(\rho)H_{k,\ell}^n(\rho,\mu_a +i \omega) d\omega
\end{align*}
for $0\leq n \leq k$ and $9\leq \ell \leq 17$ and $a=1$ in case $s \notin \mathbb{N}$ and $a=0,1$ in case it is. Analogously, we define the operators  
\begin{align*}
T_{k-1,\ell}^{n,0}f(\tau,\rho)&= \int_\R e^{i \omega\tau} H_{k-1,\ell}^n(f)(\rho, \mu_0+ i \omega)d\omega
\end{align*}
and
\begin{align*}
\dot T_{k-1,\ell}^{n,0}f(\tau,\rho)&= \int_\R \omega  e^{i \omega\tau}  H_{k-1,\ell}^n(f)(\rho,\mu_0+i \omega) d\omega
\end{align*}
$0\leq n \leq k-1$, provided $s \notin \mathbb{N}$.
\begin{lem} \label{lem:strich3a}
Let $ 3\leq d\in \mathbb{N}$ and $1\leq s \notin \mathbb{N}$ with $1\leq k=\ceil s \leq \frac{d}{2}$. Then, the estimates
\begin{align*}
\left\|\sum_{\ell=9}^{17}|.|^{-m}T_{k,\ell}^{n-m,1}f\right\|_{L^p(\R_+)L^q(\B^d_1)}\lesssim \| f\|_{W^{k-1,\frac{2}{1+2\frac{\delta}{\theta}}}(\B^d_1)}
\end{align*}
and 
\begin{align*}
\left\|\sum_{\ell=9}^{17}|.|^{-m}\dot{T}_{k,\ell}^{n-m,1}f\right\|_{L^p(\R_+)L^q(\B^d_1)}\lesssim \| f\|_{W^{k,\frac{2}{1+2\frac{\delta}{\theta}}}(\B^d_1)}
\end{align*}
hold for all  $m,n\in \mathbb{Z}$ with $0\leq m< n$, $0 \leq n \leq k$, $p,q\in [\frac{2}{1+2\frac{\delta}{\theta}},\infty]$ such that the scaling relation $$\frac{1}{p}+\frac{d}{q}=d(\frac{1}{2}+\frac{\delta}{\theta})-k+n$$ 
is satisfied, as well as all $f\in C^\infty_{rad}(\overline{\B^d_1})$.
Moreover, if $s>\frac{d-1}{2}$, then, also the estimates
\begin{align*}
\|\sum_{\ell=9}^{17} T_{k,\ell}^{0,1}(f)\|_{L^{\frac{2}{1+2\frac{\delta}{\theta}}}(\R_+)L^\infty(\B^d_1)}\lesssim \| f\|_{W^{k-1,\frac{2}{1+2\frac{\delta}{\theta}}}(\B^d_1)}
\end{align*}
and 
\begin{align*}
\|\sum_{\ell=9}^{17} \dot{T}_{k,\ell}^{0,1}(f)\|_{L^{\frac{2}{1+2\frac{\delta}{\theta}}}(\R_+)L^\infty(\B^d_1)}\lesssim \| f\|_{W^{k,\frac{2}{1+2\frac{\delta}{\theta}}}(\B^d_1)}
\end{align*}
hold.
Similarly, the estimates
\begin{align*}
\left\|\sum_{\ell=9}^{17}|.|^{-m}T_{k-1,\ell}^{n-m,0}f\right\|_{L^p(\R_+)L^q(\B^d_1)}\lesssim \| f\|_{W^{k-2,\frac{1}{\frac{1}{2}+\mu_0-\theta}}(\B^d_1)}
\end{align*}
and 
\begin{align*}
\left\|\sum_{\ell=9}^{17}|.|^{-m}\dot{T}_{k-1,\ell}^{n-m,0}f\right\|_{L^p(\R_+)L^q(\B^d_1)}\lesssim \| f\|_{W^{k-1,\frac{1}{\frac{1}{2}+\mu_0-\theta}}(\B^d_1)}
\end{align*}
hold for all  $m,n\in \mathbb{Z}$ with $0\leq m< n$, $0 \leq n \leq k-1$, $ p,q \in [\frac{1}{\frac 12 -\frac{\delta}{1-\theta}},\infty]$ such that the scaling relation $$\frac{1}{p}+\frac{d}{q}=d(\frac 12 -\frac{\delta}{1-\theta})-k+1+n$$
is satisfied, as well as all $f\in C^\infty_{rad}(\overline{\B^d_1})$.
Lastly, for $s >\frac{d-1}{2}$ the estimates
 \begin{align*}
\|T_{k-1,\ell}^{0,0}(f)\|_{L^{\frac{2}{1-\frac{2\delta}{1-\theta}}}(\R_+)L^\infty(\B^d_1)}\lesssim \| f\|_{W^{k-2,\frac{2}{1-\frac{2\delta}{1-\theta}}}(\B^d_1)}
\end{align*}
and 
\begin{align*}
\|\dot{T}_{k-1,\ell}^{0,0}(f)\|_{L^{\frac{2}{1-\frac{2\delta}{1-\theta}}}(\R_+)L^\infty(\B^d_1)}\lesssim \| f\|_{W^{k-1,\frac{2}{ 1-\frac{2\delta}{1-\theta}}}(\B^d_1)}
\end{align*}
hold for $\ell=9,\dots 17$ and all $f\in C^\infty_{rad}(\overline{\B^d_1})$.
\end{lem}

\begin{proof}
Aside from the $L^pL^\infty$ estimate, which we prove at the end, we again only establish the estimates for $a=1$. We start with $|.|^{-m}T_{k,9}^{n-m,1}$ and rewrite this expression as
\begin{align*}
\rho^{-m}T_{k,9}^{n-m,1}(f)(\tau,\rho)&=\int_\rho^{1}\int_0^{t_1}\int_0^{t_2}\dots \int_0^{t_{k-1}}\int_\R e^{i \omega \tau }\chi_{\mu_1+i\omega}(\rho)(1-\rho^2)^{\frac{s}{2}-\frac{3}{4}-\frac{\mu_1+i\omega}{2}}
\\
&\quad \times
\chi_{\mu_a+i\omega}(t_k)\frac{\O(\rho^{-n +\frac{1}{8}}t_k^{\frac{3}{4}}\langle\omega\rangle^{-1-\frac{1}{8}})}{(1-t_k^2)^{\frac{s}{2}+\frac{1}{4}-\frac{\mu_1+i\omega}{2}}}d \omega dt_k \dots d t_3 dt_2 f^{(k-1)}(t_1) dt_1.
\end{align*}
Thus,
\begin{align*}
|\rho^{-m}T_{k,9}^{n-m,1}(f)(\tau,\rho)|&\lesssim \langle\tau \rangle^{-2} \rho^{-n+\frac{1}{8}}\int_\rho^{\rho_1}\int_0^{t_1}\int_0^{t_2}\dots \int_0^{t_{k-1}}t_k^{\frac{3}{4}} dt_k \dots d t_3 dt_2 |f^{(k-1)}(t_1)| dt_1
\\
&\lesssim \langle\tau \rangle^{-2} \rho^{-n+\frac{1}{8}}\int_\rho^{1}t_1^{\frac{3}{4}+k-1} |f^{(k-1)}(t_1)| dt_1
\\
&\lesssim \langle\tau \rangle^{-2} \rho^{-n+\frac{1}{8}+k-\frac{d-1}{2}}\int_\rho^{1}t_1^{\frac{3}{4}+\frac{d-3}{2}} |f^{(k-1)}(t_1)| dt_1
\\
&\lesssim \rho^{-n+\frac{1}{8}+k-\frac{d-1}{2}}\|f^{(k-1)}\|_{L^{\frac{2}{1+2\frac{\delta}{\theta}}}(\B^d_1)}\left(\int_0^1\rho^{-\frac{1}{4}(2+\frac{2\delta}{\theta+2\delta})}d\rho\right)^{\frac{1}{2}+\frac{\delta}{\theta}}
\\
&\lesssim \rho^{-n+\frac{1}{8}+k-\frac{d-1}{2}}\|f\|_{W^{k-1,\frac{2}{1+2\frac{\delta}{\theta}}}(\B^d_1)}.
\end{align*}
Now, if $k-n\geq \frac{d-1}{2}$, then $-n+\frac{1}{8}+k-\frac{d-1}{2}$ is positive and the desired estimates follow. If not, we set $r=\frac{1}{2}+\frac{\delta}{\theta}$ and recall that the highest spatial norm which we need to estimate is given by
\begin{align*}
\|.\|_{L^{\frac{d}{dr -k+n-r}}(\B^d_1)}.
\end{align*}
Now, if $dr-k+n-r\neq 0$ neither is $d-1-2k+2n$ provided that $\delta $ is sufficiently small. Thus, one computes 
\begin{align*}
\frac{2d}{d-1-2k+2n}(-n+k+\frac{1}{8}-\frac{d-1}{2}) +d-1=-1+\frac{d}{4(d-1-2k+2n)}
\end{align*}
which implies that 
\begin{align*}
&\quad \frac{d}{d(\frac{1}{2}+\frac{\delta}{\theta})-k+n-\frac{1}{2}+\frac{\delta}{\theta}}(-n+k+\frac{1}{8}-\frac{d-1}{2}) +d-1
\\
&=-1+\frac{d}{4(d-1-2k+2n)}
\\
&\quad+\left(\frac{d}{d(\frac{1}{2}+\frac{\delta}{\theta})-k+n-\frac{1}{2}+\frac{\delta}{\theta}}-\frac{2d}{d-1-2k+2n}\right)(-n+k+\frac{1}{8}-\frac{d-1}{2}).
\end{align*}
Now,
\begin{align*}
&\quad\frac{d}{d(\frac{1}{2}+\frac{\delta}{\theta})-k+n-\frac{1}{2}+\frac{\delta}{\theta}}-\frac{2d}{d-1-2k+2n}
\\
&=\frac{-2d(d-1)\frac{\delta}{\theta}}{(dr-k+n-r)(d-1-2k+2n)}=:\widetilde\delta  
\end{align*}
which ensures that
\begin{align*}
\||.|^{-n+k+\frac{1}{8}-\frac{d-1}{2}}\|_{L^{\frac{d}{dr-k+n-r}}(\B^d_1)}^{\frac{d}{dr-k+n-r}}&\lesssim \int_0^1 \rho^{-1-\widetilde \delta+\frac{d}{4(d-1-2k+2n)}} d\rho <\infty
\end{align*}
provided $\delta$ is chosen sufficiently small.
Similarly, one estimates $|.|^{-m}\dot{T}_{k,9}^{n-m,1}$ and we move on to $|.|^{-m} T_{k,10}^{n-m,1}$. Estimating $|.|^{-m} T_{k,10}^{0,1}$ is straightforward, so we study the more involved $|.|^{-m}T_{k,10}^{n-m,1}$. For this, we first have to manipulate the integral term
\begin{align*}
\int_0^{t_1}\int_0^{t_2}\dots \int_0^{t_{k-1}}(1-\chi_\lambda(t_k))\frac{t_k^{\frac{d-1}{2}}\O(\langle\omega\rangle^{-\frac{3}{2}})[1+e_3(t_k,\lambda)]}{(1-t_k)^{s-\lambda-\frac12}}dt_k \dots d t_3 dt_2 
\end{align*} 
appropriately. Integrating by parts yields
\begin{align*}
&\quad\int_0^{t_1}\int_0^{t_2}\dots \int_0^{t_{k-1}}(1-\chi_\lambda(t_k))\frac{t_k^{\frac{d-1}{2}}\O(\langle\omega\rangle^{-\frac{3}{2}})[1+e_2(t_k,\lambda)]}{(1-t_k)^{s-\lambda-\frac12}}dt_k \dots d t_3 dt_2 
\\
&= \int_0^{t_1}\int_0^{t_2}\dots \int_0^{t_{k-2}}(1-\chi_\lambda(t_{k-1}))\frac{t_{k-1}^{\frac{d-1}{2}}\O(\langle\omega\rangle^{-\frac{5}{2}})[1+e_2(t_{k-1},\lambda)]}{(1-t_{k-1})^{s-\lambda-\frac32}}dt_{k-1} \dots d t_3 dt_2 
\\
&\quad+\int_0^{t_1}\int_0^{t_2}\dots \int_0^{t_{k-1}}\O(\langle\omega\rangle^{-\frac{5}{2}})\frac{\partial_{t_k}[(1-\chi_\lambda(t_k))t_k^{\frac{d-1}{2}}[1+e_2(t_k,\lambda)]]}{(1-t_k)^{s-\lambda-\frac32}}dt_k \dots d t_3 dt_2 
\end{align*} 
which we iterate to obtain
\begin{align*}
&\quad\int_0^{t_1}\int_0^{t_2}\dots \int_0^{t_{k-1}}(1-\chi_\lambda(t_k))\frac{t_k^{\frac{d-1}{2}}\O(\langle\omega\rangle^{-\frac{3}{2}})[1+e_3(t_k,\lambda)]}{(1-t_k)^{s-\lambda-\frac12}}dt_k \dots d t_3 dt_2 
\\
&=\sum_{j=1}^{k-1} \int_0^{t_1}\int_0^{t_2}\dots \int_0^{t_{j}}\O(\langle\omega\rangle^{-k-\frac{1}{2}})\frac{\partial_{t_{j+1}}^{j}[(1-\chi_\lambda(t_k))t_{j+1}^{\frac{d-1}{2}}[1+e_3(t_{j+1},\lambda)]]}{(1-t_{j+1})^{s-\floor s -\frac 12-\lambda}}dt_{j+1} \dots d t_3 dt_2
\\
&\quad +  \O(\langle\omega\rangle^{-k-\frac{1}{2}})\frac{(1-\chi_\lambda(t_1))t_{1}^{\frac{d-1}{2}}[1+e_3(t_{1},\lambda)]]}{(1-t_{1})^{s-\floor s-\frac12-\lambda}}.
\end{align*}
 Now, we again differentiate between $dr-r-k+n> 0$ and $dr-k-r+n <0$, starting with the former case. Given that $ s-\floor s =\theta$ an application of Fubini's Theorem and Lemma \ref{osci3} yields
\begin{align*}
&\quad T_{k,10,0}^{n-m}(f)(\tau,\rho)
\\
&:= \quad\bigg|\int_\R e^{i \omega \tau } \chi_{\mu_1+i \omega}(\rho)\partial_\rho^{n-m}\left[(1-\rho^2)^{s-(\mu_1+i\omega)-\frac12}\rho^{\frac{1-d}{2}}b_1(\rho,\mu_1+i \omega)[1+\rho^2 e(\rho,\mu_1+i\omega)]\right]
\\
&\quad \times \int_\rho^1 \O(\langle\omega\rangle^{-k-\frac{1}{2}})\frac{(1-\chi_{\mu_1+i\omega}(t_1))t_{1}^{\frac{d-1}{2}}[1+e_3(t_{1},\mu_1+i\omega)]]}{(1-t_{1})^{\theta-\frac12-\mu_1-i\omega}} f^{(k-1)}(t_1) dt_1 d\omega \bigg|
\\
&=\bigg|\int_\rho^1 f^{(k-1)}(t_1) \int_\R e^{i \omega \tau } \chi_{\mu_1+i \omega}(\rho)(1-\rho^2)^{s-(\mu_1+i\omega)-\frac12}\O(\langle\omega\rangle^{-1+\varepsilon}\rho^{k-\frac{d-1}{2}-n+m+\varepsilon})
\\
&\quad \times\frac{(1-\chi_{\mu_1+i\omega}(t_1))t_{1}^{\frac{d-1}{2}}}{(1-t_{1})^{\theta -\frac12-\mu_1-i\omega}} d\omega dt\bigg|
\\
&\lesssim \rho^{k-\frac{ d-1}{2}-n+m+\varepsilon}\int_\rho^1 \langle\tau+\log(1-t)\rangle^{-2}
|f^{(k-1)}(t)|t^{\frac{d-1}{2}}(1-t)^{-\theta+\frac{1}{2}
+\mu_1}
\\
&\quad \times |\tau-\tfrac{1}{2}\log(1-\rho^2)+\log(1-t)|^{-\varepsilon}dt
\\
& \lesssim \rho^{k-(d-1)r-n+m+\varepsilon}\int_0^1 \langle\tau+\log(1-t)\rangle^{-2}
|f^{(k-1)}(t)|t^{(d-1)r}(1-t)^{-\theta+\frac{1}{2}
+\mu_1}
\\
&\quad \times |\tau-\tfrac{1}{2}\log(1-\rho^2)+\log(1-t)|^{-\varepsilon}dt
\end{align*} 
for some small $\varepsilon$.
Consequently, Minkwoski's inequality implies
\begin{align*}
&\quad\||.|^{-m} T_{k,10,0}^{n-m}(f)(\tau,.)\|_{L^{\frac{d}{dr-r-k+n}}(\B^d_1)}
\\
&\lesssim\langle\tau+\log(1-t)\rangle^{-2}\int_0^1 \lesssim\langle\tau+\log(1-t)\rangle^{-2}
|f^{(k-1)}(t)|t^{(d-1)r}(1-t)^{-\theta+ \frac{1}{2}-\mu_1}
\\
&\quad \times \|\rho^{k-(d-1)r-n+\varepsilon}|\tau-\tfrac{1}{2}\log(1-\rho^2)+\log(1-t)|^{-\varepsilon}\|_{L_\rho^{\frac{d}{dr-r-k+n}}(\B^d_1)}dt.
\end{align*}
Observe that
\begin{align*}
&\quad\|\rho^{k-(d-1)r-n+\varepsilon}|\tau-\frac{1}{2}\log(1-\rho^2)+\log(1-t)|^{-\varepsilon}\|_{L_\rho^{\frac{d}{dr-r-k+n}}(\B^d_1)}^{\frac{d}{dr-r-k+n}}
\\
&=\int_0^1 \rho^{-1+\varepsilon\frac{d}{dr-r-k+n} }|\tau-\tfrac{1}{2}\log(1-\rho^2)+\log(1-t)|^{-\varepsilon \frac{d}{dr-r-k+n}}
\\
&\lesssim  |\tau+\log(1-t)|^{-\varepsilon \frac{d}{dr-r-k+n}}
\end{align*}
provided $\varepsilon$ is chosen small enough, by Lemma \ref{teclem4}. Hence, 
\begin{align*}
&\quad \||.|^{-m} T_{k,10,0}^{n-m}(f)(\tau,.)\|_{L^p_\tau(\R_+) L^{q}(\B^d_1)}
\\
& \lesssim 
\left\|\int_0^1 \langle\tau+ \log(1-t)\rangle^{-2}
|f^{(k-1)}(t)|t^{(d-1)r}(1-t)^{-\theta+\frac{1}{2}+\mu_1}|\tau+\log(1-t)|^{-\varepsilon} dt\right\|_{L^p_\tau(\R_+)}
\end{align*} and by changing variable according to $t=1-e^{-y}$ and using Young's inequality, we deduce that
\begin{align*}
&\quad \||.|^{-m} T_{k,10,0}^{n-m}(f)(\tau,.)\|_{L^p_\tau(\R_+) L^{q}(\B^d_1)}
\\
& \lesssim 
\left\|\int_0^\infty \langle\tau-y\rangle^{-2}
|f^{(k-1)}(1-e^{-y})|(1-e^{-y})^{(d-1)r }e^{\tau(\theta-\frac32-\mu_1)}|\tau-y|^{-\varepsilon} dy\right\|_{L^p_\tau(\R_+)}
\\
&\lesssim \|\langle\tau\rangle^{-2}|\tau|^{-\varepsilon}\|_{L_\tau^{\frac{p}{1+p(1-r)}}(\R)} 
\|f^{(k-1)}(1-e^{-y})|((1-e^{-y})^{(d-1)r }e^{-\tau r}\|_{L^{\frac{1}{r}}(\R_+)}
\\
&\lesssim 
\||.|^{(d-1)r}
f^{(k-1)}\|_{L^{\frac{1}{r}}((0,1))}\lesssim \|f\|_{W^{k-1,\frac1r}(\B^d_1)}
\end{align*}
for $\varepsilon$ small enough and all admissible pairs $n,m,p,q$ (in case $p=\infty$ the fraction $\tfrac{p}{1+p(1-r)}$ should be understood as $\tfrac{1}{1-r}$). Now, in case $ dr-r-k+n\leq 0$, we have that $n=0$ and $k\geq dr-r$. Consequently, $\rho^{k-(d-1)r+\varepsilon}\leq 1 $ and the claimed estimate follows by the same means. 
For the remaining terms, we integrate by parts once more to derive 
\begin{align*}
&\quad\int_0^{t_1}\int_0^{t_2}\dots \int_0^{t_{j}}\O(\langle\omega\rangle^{-k-\frac{1}{2}})
\frac{\partial_{t_{j+1}}^{j}[(1-\chi_\lambda(t_{j+1}))t_{j+1}^{\frac{d-1}{2}}[1+e_3(t_{j+1},\lambda)]]}{(1-t_{j+1})^{\theta-\frac12-\lambda}}dt_{j+1} \dots d t_3 dt_2 
\\
&= \int_0^{t_1}\int_0^{t_2}\dots \int_0^{t_{j-1}}\O(\langle\omega\rangle^{-k-\frac{3}{2}})
\frac{\partial_{t_{j}}^{j}[(1-\chi_\lambda(t_j))t_{j}^{\frac{d-1}{2}}[1+e_3(t_{j},\lambda)]]}{(1-t_{j})^{\theta-\frac32-\lambda}}dt_{j} \dots d t_3 dt_2 
\\
& \quad+ \int_0^{t_1}\int_0^{t_2}\dots \int_0^{t_{j}}\O(\langle\omega\rangle^{-k-\frac{3}{2}})
\frac{\partial_{t_{j+1}}^{j+1}[(1-\chi_\lambda(t_k))t_{j+1}^{\frac{d-1}{2}}[1+e_3(t_{j+1},\lambda)]]}{(1-t_{j+1})^{\theta-\frac32-\lambda}}dt_{j+1} \dots d t_3 dt_2.
\end{align*}
Then, by arguing as above and employing Lemma \ref{osci3}, one deduces that
\begin{align*}
&\quad\bigg|\int_\R e^{i \omega \tau }  \chi_{\mu_1+i\omega}(\rho)\partial_\rho^{n-m}\left[(1-\rho^2)^{s-\mu_1-i\omega-\frac12}\rho^{\frac{1-d}{2}}b_1(\rho,\mu_1+i\omega)[1+\rho^2 e(\rho,\mu_1+i\omega)]\right]
\\
&\quad \times
\O(\langle\omega\rangle^{-k-\frac{3}{2}}) \int_\rho^1  \int_0^{t_1}\int_0^{t_2}\dots \int_0^{t_{j-1}}
\frac{\partial_{t_{j}}^{j}[(1-\chi_{\mu_1+i\omega}(t_j))t_{j}^{\frac{d-1}{2}}[1+e_3(t_{j},\mu_1+i\omega)]]}{(1-t_{j})^{\theta-\frac32-\mu_1+i\omega}}
\\
&\quad \times d\omega dt_{j} \dots d t_3 dt_2 f^{(k-1)}(t_1) dt_1 \bigg|
\\
&\lesssim\rho^{k-(d-1)r-m+\varepsilon}\int_0^1 
\int_0^{t_1}\int_0^{t_2}\dots \int_0^{t_{j-1}}s^{(d-1)r-(j+1)-\varepsilon}(1-t_j)^{\frac{1}{4}}\langle\tau+\log(1-t_j)\rangle^{-2}
\\
&\quad \times dt_j\dots d t_3 dt_2 
\lesssim \langle\tau\rangle^{-2} \rho^{k-(d-1)r-m+\varepsilon}\int_0^1 
t_1^{\frac{d-1}{2}-\varepsilon}|f^{(k-1)}(t_1)| dt_1 
\\
& \quad \langle\tau\rangle^{-2} \rho^{k-(d-1)r-m+\varepsilon} \|f\|_{W^{k-1,\frac1r}(\B^d_1)}.
\end{align*} 
By progressing likewise for the remaining terms one, readily obtains the desired bounds on $|.|^{-m}T_{k,10}^{n-m}(f)(\tau,\rho)$. For $|.|^{-m}\dot{T}_{k,10}^{n-m}(f)(\tau,\rho)$, one performs one more integration by parts and then applies similar reasoning. So, we turn to $\ell=11$. For $n<k$ one can derive all the desired estimates on $|.|^{-m}T^{n-m}_{k,11}$ and $|.|^{-m}\dot T^{n-m}_{k,11}$ in the same fashion as above. For $n=k$ the operator $ T^k_{k,11}$ requires further inspection. This stems from the fact that all $\rho-$derivatives could hit the term $(1-\rho)^{s-\lambda-\frac12}$, which leads to the term
\begin{align}
T^{k'}_{k,11}(f)(\tau,\rho)&:=\int_\R e^{i\omega\tau} (1-\chi_{\mu_1+i \omega}(\rho))\rho^{\frac{1-d}{2}}(1-\rho)^{\theta-\mu_1+i\omega-\frac{3}{2}}\nonumber
\\
&\quad \times[1+(1-\rho)\O(\rho^{-1}\langle\omega\rangle^{-1})]
\int_\rho^{1}\int_0^{t_1}\int_0^{t_2}\dots \int_0^{t_{k-1}}\chi_{\mu_1+i\omega}(t_k)
\\
&\quad \times \frac{\O(t_k\langle\omega\rangle^{k-\frac{d+1}{2}})}{(1-t_k^2)^{\frac{s}{2}+\frac{1}{4}-\frac{\mu_1+i\omega}{2}}}dt_k \dots d t_3 dt_2 f^{(k-1)}(t_1) dt_1 d\omega. \nonumber
\end{align} 
Lemma \ref{osci3} yields the estimate
\begin{align*}
|T^{k'}_{k,11}(f)(\tau,\rho)|&\lesssim\rho^{(1-d)r}(1-\rho)^{\theta-\mu_1-\frac{3}{2}}\langle\tau-\log(1-\rho)\rangle^{-2}\int_0^1 t^{(d-1)r}|f^{(k-1)}(s)| dt
\\
& \lesssim \rho^{(1-d)r}(1-\rho)^{\theta-\mu_1-\frac{3}{2}}\langle\tau-\log(1-\rho)\rangle^{-2}\|f\|_{W^{k-1,\frac1r}(\B^d_1)}.
\end{align*}
Next, we once more change coordinates according to $\rho= (1-e^y)$,  to compute that
\begin{align*}
\|T^{k'}_{k,11}(f)(\tau,\rho)\|_{L^\infty_\tau(\R_+) L^{\frac1r}_\rho(\B^d_1)}^{\frac1r}&\lesssim \left\| \int_0^1(1-\rho)^{-1} \langle\tau -\log(1-\rho)\rangle^{-4} d\rho \right\|_{L^\infty_\tau(\R_+)}^{\frac1r}\|f\|_{W^{k-1,r}(\B^d_1)}^{\frac1r}
\\
&\leq  \left\|\int_{-\infty}^0 \langle\tau -y\rangle^{-4} dy \right\|_{L^\infty_\tau(\R_+)}^{\frac{1}{r}} \|f\|_{W^{k-1,\frac 1r}(\B^d_1)}^{\frac1r}\lesssim \|f\|_{W^{k-1,r}(\B^d_1)}^{\frac1r}.
\end{align*}
Likewise, one bounds the remaining terms and also deals with $\dot{T}^k_{k,11}$. By similar means one bounds $|.|^{-m}T_{k,12}^{n-m}$ as well as $|.|^{-m}\dot T_{k,12}^{n-m} $ for all desired $n,m,p,q$. Furthermore, by using already employed means, one also readily derives the claimed estimates on $|.|^{-m}T_{k,13}^{n-m}$ for all $0\leq n\leq k$, $0\leq m <n$ and on $|.|^{-m}\dot T_{k,13}^{n-m}$ for $0\leq n\leq k-1$, $0\leq m <n$. However, to establish estimates on $\dot{T}^k_{k,13}$, one once more requires a cancellation. In particular, by integrating by parts $k$ times in both the kernel $H_{k,13}^k$ and $H_{k,16}^k$, one obtains troubling boundary parts, which however cancel in pairs. This procedure can be carried out as in the proof of Lemma \ref{lem:strichartz2a} and leads to desired estimates on $\dot{T}^k_{k,13}+\dot{T}^k_{k,16}$.
Furthermore, for $\ell=14,15$, one can bound the associated operators by again trading powers of $t$ and $\rho$ for decay. Lastly, since the integral in $H^n_{k,17}$ is over the whole interval $(0,1)$, bounding the associated oscillatory integrals can be achieved by integrating by parts as often as needed and then applying our technical Lemmas. Further, since no boundary terms will pop up when performing the integrations by parts, this is straightforward and needs no substantial modification of the strategies used so far. Consequently, only the estimates 
 \begin{align*}
\|T_{k-1,9}^{0,0}(f)\|_{L^{\frac{2}{1-\frac{2\delta}{1-\theta}}}(\R_+)L^\infty(\B^d_1)}\lesssim \| f\|_{W^{k-2,\frac{2}{1-\frac{2\delta}{1-\theta}}}(\B^d_1)}
\end{align*}
and 
\begin{align*}
\|\dot{T}_{k-1,9}^{0,0}(f)\|_{L^{\frac{2}{1-\frac{2\delta}{1-\theta}}}(\R_+)L^\infty(\B^d_1)}\lesssim \| f\|_{W^{k-1,\frac{2}{ 1-\frac{2\delta}{1-\theta}}}(\B^d_1)}
\end{align*}
remain to be shown. By mimicking the computations used to bound $T_{k,9}^{n,1}$ on concludes that 
$$
|T_{k-1,9}^{0,0}(f)(\rho,\tau)|\lesssim_\varepsilon \langle\tau \rangle^{-2} \int_\rho^{\rho_1}\int_0^{t_1}\int_0^{t_2}\dots \int_0^{t_{k-2}}t_{k-1}^{1-\varepsilon} dt_{k-1} \dots d t_3 dt_2 |f^{(k-2)}(t_1)| dt_1
$$
which, for $\varepsilon$ small enough, implies that
\begin{align*}
\|T_{k-1,9}^{0,0}(f)\|_{L^{\frac{2}{1-\frac{2\delta}{1-\theta}}}(\R_+)L^\infty(\B^d_1)}&\lesssim_\varepsilon \int_0^1|f^{(k-2)}(t)|t^{k-1-\varepsilon} dt&\lesssim_\varepsilon \|f^{(k-2)}(t)t^{\frac{d-1}{2}}\|_{L^{\frac{2}{ 1-\frac{2\delta}{1-\theta}}}((0,1))}
\\
&\lesssim \| f\|_{W^{k-2,\frac{2}{1-\frac{2\delta}{1-\theta}}}(\B^d_1)}.
\end{align*}
A minor modification then shows 
\begin{align*}
\|\dot T_{k-1,\ell}^{0,0}(f)\|_{L^{\frac{2}{1-\frac{2\delta}{1-\theta}}}(\R_+)L^\infty(\B^d_1)}&\lesssim_\varepsilon \int_0^1|f^{(k-2)}(t)|t^{k-2-\varepsilon} \lesssim \| f\|_{W^{k-1,\frac{2}{1-\frac{2\delta}{1-\theta}}}(\B^d_1)}
\end{align*}
and we move to $T_{k-1,10}^{0,0} $. For this, we look at
\begin{align*}
H_{k-1,10}^0(f)(\rho,\lambda)&=\chi_\lambda(\rho)(1-\rho^2)^{\frac{s}{2}-\frac{3}{4}-\frac{\lambda}{2}}\rho^{\frac{1-d}{2}}b_1(\rho,\lambda)
\\
&\quad \times [1+\rho^2 e_3(\rho,\lambda)]
\int_\rho^{1}\int_0^{t_1}\int_0^{t_2}\dots \int_0^{t_{k-2}}(1-\chi_\lambda(t_{k-1}))
\frac{t_{k-1}^{\frac{d-1}{2}}\O(\langle\omega\rangle^{-\frac{3}{2}})}{(1-t_{k-1})^{s-\lambda-\frac12}}
\\
&\quad \times [1+e_2(t_{k-1},\lambda)] dt_{k-1} \dots d t_3 dt_2 f^{(k-2)}(t_1) dt_1
 + P_{k-1,10}(f)(\rho,\lambda).
\end{align*}
Now, we perform a number of integrations by parts to conclude 
\begin{align*}
&\quad\int_0^{t_1}\int_0^{t_2}\dots \int_0^{t_{k-2}}(1-\chi_\lambda(t_k))\frac{t_{k-1}^{\frac{d-1}{2}}\O(\langle\omega\rangle^{-\frac{3}{2}})[1+e_3(t_{k-1},\lambda)]}{(1-t_{k-1})^{s-\lambda-\frac12}}dt_{k-1} \dots d t_3 dt_2 
\\
&=\sum_{j=1}^{k-2} \int_0^{t_1}\int_0^{t_2}\dots \int_0^{t_{j}}\O(\langle\omega\rangle^{-k+\frac{1}{2}})\frac{\partial_{t_{j+1}}^{j}[(1-\chi_\lambda(t_{k-1}))t_{j+1}^{\frac{d-1}{2}}[1+e_3(t_{j+1},\lambda)]]}{(1-t_{j+1})^{s-\floor s +\frac 12-\lambda}}dt_{j+1} \dots d t_3 dt_2
\\
&\quad +  \O(\langle\omega\rangle^{-k+\frac{1}{2}})\frac{(1-\chi_\lambda(t_1))t_{1}^{\frac{d-1}{2}}[1+e_3(t_{1},\lambda)]]}{(1-t_{1})^{s-\floor s+\frac12-\lambda}}.
\end{align*}
So, the most difficult term to bound is given by
\begin{align*}
T_{k-1,10_1}^{0,0}(f)(\rho,\tau)&:=\int_\R e^{i \omega \tau}\chi_{\mu_0+i\omega}(\rho)(1-\rho^2)^{\frac{s}{2}-\frac{3}{4}-\frac{\mu_0+i\omega}{2}}\rho^{\frac{1-d}{2}}b_1(\rho,\mu_0+i\omega)
\\
&\quad \times \int_\rho^1
\O(\langle\omega\rangle^{-k+\frac{1}{2}})\frac{(1-\chi_\lambda(t_1))t_{1}^{\frac{d-1}{2}}[1+e_3(t_{1},\lambda)]]}{(1-t_{1})^{s-\floor s+\frac12-\lambda}} dt_1 d\omega.
\end{align*}
An application of Lemma \ref{osci3} shows
\begin{align*}
|T_{k-1,10_1}^{0,0}(f)(\rho,\tau)|&\lesssim \int_0^1|\tau-\log(1-\rho^2)+\log(1-t)|^{-\frac12}\langle\tau-\log(1-\rho^2)+\log(1-t)\rangle^{-1}
\\
&\quad \times \langle\tau+\log(1-t)\rangle^{-2}\frac{t^{\frac{d-1}{2}}}{(1-t)^{s-\floor s+\frac12-\mu_0}} |f^{(k-2})(t)| dt
\\
&= \int_0^\infty|\tau-\log(1-\rho^2)-y|^{-\frac12}\langle\tau-\log(1-\rho^2)-y\rangle^{-1}
\\
&\quad \times \langle\tau-y\rangle^{-1}\frac{(1-e^{-y})^{\frac{d-1}{2}}}{e^{y(s-\floor s+\frac12-\mu_0)}}e^{-y} |f^{(k-2}(1-e^{-y})| dt.
\end{align*}
Further, as $\mu_0=\theta -\frac{\delta}{1-\theta}$ an application of Hölder's inequality yields
\begin{align*}
|T_{k-1,10_1}^{0,0}(f)(\rho,\tau)|&\lesssim
\||.|^{-\frac{1}{2}} \langle.\rangle^{-1}\|_{L^{\frac{2(1-\theta)}{1-\theta+2\delta}}(\R)}
\\
&\quad\times \|\langle\tau-y\rangle^{-1}(1-e^{-y})^{\frac{d-1}{2}}e^{-y(\frac{1}{2}-\frac{\delta}{1-\theta})} f^{(k-2}(1-e^{-y})\|_{L^{\frac{2}{1-\frac{2\delta}{1-\theta}}}_y(\R_+)}
\\
& \lesssim
\|\langle\tau-y\rangle^{-1}(1-e^{-y})^{\frac{d-1}{2}}e^{-y(\frac{1}{2}-\frac{\delta}{1-\theta})} f^{(k-2}(1-e^{-y})\|_{L^{\frac{2}{1-\frac{2\delta}{1-\theta}}}_y(\R_+)}.
\end{align*}
Therefore, an application of Young's inequality yields
\begin{align*}
\|T_{k-1,10_1}^{0,0}(f)\|_{L^{\frac{2}{1-\frac{2\delta}{1-\theta}}}_y(\R_+)L^\infty(\B^d_1)}^{\frac{2}{1-\frac{2\delta}{1-\theta}}}&\lesssim \int_0^\infty \int_0^\infty |\langle\tau-y\rangle^{-2}(1-e^{-y})^{d-1}e^{-y} |f^{(k-2}(1-e^{-y})|^{\frac{2}{1-\frac{2\delta}{1-\theta}}} dy d\tau
\\
&
\lesssim \int_\R \langle\tau\rangle^{-2} d\tau \int_0^\infty (1-e^{-y})^{d-1}e^{-y} |f^{(k-2}(1-e^{-y})|^{\frac{2}{1-\frac{2\delta}{1-\theta}}} dy
\\
&\lesssim \|f^{(k-1)}\|_{L^{\frac{2}{1-\frac{2\delta}{1-\theta}}}(\B^d_1)}^{\frac{2}{1-\frac{2\delta}{1-\theta}}} \lesssim \|f\|_{W^{k-2,\frac{2}{1-\frac{2\delta}{1-\theta}}}(\B^d_1)}^{\frac{2}{1-\frac{2\delta}{1-\theta}}}.
\end{align*}
Now, to bound $\dot T_{k-1,10}$, one performs another integration by parts and argues likewise. Furthermore, as the remaining estimates follow by employing the same tools, we conclude this proof. 
\end{proof}
\begin{lem} \label{lem:strich3b}
Let $ 3\leq d\in \mathbb{N}$ and $1\leq s \in \mathbb{N}$ with $1\leq s < \frac{d}{2}$. Then, the estimates
\begin{align*}
\left\|\sum_{\ell=9}^{17}|.|^{-m}T_{s,\ell}^{n-m,1}f\right\|_{L^p(\R_+)L^q(\B^d_1)}\lesssim \| f\|_{W^{s-1,\frac{2}{1+2\delta}}(\B^d_1)}
\end{align*}
and 
\begin{align*}
\left\|\sum_{\ell=9}^{17}|.|^{-m}\dot{T}_{s,\ell}^{n-m,1}f\right\|_{L^p(\R_+)L^q(\B^d_1)}\lesssim \| f\|_{W^{s,\frac{2}{1+2\delta}}(\B^d_1)}
\end{align*}
hold for all  $n,m\in \mathbb{Z}$ with $0\leq n \leq s$, $0\leq m<n$,
 $p,q\in [\frac{2}{1 +2\delta},\infty]$ such that the scaling relation  $$\frac{1}{p}+\frac{d}{q}=d(\frac{1}{2} +\delta)-s+n$$ 
is satisfied, as well as all $f\in C^\infty_{rad}(\overline{\B^d_1})$.
 Furthermore, the estimates 
\begin{align*}
\left\|\sum_{\ell=9}^{17}|.|^{-m}T_{s,\ell}^{n-m,0}f\right\|_{L^p(\R_+)L^q(\B^d_1)}\lesssim \| f\|_{W^{s-1,\frac{2}{1-2\delta}}(\B^d_1)}
\end{align*}
and 
\begin{align*}
\left\|\sum_{\ell=9}^{17}|.|^{-m}\dot{T}_{s,\ell}^{n-m,0}f\right\|_{L^p(\R_+)L^q(\B^d_1)}\lesssim \| f\|_{W^{s,\frac{2}{1-2\delta}}(\B^d_1)}
\end{align*}
hold for the same range of $n,m$ and all $ p,q \in [\frac{2}{1-2\delta}, \infty ]  $ such that the scaling relation  $$\frac{1}{p}+\frac{d}{q}=d(\frac12-\delta)-s+n$$ is satisfied, as well as all $f\in C^\infty_{rad}(\overline{\B^d_1})$.
Lastly, in case $s=\frac{d-1}{2}$, one also has that 
\begin{align*}
\|T_{k,\ell}^{0,1}f\|_{L^{\frac{2}{1+2\delta}}(\R_+)L^\infty(\B^d_1)}&\lesssim \| f\|_{W^{s-1,\frac{2}{1+2\delta}}(\B^d_1)}
\\
\| \dot{T}_{k,\ell}^{0,1}f\|_{L^{\frac{2}{1+2\delta}}(\R_+)L^\infty(\B^d_1)}&\lesssim \| f\|_{W^{s,\frac{2}{1+2\delta}}(\B^d_1)}
\end{align*}
and
\begin{align*}
\|T_{k,\ell}^{0,0}f\|_{L^{\frac{2}{1-2\delta}}(\R_+)L^\infty(\B^d_1)}&\lesssim \| f\|_{W^{s-1,\frac{2}{1-2\delta}}(\B^d_1)}
\\
\| \dot{T}_{k,\ell}^{0,0}f\|_{L^{\frac{2}{1-2\delta}}(\R_+)L^\infty(\B^d_1)}&\lesssim \| f\|_{W^{s,\frac{2}{1-2\delta}}(\B^d_1)}
\end{align*}
for  and $\ell=9,\dots,17$ and all $f\in C^\infty_{rad}(\overline{\B^d_1})$.
\end{lem}

To prove Theorem \ref{thm:strichartz} we first need a proper definition of admissible Strichartz indices. 
\begin{defi}
Let $d\geq 3$ and $s\notin \mathbb{N}$ with $1\leq s<\frac{d}{2}$. Then, we call the triple $(p_1,q_1,n_1)$, with $ \ceil s =k \geq n_1\in \mathbb{N}$ and $p_1,q_1\in [\frac{2}{1+2\frac{\delta}{\theta}},\infty],$ $ \ceil s$-Strichartz admissible, provided that
$$\frac{1}{p_1}+\frac{d}{q_1}=d(\frac{1}{2}+\frac{\delta}{\theta})-k+n_1.$$
Similarly, the triple $(p_0,q_0,n_0)$ with $k-1\geq n_0\in \mathbb{N}$ and $p_0,q_0\in [\frac{2}{1-\frac{2\delta}{1-\theta}},\infty]$ is called $ \floor s$-Strichartz admissible, provided that
$$\frac{1}{p_0}+\frac{d}{q_0}=d(\frac{1}{2}-\frac{\delta}{1-\theta})-s+1+n_0.$$

In case $s \in \mathbb{N}$, the triple $(p,q,n)$ with $s\geq n \in \mathbb{N}$ is called $ s^+$-Strichartz admissible, provided that
\begin{equation}\label{eq:defi ad}
\frac{1}{p}+\frac{d}{q}=d(\frac{1}{2}-\delta)-s+n
\end{equation}
and $ s^-$-Strichartz admissible if, instead of \eqref{eq:defi ad}, the relation
\begin{equation*}
\frac{1}{p}+\frac{d}{q}=d(\frac{2}{1+2\delta})-s+n
\end{equation*}
is satisfied.
Finally, any triple $(p,q,n)$ as in the Lemma \ref{Strichart} is called $s$-admissible.
\end{defi}
The following remark highlights some important properties of these triples.
\begin{rem} \label{rem:1}
Let $d \geq 3$ and $1\leq s < \ceil s\leq \frac{d}{2}$ be fixed. Then, for any $s$-Strichartz admissible triple $(p,q,n)$, with $q\neq \infty$ and where $n\in \mathbb{N}$ is such that $(1-\theta) n_0 +\theta n_0$ with $n_0\geq  n_1 \in \mathbb{N}$, there exists a $\ceil s$-Strichartz admissible triple $(p_1,q_1,n_1),$ as well as a $\floor s$-Strichartz admissible triple $(p_0,q_0,n_0),$ such that
\begin{align*}
\frac{1-\theta}{p_0}+\frac{\theta}{p_1}=\frac{1}{p}, \quad \frac{1-\theta}{q_0}+\frac{\theta}{q_1}=\frac{1}{q},\quad (1-\theta) n_0+ \theta n_1 =n.
\end{align*}
This follows from
\begin{align*}
&\quad \frac{1-\theta}{p_0}+\frac{\theta}{p_1}+ d\left( \frac{1-\theta}{q_0}+\frac{\theta}{q_1}\right)
\\
&=
(1-\theta )\left[d\left(\frac{1}{2}-\frac{\delta}{1-\theta}\right)-k+1+n_0\right]+\theta \left[d\left(\frac{1}{2}+\frac{\delta}{\theta}\right)-k+n_1\right]
\\
&=\frac{d}{2}-s+ (1-\theta) n_0 +\theta n_1. 
\end{align*} 
Moreover, if an $s$ admissible triple is of the form $(p,\infty,0)$, which implies $s\geq \frac{d-1}{2}$, then the additional $L^p L^\infty$ estimates from Lemmas \ref{lem:strich1a}, \ref{lem:strichartz2a}, and \ref{lem:strich3a} ensure that we also find triples $ (p_1,\infty,0), (p_0,\infty,0)$ that interpolate into the desired space.
Furthermore, in the integer case, the analogous statement holds as well.
\end{rem}
Further, to deal with the isolated eigenvalues, we will rely on the following result.
\begin{lem}\label{lem:proj}
Let $H$ be a Hilbert space. Then, for any densely defined
operator $T:D(T)\subset H \to H$ with finite rank, there exists a
dense subset $X \subset H$ with $X\subset D(T)$ and a bounded linear operator $\widehat{T}:H \to H$ such that
\begin{equation*}
T|_X=\widehat{T}|_X.
\end{equation*}
\end{lem}
\begin{proof}
See Lemma 4.2 in \cite{DonWal23}.
\end{proof}
Applying this Lemma first to the projection $\Qf$ (viewed as a densely defined operator of finite rank in the $\mathcal{H}^s$ universe) and then to the restriction of $\Pf$ to the constructed dense subset, yields the following. 
\begin{lem}\label{lem:dense}
There exists a dense subset $X\subset \mathcal{H}^s$ and bounded linear operators $\widehat \Qf, \widehat \Pf :\mathcal{H}^s \to \mathcal{H}^s$ such that
\begin{align*}
\Qf|_X=\widehat{\Qf}|_X \text{ and } \Pf|_X=\widehat{\Pf}|_X.
\end{align*}
\end{lem}
We finally come to the proof of Theorem \ref{thm:strichartz}.
\begin{proof}[Proof of Theorem \ref{thm:strichartz}]

Let $\ff \in X$ and set $\widetilde \ff = (\I-\Qf)(\I-\Pf)\ff$ and assume that $s$ is not an integer.
Further, let $(p,q,n)$ be $s$-Strichartz admissible and let $n$ in addition be such that $ \theta n_1+ (1-\theta)n_0=n$ for $n_1, n_0 \in \mathbb{N}$. Further, let $(p_1,q_1,n_1)$ and $(p_0,q_0,n_0)$ be the $\ceil s$, respectively $\floor s$ admissible Strichartz triples such that
\begin{align*}
\frac{1-\theta}{p_0}+\frac{\theta}{p_1}=\frac{1}{p}, \quad \frac{1-\theta}{q_10}+\frac{\theta}{q_1}=\frac{1}{q},\quad (1- \theta) n_0+\theta n_1  =n.
\end{align*}
Then, by construction, and Lemmas \ref{lem:strich1a}, \ref{lem:strichartz2a}, and \ref{lem:strich3a} we know that
\begin{align*}
&\quad \|[e^{-\mu_1\tau}(\Sf(\tau)-\Sf_0(\tau))\widetilde \ff]_1\|_{L^{p_1}_\tau (\R_+) W^{n_1,q_1}(\B^d_1)}
\\
&\lesssim 
\sum_{m=0}^{n_1-1}\left\|\sum_{j=1}^{k-1}\sum_{\ell=1}^{8}|.|^{-m}T_{j,\ell}^{n_1-m,1}(f_1+f_2) + |.|^{-m}\dot T_{j,\ell}^{n_1-m,1}f_1\right\|_{L^{p_1}(\R_+)L^{q_1}(\B^d_1)}
\\
&\quad +\sum_{m=0}^{n_1-1}\left\|\sum_{\ell=9}^{17}|.|^{-m}T_{k,\ell}^{n_1-m,1}(f_1+f_2) + |.|^{-m}\dot T_{k,\ell}^{n_1-m,1}f_1\right\|_{L^{p_1}(\R_+)L^{q_1}(\B^d_1)}
\\
&\lesssim \|\widetilde \ff\|_{W^{k,\frac{2}{1+2\frac{\delta}{\theta}}}\times W^{k-1,\frac{2}{1+2\frac{\delta}{\theta}}}(\B^d_1)}
\end{align*}
and
\begin{align*}
&\quad \|[e^{-\mu_0\tau}(\Sf(\tau)-\Sf_0(\tau))\widetilde \ff]_1\|_{L^{p_0}_\tau (\R_+) W^{n_0,q_0}(\B^d_1)}
\\
&\lesssim \sum_{m=0}^{n_0-1}\left\|\sum_{j=1}^{k-2}\sum_{\ell=1}^{8}|.|^{-m}T_{j,\ell}^{n_0-m,1}(f_1+f_2) + |.|^{-m}\dot T_{j,\ell}^{n_0-m,1}f_1\right\|_{L^{p_0}(\R_+)L^{q_0}(\B^d_1)}
\\
&\quad +\sum_{m=0}^{n_0-1}\left\|\sum_{\ell=9}^{17}|.|^{-m}T_{k-1,\ell}^{n_0-m,1}(f_1+f_2) + |.|^{-m}\dot T_{k-1,\ell}^{n_0-m,1}f_1\right\|_{L^{p_0}(\R_+)L^{q_0}(\B^d_1)}
\\
&\lesssim \|\widetilde \ff\|_{W^{k-1,\frac{2}{1-2\frac{\delta}{1-\theta}}}\times W^{k-2,\frac{2}{1-2\frac{\delta}{1-\theta}}}(\B^d_1)}.
\end{align*}
Thus, as $\theta \mu_1 +(1-\theta) \mu_0=0$, an application of proposition \ref{prop:interpolation} yields 
\begin{align*}
\|[(\Sf(\tau)-\Sf_0(\tau))\widetilde \ff]_1\|_{L^{p}_\tau (\R_+) W^{n,q}(\B^d_1)}\lesssim \|\widetilde \ff\|_{\mathcal{H}^s}
\end{align*}
which, combined with the free Strichartz estimates from Lemma \ref{Strichart}, implies

\begin{align*}
\|[\Sf(\tau)\widetilde \ff]_1\|_{L^{p}_\tau (\R_+) W^{n,q}(\B^d_1)}\lesssim \|\widetilde \ff\|_{\mathcal{H}^s}.
\end{align*}
Furthermore, 
\begin{align*}
\|\widetilde \ff\|_{\mathcal{H}^s}=\|(\I-\Qf)(\I-\Pf) \ff\|_{\mathcal{H}^s}=\|(\I-\widehat\Qf)(\I-\widehat \Pf)  \ff\|_{\mathcal{H}^s}\lesssim \| \ff\|_{\mathcal{H}^s}.
\end{align*}
Recall, from Lemma \ref{reg:eigenfunctions} that $\gf\in rg \Qf \implies g_1 \in L^\infty(\B^d_1)\cap W^{n,\frac{2d}{d+2n-2s+1}}(\B^d_1)$ for all $1\leq n\leq k-1$. Therefore, as the range of $\Qf$ is finite dimensional we have that
\begin{align*}
|[\Qf (\I-\Pf) \ff]_1\|_{W^{n,q}(\B^d_1)}\lesssim \|\Qf (\I-\Pf) \ff\|_{\H^{\ceil s}}&\lesssim  \|\Qf (\I-\Pf) \ff\|_{\H^{ s}}=\|\widehat \Qf (\I-\widehat \Pf) \ff\|_{\H^{ s}}\lesssim \| \ff\|_{\H^{ s}}
\end{align*}
for all values of $n$ and $q$ with which we are concerned.
Furthermore, as the range of $\Qf$ is contained in the union of finitely many generalised eigenspaces, all corresponding to eigenvalues with negative real part, we conclude that there exists an $\varepsilon>0$ such that
\begin{align*}
\|[\Sf(\tau)\Qf(\I-\Pf)\ff]_1\|_{L^{p}_\tau (\R_+) W^{n,q}(\B^d_1)}&\lesssim \| e^{-\varepsilon \tau}[\Qf(\I-\Pf) \ff]_1\|_{L^{p}_\tau (\R_+) W^{n,q}(\B^d_1)}
\\
&\lesssim |[\Qf (\I-\Pf) \ff]_1\|_{W^{n,q}(\B^d_1)}.
\end{align*}
Consequently, we conclude that
\begin{align*}
\|[\Sf(\tau)(\I-\Pf) \ff]_1\|_{L^{p}_\tau (\R_+) W^{n,q}(\B^d_1)}\lesssim \|\widetilde \ff\|_{\mathcal{H}^s}
\end{align*}
for all $\ff \in X$. Therefore, as $\Pf$ agrees with $\widehat \Pf$ on $X$, we obtain the existence of a bounded operator with finite rank $\widehat{\Pf}$ such that
\begin{align*}
\|[\Sf(\tau)(\I-\widehat\Pf)\ff]_1\|_{L^{p}_\tau (\R_+) W^{n,q}(\B^d_1)}&\lesssim  \|\ff\|_{\H^s}
\end{align*}
for all $\ff \in X$. Hence the desired homogeneous estimates follow from a density argument. The inhomogeneous ones are then a consequence of Minkowski's inequality, as exhibited in the proof of Proposition 2.2 in \cite{Don17}. In the same fashion, one obtains the estimate 
\begin{align*}
\|[\Sf(\tau)(\I-\Pf)\ff]_1\|_{L^\infty_\tau(\R_+)H^s(\B^d_1)}\lesssim \|(\I-\Pf)\ff\|_{H^s\times H^{s-1}(\B^d_1)}.
\end{align*}
Furthermore, by construction, the second component of $\Sf(\tau)\widetilde \ff$ is given by
\begin{align*}
\left[\partial_\tau+\rho\partial_\rho+\frac{d-2s}{2}\right][\Sf(\tau)\widetilde \ff(\rho)]_1.
\end{align*} 
Moreover, given that $\partial_\tau$ essentially amounts to multiplying by $\lambda$, the operators one obtains this way are comparable to the $[\lambda+\rho\partial_\rho+\frac{d-2s}{2}]T^{n}_{k,\ell}(f)(\rho,\lambda)$. Thus, one readily obtains
\begin{align*}
\|[\Sf(\tau)(\I-\widehat \Pf)\ff]_2\|_{L^\infty_\tau(\R^+)H^{s-1}(\B^d_1)}\lesssim \|(\I-\widehat \Pf)\ff\|_{H^s\times H^{s-1}(\B^d_1)}.
\end{align*}
In the integer case, one argues likewise and we conclude the proof of Theorem \ref{thm:strichartz}.
\end{proof}

\section{Blowup stability}
Establishing our result on optimal blowup stability is now a simple task and we begin by recalling the explicit form of the ODE Blowup 
$$
u^T(t):=\left(\frac{3}{4}\right)^{\frac{1}{4}}(T-t)^{-\frac12}
$$
which, transformed to similarity variables, is just the constant function 
$$
\Psi_*=\begin{pmatrix}
\left(\frac{3}{4}\right)^{\frac{1}{4}}
\\
\frac{1}{2} \left(\frac{3}{4}\right)^{\frac{1}{4}}
\end{pmatrix}.
$$
If we now linearise the nonlinearity $\N_0(\uf)$ 
that, in accordance with our transformations in section 2, is (at least as an formal expression) given by
$$
\Nf_0(\uf)=\begin{pmatrix}
0\\
u_1^5
\end{pmatrix}
$$
 around $\Psi_*$, we obtain a potential operator $\Lf'$ given by 
 \begin{align}\label{def:lf}
 \Lf' \uf:= \begin{pmatrix}
 0
 \\
 \frac{15}{4} u_1
 \end{pmatrix}
 \end{align}
 and a nonlinearity given by
 \begin{align*}
 \Nf(\uf)=\begin{pmatrix}
 0
 \\
 \sum_{j=2}^5 c_j u^j_1
 \end{pmatrix}
 \end{align*}
  for some positive constants $c_j$. Finally, we remark, that in accordance with the scaling of the quintic wave equation, the space we work in is given by 
  \begin{align*}
  H^{\frac{d-1}{2}}\times H^{\frac{d-3}{2}}(\B^d_1).
  \end{align*}
 \begin{lem}
 The point spectrum $\sigma_p$ of the operator $\Lf= \overline{\Lf_0+\Lf'}$ with $\Lf'$ defined by Eq.~\eqref{def:lf} satisfies
 \begin{align*}
 \sigma_p(\Lf)\subset \{z\in \C: \Re z<0\} \cup \{1\}.
\end{align*}  
Furthermore, $1$ is a simple eigenvalue and an eigenfunction is given by
$$
\gf=\begin{pmatrix}
2\\ 
3
\end{pmatrix}.
$$
 \end{lem}
 \begin{proof}
 Note that, as before, we can reduce the eigenvalue equation  $(\lambda-\Lf)\uf=0$ to a second order ODE. Explicitly, the first component of any eigenfunction needs to satisfy
\begin{align*}
(\rho^2-1)u''(\rho)&+\left((2\lambda+3)\rho-\frac{d-1}{\rho}\right)u_1'(\rho)
\\
&
+\frac{1+2\lambda}{4}(2\lambda+3)u_1(\rho)-\frac{15}{4}u_1(\rho)=0.
\end{align*} 
This equation can be transformed into a hypergeometric equation. Then, one establishes that the only eigenvalue with nonnegative real part is given by $1$, by adapting the arguments from the proof of Lemma 4.10 in \cite{DonSch17}. To show that $1$ is indeed simple, one modifies the considerations in Lemma 4.11 in \cite{DonSch17} in a straightforward way.
\end{proof}
Consequently, the only task that remains is to establish control over the nonlinearity. 
To put the tediously involved estimates on $\Nf$ into a compact expression, we let $A$ be the set of all admissible triples $(p,q,n)$ i.e. all numbers $p,q\in \R$ and $n\in \mathbb N_0$ with $0 \leq n\leq \frac{d-1}{2}$, 
$p\in [2,\infty]$ and $q\in [2,\frac{d}{n}]$ that satisfy
$$
 \frac{1}{p}+\frac{d}{q}=\frac{1}{2}+n
$$
and
\begin{align*}
\frac{1}{p}+\frac{d-1}{2q}&\leq \frac{d-1}{4}.
\end{align*}
Moreover, we set 
\begin{align*}
\|.\|_\X:= \sup_{(p,q,n)\in A} \|.\|_{L^p_\tau(\R_+)W^{n,q}(\B^d_1)}+\|.\|_{L^\infty_\tau(\R_+)H^{\frac{d-1}{2}}(\B^d_1)}
\end{align*}
 and let $\X$ be the completion of $C^\infty_{c_{rad}}(\R_+\times \overline{\B^d_1})$ with respect to $ \|.\|_\X$.

\begin{lem}\label{lem:nonlinear1}
The nonlinearity $\Nf$ satisfies the estimates
\begin{align*}
\int_0^\infty \|\N((\phi(\sigma),0))\|_{\mathcal{H}} d\sigma \lesssim \|\phi\|_\X^2+\|\phi\|_\X^5
\end{align*} 
and 
\begin{align*}
\int_0^\infty \|\N((\phi(\sigma),0))-\N((\phi(\sigma),0))\|_{\mathcal{H}} d\sigma \lesssim \left[\|\phi\|_\X+\|\phi\|_\X^{4}+\|\psi\|_\X+\|\psi\|_\X^{4}\right]\|\phi-\psi\|_{\X}
\end{align*}
for all $\phi \in C^\infty_c(\R_+\times \overline{\B^d_1})$.
\end{lem}
\begin{proof}
Let $u\in C^\infty(\overline{\B^d_1})$ and note that
$$
\|\Nf((u,0))\|_{\mathcal{H}}\lesssim \sum_{j=2}^5\|u^j\|_{H^{\frac{d-3}{2}}(\B^d_1)}.
$$
Now, clearly the most complicated term to estimate is given by $\|u^5\|_{H^{\frac{d-3}{2}}(\B^d_1)}$ and so we showcase the general procedure with this expression. If $d$ is odd, we have that $\frac{d-3}{2}\in \mathbb{N}$ and we compute that
\begin{align*}
\|u^5\|_{H^{\frac{d-3}{2}}(\B^d_1)}&\lesssim \sum_{|\alpha_1|+\dots +|\alpha_5|\leq\frac{d-3}{2}} \|\partial^{\alpha_1} u\cdot\, \dots\, \cdot \partial^{\alpha_5} u\|_{L^{2}(\B^d_1)}
\end{align*}
for multiindices $\alpha_1,\dots, \alpha_5\in \mathbb{N}^d$.
The hardest terms to bound are given in case
$$|\alpha_1|+\dots +|\alpha_5|=\frac{d-3}{2}$$, so we will focus on this particular case.
Assume now, without loss of generality, that $|\alpha_1|\geq |\alpha_2|\geq\dots \geq |\alpha_5|$ and observe that $L^\infty(\R_+)W^{|\alpha_1|,\frac{2d}{1+2|\alpha_1|}}(\B^d_1)$ and $L^{4}(\R_+)W^{\frac{4d}{1+4|\alpha_j|}}(\B^d_1)$  are admissible Strichartz spaces. Hence, given that
\begin{align*}
\frac{1}{2}=\frac{2+4|a_1|}{4d}+\sum_{j=2}^5\frac{1+4|a_j|}{4d},
\end{align*}
we can use Hölder's inequality to conclude that
\begin{align*}
\int_0^\infty\|\phi(\sigma,.)^5\|_{\dot H^{\frac{d-3}{2}}(\B^d_1)}d \sigma &\lesssim \sum_{|\alpha_1|+\dots +|\alpha_5|=\frac{d-3}{2}} \|\psi\|_{L^\infty(\R_+)W^{|\alpha_1|,\frac{2d}{1+2|\alpha_1|}}(\B^d_1)}
\\
&\quad \times \|\psi\|_{L^{4}(\R_+)W^{|\alpha_2|,\frac{4d}{1+4|\alpha_2|}}(\B^d_1)}
\\
&\times \dots \times \|\psi\|_{L^{4}(\R_+)W^{|\alpha_5|,\frac{4d}{1+4|\alpha_5|}}(\B^d_1)}
\\
&\lesssim \|\psi\|_\X^5.
\end{align*}
We now turn to the case of even $d$. Here, we make use of the Sobolev embedding
\begin{align*}
H^{\frac{d-3}{2}}(\B^d_1)\subset W^{\frac{d}{2}-1,\frac{2d}{d+1}}(\B^d_1)
\end{align*}
to infer 
\begin{align*}
\|u^5\|_{H^{\frac{d-3}{2}}(\B^d_1)} &\lesssim \sum_{|\alpha_1|+\dots +|\alpha_5|\leq\frac{d}{2}-1} \|\partial^{\alpha_1} u\cdot\, \dots\, \cdot \partial^{\alpha_5} u\|_{L^{\frac{2d}{d+1}}(\B^d_1)}.
\end{align*}
This time, the most difficult terms to estimate are of course given in case $|\alpha_1|+\dots +|\alpha_5|=\frac{d}{2}-1$ .
Again, we assume that $|\alpha_1|\geq |\alpha_2|\geq\dots \geq |\alpha_p|$ and note that
\begin{align*}
\frac{d+1}{2d}= \frac{1+2|\alpha_1|}{2d}+\sum_{j=2}^5 \frac{1+4|\alpha_j|}{4d}.
\end{align*}
So, 
we can use Hölder's inequality as above to derive the desired estimate
\begin{align*}
\int_0^\infty \|\N((\phi(\sigma),0))\|_{\mathcal{H}} d\sigma \lesssim \|\phi\|_\X^2+\|\phi\|_\X^5.
\end{align*} 
Likewise, one establishes the local Lipschitz estimates and we conclude this proof.
\end{proof}

To proceed, let $\uf\in \mathcal{H}$. Then, for $\phi \in C^\infty_{c,rad}(\R_+\times \overline{\B^d_1})$ we define the mapping
\begin{align*}
\K_\uf (\phi)(\tau)=[\Sf(\tau) \uf]_1
+\int_0^\tau [\Sf(\tau-\sigma) \Nf((\phi(\sigma),0))]_1 d\sigma -\Cf_\uf(\phi)(\tau)
\end{align*}
where
\begin{align*}
\Cf_\uf (\phi)(\tau)=e^\tau \left[\Pf\left(\uf
+\int_0^\tau e^{-\sigma}\Nf((\phi(\sigma),0))d\sigma \right)\right]_1.
\end{align*}
Furthermore, we denote by $\X_\delta$ the closed ball of radius $\delta$ around $0$ in the $\X$ topology.
We now also state our precise definition of a solution.
\begin{defi}\label{def:solutionstrichartz}
Let
  \[ \Gamma^T:=\{(t,r)\in [0,T)\times [0,\infty): r\leq T-t\}. \]
  We say that $u: \Gamma^T\to \R$ is a Strichartz solution of
\[ \left(\partial_t^2-\partial_r^2-\frac{d-1}{r}\partial_r\right)
  u(t,r)= u(t,r)^5\]

if $\phi:=\Phi_1=[\Psi-(\tfrac 34)^{\frac14}(1,\tfrac{1}{2})]_1$, with
\[ \Psi(\tau,\rho):=
  \begin{pmatrix}
    \psi(\tau,\rho) \\ \left(\frac{1}{2}+\partial_\tau+\rho\partial_\rho\right)\psi(\tau,\rho)
  \end{pmatrix},\qquad \psi(\tau,\rho):=(Te^{-\tau})^{\frac{1}{2}}u(T-Te^{-\tau}, Te^{-\tau}\rho),
\]
  belongs to $\mathcal{X}$ and satisfies
\begin{align*}
\phi=\K_{\Phi(0)}(\phi)
\end{align*}
and $\Cf(\phi,\Phi(0))=\textup{\textbf{0}}$.
\end{defi}
\begin{lem}
The estimates 
\begin{align*}
\|\K_\uf(\phi)\|_\X \lesssim \|u\|_{\mathcal{H}}+\|\phi\|_\X^2+\|\phi\|_\X^5
\end{align*}
and 
\begin{align*}
\|\K_\uf(\phi)-\K_\uf (\psi)\|_\X \lesssim \left(\|\phi\|_\X+\|\phi\|_\X^4+\|\psi\|_\X+\|\psi\|_\X^4\right)\|\phi-\psi\|_\X
\end{align*}
hold for all $\uf \in \mathcal{H}$ and all $\phi,\psi\in C^\infty_{c,rad}(\R_+\times \overline{\B^d_1})$.
\end{lem}
\begin{proof}
We split $\K_\uf$ into
\begin{align*}
(\I-\Pf)\K_\uf(\phi)(\tau):=[\Sf(\tau) (\I-\Pf)\uf]_1
+\int_0^\tau [\Sf(\tau-\sigma) (\I-\Pf)\Nf((\phi(\sigma),0))]_1 d\sigma
\end{align*}
and
\begin{align*}
\Pf\K_\uf(\phi)(\tau):=
\int_\tau^\infty [e^{\tau-\sigma}\Pf\Nf((\phi(\sigma),0))]_1 d\sigma.
\end{align*}
By employing Theorem \ref{thm:strichartz} and Lemma \ref{lem:nonlinear1}, we conclude that
\begin{align*}
\|(\I-\Pf)\K_\uf(\phi)\|_\X &\lesssim \|\uf \|_{\H}
+\int_0^\infty \|\Nf((\phi(\sigma),0))\|_{\H} d\sigma
\\
&\lesssim
\|\uf \|_{\H}+\|\phi\|_\X^2+\|\phi\|^5_\X.
\end{align*}
Note now, that the range of $\Pf$ is one dimensional, which implies the existence of a unique $\widetilde{\gf} \in \H$ such that
\begin{align*}
\Pf \ff=(\ff|\widetilde{\gf})_\H \gf
\end{align*}
for all $\ff\in \H$. Thus,
\begin{align*}
\|\Pf\K_\uf(\phi)(\tau)\|_{W^{n,q}(\B^d_1)} &\lesssim \int_\tau^\infty e^{\tau-\sigma} \|\Nf((\phi(\sigma),0))\|_{\H} d\sigma
\\
&=\int_\R e^{\tau-\sigma}1_{(-\infty,0)}(\tau-\sigma) \|\Nf((\phi(\sigma),0))\|_{\H} d\sigma
\end{align*}
for all $n\in \mathbb{N}$ and all $q\geq 1$. Therefore, Young's inequality implies
\begin{align*}
\|\Pf\K_\uf(\phi)\|_{L^p(\R_+)W^{n,q}(\B^d_1)}&\lesssim \|1_{(-\infty,0)}(\tau)e^\tau\|_{L^p_\tau(\R)} \int_0^\infty  \|\Nf((\phi(\sigma),0))\|_{\H} d\sigma 
\\
&\lesssim
\|\phi\|_\X^2+\|\phi\|^5_\X.
\end{align*}
In the same fashion, one obtains the desired Lipschitz estimate.
\end{proof}
An application of the contraction mapping principle yields the next Lemma.
\begin{lem}
There exists constants $\delta>0$ and $M>0$ such that the following holds. Let $\uf \in \mathcal{H}$ be such that $\|\uf\|_\H \leq \frac{\delta}{M}$. Then, the operator $\K_\uf$ extends to an operator from $\X_\delta $ to $ \X_\delta$ such that there exists a unique $\phi \in \X_\delta$ with
\begin{align*}
\K_\uf (\phi)=\phi.
\end{align*}
\end{lem}
\subsection{Variation of blowup time}
Our next task is to use our freedom in picking the blowup time $T$ to make correction term $\Cf_\uf$  vanish. To do this, we recall that, due to our transformations, the prescribed initial data are given by
\begin{align*}
\phi_1(0,\rho)&=\psi_1(0,\rho)-c_5=T^{\frac{1}{2}}f(T\rho)-c_5
\\
\phi_2(0,\rho)&=\psi_2(0,\rho)-\frac{1}{2}c_5=T^{\frac{3}{2}} g(T\rho)-\frac{1}{2}c_5
\end{align*}
with $c_5=\left(\frac{3}{4}\right)^{\frac{1}{4}}$.
Moreover, given that by assumption, our initial data lies close to that of $u^1$, we also recall that the initial data  of $u^1$ in similarity coordinates is of the form
\begin{align*}
\psi^1_1(0,\rho)&=T^{\frac{1}{2}}c_5,\qquad
\psi^1_1(0,\rho)=\frac{T^{\frac{3}{2}}}{2}c_5.
\end{align*}
Consequently, we can recast our initial data as
\begin{align*}
\begin{pmatrix}
 T^{\frac{1}{2}}f(T\rho)\\
 T^{\frac{3}{2}}g(T\rho)
\end{pmatrix}-c_5\begin{pmatrix}
T^{\frac12}
\\
\frac12 T^{\frac32}
\end{pmatrix}+c_5\begin{pmatrix}
T^{\frac12}
\\
\frac12 T^{\frac32}
\end{pmatrix}-
c_5\begin{pmatrix}
1
\\
\frac12 
\end{pmatrix}
\end{align*}
which naturally leads to defining the operator
\begin{align*}
\Uf:H^{\frac{d-1}{2}}\times H^{\frac{d-3}{2}}(\B^d_{1+\delta})\times [1-\delta,1+\delta]\to \H
\end{align*}
as 
\begin{align*}
\Uf(\vf,T)(\rho)=\begin{pmatrix}
 T^{\frac{1}{2}}v_1(T\rho)\\
 T^{\frac{3}{2}}v_2(T\rho)
\end{pmatrix}+
c_5\begin{pmatrix}
T^{\frac12}
\\
\frac12 T^{\frac32}
\end{pmatrix}- c_5\begin{pmatrix}
1
\\
\frac12 
\end{pmatrix}.
\end{align*}
Note that, $$\Phi(0)=U(f-c_5,g-\frac{1}{2}c_5).
$$
Moreover, $\Uf(.,T)$ is uniformly continuous for all  $T\in [\frac{1}{2},\frac{3}{2}]$ and satisfies
$U(\textup{\textbf{0}},1)=\textup{\textbf{0}}$.
\begin{lem}\label{lem:finallem}
There exist $\delta_0>0$ and $M>0$ such that for all $\delta\in (0,\delta_0)$ the following holds. Let $\vf\in H^{\frac{d-1}{2}}\times H^{\frac{d-3}{2}}(\B^d_{1+\delta})$ be such that
\begin{align*}
\|\vf\|_{H^{\frac{d-1}{2}}\times H^{\frac{d-3}{2}}(\B^d_{1+\delta})}\leq \frac{\delta}{M}.
\end{align*}
Then, there exists a unique time $T^*\in [1-\delta,1+\delta]$ and a unique function $\phi$ in $\X$ with
\begin{align*}
\K_{\Uf(\vf,T^*)}(\phi)=\phi, \qquad \Cf_{\Uf(\vf,T^*)}(\phi)=\textup{\textbf{0}}.
\end{align*}
\end{lem}
\begin{proof}
The existences of $T^*$ and $\phi$ follow as Lemma 5.6 in \cite{Wal22} and their respective uniqueness as in Lemma 5.7 of \cite{Wal22}.
\end{proof}
With this, we come to the proof of the second theorem.
\begin{proof}[Proof of Theorem \ref{thm: stability}]
Let $\vf= \uf -(c_5,\frac{1}{2}c_5)\in H^{\frac{d-1}{2}}\times H^{\frac{d-3}{2}}(\B^d_{1+\delta}) $ be small enough to satisfy the assumptions of Lemma \ref{lem:finallem}, $\phi$ be the associated fixed point in $\X$ with vanishing correction term, and $T$ be the associated blow up time.
Then,
\begin{align*}
\delta^2&\geq \|\phi\|_{L^2(\R_+)L^\infty(\B^d_1)}^2=\|\psi-c_5\|_{L^2(\R_+)L^\infty(\B^d_1)}^2
\\
&=\int_0^T (T-t)^{-1}\|\psi(-\log(T-t)+\log T,.)-c_5\|_{L^\infty(\B^d_1)}^2dt
\\
&=\int_0^T (T-t)^{-1}\|\psi(-\log(T-t)+\log T,\frac{.}{T-t})-c_5\|_{L^\infty(\B^d_{T-t})}^2 dt
\\
&=\int_0^T \|u(t,.)-u^T(t)\|_{L^\infty(\B^d_{T-t})}^2 dt.
\end{align*}
Finally, since $u^T$ is constant in space the remaining estimates stated in the theorem follow likewise.
\end{proof}

\appendix

\section{Interpolation Theory}
In this section, we recall some of the basic notions of interpolation theory and prove the used interpolation Lemma.
This exposition follows the book ``Interpolation Spaces'' by J. Bergh and J. Löfström
\cite{BerLof12} and largely uses the same notation. Thus, for a tuple $(X_0,X_1)$  of Banach spaces, we construct another Banach space $(X_0+X_1,\|.\|_{X_0+X_1})$ with
\begin{align*}
\|x\|_{X_0+X_1}:=\inf_{x=x_0+x_1,x_j\in X_j, j=1,2}  ( \|x_0\|_{X_0}+\|x_1\|_{X_1})
\end{align*}
for $x\in X_0+ X_1$. Now, consider the strip $S:=\{ z\in \C:0\leq z\leq 1\}$ and the set $F(X_0,X_1)$ consisting of all continuous functions $f:S\to X_0+X_1$ that are analytic on the interior of $S$ and additionally satisfy that the map
$t\mapsto f(j+it)$, for $j=0,1$ is a continuous function from $\R$ to $X_j$ which tends to $0$ as $|t| \to \infty.$
Then, $F(X_0,X_1)$ is a vector space and by equipping it with the norm
\begin{align*}
\|f\|_{F(X_0,X_1)}:=\max\left \{ \sup_{t\in\R} \| f(it)\|_{X_0},
  \sup_{t\in\R} \|f(1+it)\|_{X_1}\right \}
\end{align*}
it becomes a Banach space, see Lemma \cite[p.~88, Lemma 4.1.1.]{BerLof12}.
Further, for $\theta \in (0,1)$, the interpolation functor $C_\theta$ is defined in the following way. Let $(X_0,X_1)_{[\theta]}=C_\theta(X_0,X_1)$ be the set of all $ x\in X_0+X_1$ for which there exists an $f \in F(X_0,X_1)$ with $f(\theta)=x$. For any such $x$, we set
\begin{align*}
\|x\|_{(X_0,X_1)_{[\theta]}}:=\inf\{ \|f\|_{F(X_0,X_1)},f\in F(X_0,X_1): f(\theta)=x \}.
\end{align*}
Then, $((X_0,X_1)_{[\theta]},\|.\|_{(X_0,X_1)_{[\theta]}})$ is a Banach space and $C_\theta$ is an exact interpolation functor of order $\theta$ (see \cite[p.~88, Theorem 4.1.2.]{BerLof12}). 
Moreover, for any given Sobolev norm $\|.\|_{W^{s,q}(\B^d_1)}$, with $s\geq 0$ and $1\leq q<\infty$ as well as $a\in \R$, we let $L^p(\R_+, e^{a\tau}d\tau)W^{s,q}(\B^d_1)$ with $1\leq p\leq\infty$ be the completion of $C^\infty_c (\R_+\times\overline{\B^d_1})$ with respect to the norm
\begin{align*}
\|f\|_{L^p(\R_+,e^{a\tau}d\tau) W^{s,q}(\B^d_1)}^p:=\int_{\R_+}\|f(\tau,.)\|_{W^{s,q}(\B^d_1)}^p e^{a\tau} d\tau.
\end{align*}
Note that according to \cite[p.~317, Subsection 4.3.1.1, Theorem 1]{Tri95} one has that
\begin{align*}
(W^{s_0,q_0}(\B^d_1),W^{s_1,q_1}(\B^d_1))_{[\theta]}=W^{s_{\theta},q_{\theta}}(\B^d_1)
\end{align*}
for  $1\leq p,q_0,q_1< \infty$ and $0\leq s_0,s_1<\infty$ where $s_{\theta}=(1-\theta)s_0+\theta s_1$, $\tfrac{1}{q_{\theta}}=\tfrac{1-\theta}{q_0}+\tfrac{\theta}{q_1}$.
Having concluded these preliminaries, we come to the desired interpolation result.

\begin{prop}\label{prop:interpolation}
Let  $1\leq p_0,p_1<\infty$, $1\leq q_0,q_1< \infty$, $0\leq s_0,s_1<\infty$, and $\mu_0, \mu_1\in \R$ be such that
\begin{align*}
\theta \mu_1 +(1-\theta) \mu_0=0.
\end{align*}
  Then,
\begin{align*}
\left(L^{p_0}(\R_+,e^{\mu_0 p_0\tau}d\tau) W^{s_0,q_0}(\B^d_1) ,L^{p_1}(\R_+,e^{\mu_1 p_1\tau}d\tau)W^{s_1,q_1}(\B^d_1) \right)_{[\theta]}=L^{p_{\theta}}(\R_+)W^{s_{\theta},q_{\theta}}(\B^d_1)
\end{align*}
where $s_\theta, p_\theta,$ and $q_\theta$ are such that 
$$s_\theta= (1-\theta)s_0+ \theta s_1 ,\qquad \frac{1}{p_\theta}=\frac{1-\theta}{p_0} +\frac{\theta}{p_1},\qquad \frac{1}{q_\theta}=\frac{1-\theta}{q_0}+\frac{\theta}{q_1}.
$$
Likewise, for $1\leq r\leq \infty$ one has that 
\begin{align*}
\left(L^{p_0}(\R_+,e^{\mu_0 p_0\tau}d\tau) L^r(\B^d_1) ,L^{p_1}(\R_+,e^{\mu_1 p_1\tau}d\tau)L^{\infty}(\B^d_1) \right)_{[\theta]}=L^{p_{\theta}}(\R_+)L^{r_{\theta}}(\B^d_1)
\end{align*}
with $p_0,p_1,$ and $p_\theta$ as above and $ \frac{1}{r_\theta}=\frac{1-\theta}{r}$.
\end{prop}
\begin{proof}
This proposition essentially follows from the same considerations as \cite[p.~107, Theorem 5.1.2]{BerLof12}), which we illustrate here for the convenience of the reader. To simplify notation, we set $W_0=W^{s_0,q_0}(\B^d_1)$, $W_1=W^{s_1,q_1}(\B^d_1)$ and $p=p_{\theta}$. By construction, $C^\infty_c(\R_+\times \overline{\B^d_1})$ lies dense in $L^{p_0}(\R_+,e^{\mu_0 p_0\tau}d\tau)W_0\cap L^{p_1}(\R_+,e^{ \mu_1 p_1\tau}d\tau)W_1$. Thus, thanks to \cite[p.~91, Theorem 4.2.2]{BerLof12}, it is also a dense subset of
 $$\left(L^{p_0}(\R_+,e^{\mu_0 p_0\tau}d\tau) W_0,L^{p_1}(\R_+,e^{\mu_1 p_1\tau}d\tau)W_1\right)_{[\theta]}\quad \text{ and } \quad L^p(\R_+)(W_0,W_1)_{[\theta]}.$$ Consequently, it suffices to consider $C^\infty_c(\R_+\times \overline{\B^d_1}).$ We start with the inequality
\begin{align*}
\|u\|_{(L^{p_0}(\R_+,e^{\mu_0 p_0\tau}d\tau)W_0,L^{p_1}(\R_+,e^{\mu_1 p_1\tau}d\tau)W_1 )_{[\theta]}}\leq \|u\|_{L^p(\R_+)(W_0,W_1)_{[\theta]}}.
\end{align*}
Let $u\in C^\infty_c(\R_+\times \overline{\B^d_1})$ with $u \neq 0$. Then, for every
$\varepsilon>0$ and every $\tau\geq 0$, there exists an $f(\tau)\in F(W_0,W_1)$ with $f(\tau)(\theta)=u(\tau,.)$ and 
$$
\|f(\tau)\|_{F(W_0,W_1)}\leq (1+\varepsilon)\|u(\tau,.)\|_{(W_0,W_1)_{[\theta]}}.
$$
Set
$$
g(\tau)(z)=f(\tau)(z)e^{\varepsilon(z^2-\theta^2)}e^{(z\mu_1+(1-z)\mu_0)\tau}\left(\frac{\|u(\tau)\|_{(W_0,W_1)_{[\theta]}}}{\|u\|_{L^{p}(\R_+)(W_0,W_1)_{[\theta]}}}\right)^{p(\frac{1}{p_0}-\frac1{p_1})(\theta-z)}.
$$
Then, clearly $g(\tau)(\theta)=u(\tau,.)$. Further, one computes that
$$
p_0 p(\frac{1}{p_0}-\frac1{p_1})\theta=p-p_0.
$$ 
Consequently,
\begin{align*}
&\quad \|g(\tau)(it)\|_{L^{p_0}(\R_+,e^{-\mu_0 p_0\tau}d\tau)W_0}^{p_0} 
\\
&\leq e^{-p_0\varepsilon(t^2+\theta^2)}\int_{\R_+}\|f(\tau)(it)\|_{W_0}^{p_0}\left(\frac{\|u(\tau)\|_{(W_0,W_1)_{[\theta]}}}{\|u\|_{L^{p}(\R_+)(W_0,W_1)_{[\theta]}}}\right)^{p_0 p(\frac{1}{p_0}-\frac1{p_1})\theta} d\tau
\\
&\leq e^{-p_0\varepsilon(t^2+\theta^2)}(1+\varepsilon)^{p_0}\|u\|_{L^{p}(\R_+)(W_0,W_1)_{[\theta]}}^{-p_0 p(\frac{1}{p_0}-\frac1{p_1})\theta}\int_{\R_+}\|u(\tau,.)\|_{(W_0,W_1)_{[\theta]}}^{p} d\tau
\\
&=e^{-p_0\varepsilon(t^2+\theta^2)}(1+\varepsilon)^{p_0}\|u\|_{L^p(\R_+)(W_0,W_1)_{[\theta]}}^{p_0} 
\end{align*}
and similarly
\begin{align*}
\|g(\tau)(1+it)\|_{L^{p_1}(\R_+,e^{a p_1\tau}d\tau)W_1 }^{p_1} &\leq e^{p_1\varepsilon(1-t^2-\theta^2)}(1+\varepsilon)^{p_1}\|u\|_{L^p(\R_+)(W_0,W_1)_{[\theta]}}^{p_1}.
\end{align*}
Hence, as $\varepsilon>0$ was chosen arbitrarily, the claim follows.

For the other inequality, we invoke \cite[p.~93, Lemma 4.3.2]{BerLof12}, which states that any $f\in F(W_0,W_1)$ satisfies
\begin{align}\label{Eq:Poisson kernel}
\|f\|_{(W_0,W_1)_{[\theta]}}\leq \left(\frac{1}{1-\theta}\int_\R \|f(it)\|_{W_0} P_0(\theta,t) dt\right)^{1-\theta}\left(\frac{1}{\theta}\int_\R \|f(1+it)\|_{W_1} P_1(\theta,t) dt\right)^{\theta}
\end{align}
where 
$$
P_j(x+iy,t):=\frac{e^{-\pi(t-y)}\sin (\pi x)}{\sin(\pi x)^2+(\cos(\pi x)-e^{ij\pi-\pi (t-y)})^2}
$$
are the Poisson kernels of the strip $S$. 
Next, for $u\in C^\infty_c(\R_+\times \overline{\B^d_1})$ let $f(\tau)\in F(W_0,W_1)$ be such that $f(\tau)(\theta)=u(\tau,.)$. Then, \eqref{Eq:Poisson kernel}, Hölder's inequality, and the identity
$\frac 1p =\frac{1-\theta}{p_0}+\frac{\theta}{p_1}$ imply that
\begin{align*}
\|u\|_{L^p(\R_+)(W_0,W_1)_{[\theta]}}&\leq \bigg\| \int_{\R_+} \bigg[\left(\frac{1}{1-\theta}\int_\R \|f(\tau)(it)\|_{W_0} P_0\left(\theta,t\right) dt\right)^{1-\theta}
\\
&\quad \times \left(\frac{1}{\theta}\int_\R \|f(\tau)(1+it)\|_{W_1} P_1\left(\theta,t\right) dt\right)^{\theta}\bigg\|_{L^p_\tau(\R_+)}
\\
&\leq \theta^{-\theta}(1-\theta)^{\theta-1}\left\|e^{\mu_0\tau}\int_\R \|f(\tau)(it)\|_{W_0} P_0\left(\theta,t\right) dt \right\|_{L^{p_0}_\tau(\R_+)}^{1-\theta}
\\
&\quad\times \left\| e^{\mu_1\tau}\int_\R\|f(\tau)(1+it)\|_{W_1} P_1\left(\theta,t\right) dt\right\|_{L^{p_1}_\tau(\R_+)}^{\theta}.
\end{align*}
Moreover, an application of Minkowski's inequality shows that
\begin{align*}
\left\|\int_\R \|f(\tau)(it)\|_{W_0} P_0\left(\theta,t\right) dt e^{\mu_0\tau}\right\|_{L^{p_0}_\tau(\R_+)}&\leq \int_\R \left\|\|f(\tau)(it)\|_{W_0} e^{\mu_0\tau}\right\|_{L^{p_0}_\tau(\R_+)} P_0\left(\theta,t\right) dt
\\
&\leq\sup_{t\in \R}\|f(\tau)(it)\|_{L^{p_0}(\R_+,e^{\mu_0 p_0 \tau} d\tau) W_0}
\\
&\quad\times \int_\R P_0\left(\theta,t\right) dt
\end{align*}
and analogously one estimates the second factor.
Observe now, that a primitive (with respect to $t$) of 
$$
P_0(\theta,t)=
\sin (\pi \theta) \frac{e^{-\pi t}}{\sin(\pi \theta)^2+(\cos(\pi \theta)-e^{-\pi t})^2} 
$$
is given by
$$
\frac{1}{\pi}\arctan\left(\frac{\cos(\pi \theta)-e^{-\pi t}}{\sin (\pi \theta)}\right).
$$
Therefore,
\begin{align*}
\int_\R P_0 \left(\theta,t\right) dt = 1-\theta.
\end{align*}
Likewise, 
$$
-\frac{1}{\pi}\arctan\left(\frac{\cos(\pi \theta)+e^{-\pi t}}{\sin (\pi \theta)}\right)
$$
is a primitive of
$$
P_1(\theta,t)=
\sin (\pi \theta) \frac{e^{-\pi t}}{\sin(\pi \theta)^2+(\cos(\pi \theta)-e^{i\pi -\pi t})^2} 
$$
which implies that
\begin{align*}
\int_\R P_1 \left(\theta,t\right) dt =\theta.
\end{align*}
Thus,
\begin{align*}
\|u\|_{L^p(\R_+)(W_0,W_1)_{[\theta]}} &\leq \sup_{t\in \R}\|f(\tau)(it)\|_{L^{p_0}(\R_+,e^{\mu_0 p_0\tau} d\tau) W_0}^{1-\theta}\sup_{t\in \R}\|f(\tau)(1+it)\|_{L^{p_1}(\R_+,e^{\mu_1\tau} d\tau) W_1}^{\theta}
\\
&\leq \|f\|_{F(L^p(\R_+,e^{\mu_0 p_0\tau}d\tau)W_0, L^p(\R_+,e^{\mu_1\tau}d\tau)W_1)}.
\end{align*}
\end{proof}
\section{Upgraded regularity of the Resolvent in the integer case}
In this section we prove Lemma \ref{lem:extendres}. We start with a definition of the Sobolev-Slobodeckij spaces.
\begin{defi}
Let $s\in (0,\infty)\setminus \mathbb{N}$. Then, for any open set $U\subset \R^d$ and any $f\in C^\infty(\overline{U})$ we define the Sobolev-Slobodeckij seminorm $[.]_{H^{s}(U)}$ as 
\begin{align*}
[f]_{H^{s}(U)}^2=\sum_{a_1+\dots+a_d=\floor s} \int_{U}\int_{U} \frac{|\partial_{x_1}^{a_1}\dots \partial_{x_d}^{a_d}f(x)-\partial_{y_1}^{a_1}\dots \partial_{y_d}^{a_d}f(y)|^2}{|x-y|^{d+2(s-\floor s)}}dx dy.
\end{align*}
Furthermore, we define the space $H^{s}_{rad}(\B^d_1)$ as the completion of $C^\infty_{rad}(\B^d_1)$ with respect to the norm
\begin{align*}
\|.\|_{H^s(\B^d_1)}^2:= \|.\|_{H^{\floor s}(\B^d_1)}^2+[.]_{H^{s}(\B^d_1)}^2.
\end{align*}

\end{defi}
For us, the following bound will be important.
\begin{lem}
The estimate \begin{align*}\label{lem:apB1}
\|f\|_{H^s(\B^d_1)}\lesssim \|f\|_{H^s(\B^d_{\frac{1}{2}})}^2+\|f\|_{H^s(\frac{1}{2},1))}^2
\end{align*}
holds for all  $f \in C^\infty_{rad}(\B^d_1)$. Consequently, it also holds for all $f\in H^s_{rad}(\B^d_1)$.
\end{lem}
\begin{proof}
The estimate 
\begin{align*}
\|f\|_{H^{\floor s }(\B^d_1)}^2\lesssim \|f\|_{H^{\floor{s}}(\B^d_{\frac{1}{2}})}^2+\|f\|_{H^{\floor s}(\frac{1}{2},1))}^2
\end{align*}
is immediate. Thus, given that 
\begin{align*}
[f]_{H^s(\B^d_1)}^2\lesssim [f]_{H^s(\B^d_{\frac{1}{2}})}^2+[f]_{H^s(\B^d_1\setminus \B^d_{\frac{1}{2}})}^2,
\end{align*}
the only thing that remains to be shown is the estimate
\begin{align*}
[f]_{H^s(\B^d_1\setminus \B^d_{\frac{1}{2}})}^2\lesssim \|f\|_{H^s(\frac{1}{2},1))}^2.
\end{align*}
To see this, one first notes that 
\begin{align*}
\|.\|_{H^{\ceil s}_{rad}(\B^d_1\setminus \B^d_{\frac{1}{2}})}\simeq \|\|_{H^{\ceil s}(\frac{1}{2},1))}\text{ and } \|\|_{H^{\floor s}_{rad}(\B^d_1\setminus \B^d_{\frac{1}{2}})}\simeq \|\|_{H^{\floor s}(\frac{1}{2},1))}.
\end{align*}
Thus, as \begin{align*}
H^{s}(\tfrac{1}{2},1))=[H^{\ceil s}(\tfrac{1}{2},1))H^{\floor s}(\tfrac{1}{2},1))]_\theta \text{ and } H^{s}(\B^d_1\setminus \B^d_{\frac{1}{2}})=[H^{\ceil s}(\B^d_1\setminus \B^d_{\frac{1}{2}}),H^{\floor s}(\B^d_1\setminus \B^d_{\frac{1}{2}})]_\theta
\end{align*}
for some appropriately chosen $\theta\in (0,1)$, the claim follows.
\end{proof}
\begin{lem}\label{lem:fraclem}
Let $p\in (0,1)$ and $g\in C^1([0,1])$. Then $f(x):=(1-x)^pg(x)$ is an element of $H^s((0,1))$ provided $s<p$.
\end{lem}
\begin{proof}
We compute that
\begin{align*}
[f]_{W^{s,q}(\B^d_1)}^q& =\int_0^1\int_0^1 \frac{|(1-x)^pg(x)-(1-y)^pg(y)|^q}{|x-y|^{1+sq}}dx dy
\\
&\lesssim \int_0^1\int_0^1\frac{|(1-x)^p-(1-y)^p|^q}{|x-y|^{1+sq}}dx dy+
\int_0^1\int_0^1\frac{|g(x)-g(y)|^q}{|x-y|^{1+sq}}dx dy
\\
&\lesssim \int_0^1\int_0^1\frac{1}{|x-y|^{1+sq-pq}}dx dy
\end{align*}
where the last step follows from the Hölder continuity of the function $x^p$ and the regularity of $g$. Hence, the claim follows, as by assumption $1+(s-p)q<1$.
\end{proof}
Now, we can at last come to the proof of Lemma \ref{lem:extendres}.

\begin{proof}[Proof of Lemma \ref{lem:extendres}]
We start by noting that, if $\Rm (f)(.,\lambda) \in H^{s+\frac{1}{100}}(\B^d_1)$, then, it is the unique such solution, as otherwise $\lambda$ would be an eigenvalue of $\Lf$.
Hence, only the regularity of $\Rm (f)(.,\lambda)$ needs to be studied.
For this, one readily computes that $\Rm (f)(.,\lambda) \in C^{s+1}((0,1))$, as all the involved functions are smooth on that open interval. Near $0$, one again uses the explicit forms of $u_0$ and $u_1$ to readily conclude that
$ \Rm (f)(.,\lambda)\in H^{s+1}(\B^d_1)
$
for all $\lambda\in \widehat{S}$.
Near 1, we once more rewrite
\begin{align*}
-u_0(\rho,\lambda)U_{1,j}(\rho,\lambda) 
+u_1(\rho,\lambda)U_{0,j}(\rho,\lambda) = 
-u_2(\rho,\lambda)U_{1,j}(\rho,\lambda) 
+u_1(\rho,\lambda)U_{2,j}(\rho,\lambda)
\end{align*}
Then, by our choice of $\kappa_1(f)$ we infer that we can rewrite  
$ 
 u_2(\rho,\lambda)[U_{1,1}(\rho,\lambda) f(\rho)-\kappa_1(f)(\lambda)]
$
as 
\begin{align*}
&\quad u_2(\rho,\lambda)[U_{1,1}(\rho,\lambda) f(\rho)-\kappa_1(f)(\lambda)]
\\
&=u_2(\rho,\lambda) \frac{\int_\rho^1 f'(t) dt}{\sqrt{2s-2\lambda-1}}\int_0^1 \frac{(1-t)^{\ceil s -s- \frac12+\lambda}}{\prod_{j=1}^{\floor s}(\lambda+\frac{1}{2}+j-s)} 
\partial_t^{\ceil s-1}\left(\frac{t^{d-1}u_1(t,\lambda)}{(1+t)^{s-\lambda-\frac12}}\right) dt 
\\
&\quad+
\frac{u_2(\rho,\lambda)f(\rho)}{\sqrt{2s-2\lambda-1}}\int_\rho^1 \frac{(1-t)^{\ceil s-s-\frac12+\lambda}}{\prod_{j=1}^{\ceil s-1}(\lambda+\frac{1}{2}+j-s)} 
\partial_t^{\ceil s-1}\left(\frac{t^{d-1}u_1(t,\lambda)}{(1+t)^{s-\lambda-\frac12}}\right) dt 
\\
&\quad +
\frac{u_2(\rho,\lambda)\int_\rho^1 f'(t) dt}{\sqrt{2s-2\lambda-1}}\sum_{j=1}^{\ceil s-2}\lim_{\rho \to 0}\partial_\rho^j\left(\frac{\rho^{d-1}u_1(\rho,\lambda)}{(1+\rho)^{s-\lambda-\frac12}}\right) \prod_{\ell=1}^{j+1}\frac{1}{\lambda+\frac{1}{2}+\ell-s}
\\
&=:I_1(\rho,\lambda)+I_2(\rho,\lambda)+I_3(\rho,\lambda).
\end{align*}
By scaling one now infers that $I_2(\rho,\lambda)$ is smooth at $\rho=1$, while 
\begin{align*}
I_1(\rho,\lambda)+I_3(\rho,\lambda)=(1-\rho)^{s-\lambda+\frac12} g(\rho,\lambda)
\end{align*}
where $g$ is smooth at $\rho=1$. Moreover, 
$ \partial_\rho^s (1-\rho)^{s-\lambda+\frac12} g(\rho,\lambda)= (1-\rho)^{-\lambda+\frac12} h(\rho,\lambda)$
for some smooth $h$. Consequently, by Lemma \ref{lem:fraclem}, one has that 
$$
 \|u_2(\rho,\lambda)[U_{1,1}(\rho,\lambda) f(\rho)-\kappa_1(f)(\lambda)]\|_{H^{s+\frac{1}{100}}(\frac{1}{2},1))}<\infty
$$
for all stated values of $\lambda$.
By bounding the remaining terms in the same fashion, one establishes that 
$$
\Rm(f)(.,\lambda)\in H^{s+\frac{1}{100}}(\frac{1}{2},1)).
$$
Thus, the claim follows from Lemma \ref{lem:fraclem}.
\end{proof}

\bibliography{references}
\bibliographystyle{plain}
\end{document}